\documentclass[a4paper, 10pt]{amsart}
\usepackage{amscd}
\usepackage{amssymb}
\usepackage[all]{xy}
\usepackage[T1]{fontenc}
\usepackage{lmodern}
\date{5 September 2014}
\title[Twisted Deformation Quantization]{Twisted Deformation
Quantization of Algebraic Varieties}
\dedicatory{Dedicated to Rina Yekutieli, my Mother, on her 80th Birthday}

\usepackage{graphicx}
\usepackage{hyperref}
\hypersetup{colorlinks=false}
\author{Amnon Yekutieli}
\address{A. Yekutieli: Department of  Mathematics
Ben Gurion University,
Be'er Sheva 84105,
Israel}
\email{amyekut@math.bgu.ac.il}
\thanks{
Primary: 53D55; Secondary: 14B10, 16S80, 17B40, 18D05.}
\thanks{This research was supported by the US-Israel Binational
Science Foundation and by the Israel Science Foundation.}

\newtheorem{thm}[equation]{Theorem}
\newtheorem{cor}[equation]{Corollary}
\newtheorem{prop}[equation]{Proposition}
\newtheorem{lem}[equation]{Lemma}
\theoremstyle{definition}
\newtheorem{dfn}[equation]{Definition}
\newtheorem{rem}[equation]{Remark}
\newtheorem{exa}[equation]{Example}

\newtheorem{que}[equation]{Question}

\newtheorem{setup}[equation]{Setup}

\numberwithin{equation}{section}
\newcommand{\iso}{\xrightarrow{\simeq}}

\newcommand{\xar}{\xrightarrow}
\newcommand{\opn}{\operatorname}
\newcommand{\cat}[1]{\operatorname{\mathsf{#1}}}
\newcommand{\bdot}{\bsym{\cdot}}

\newcommand{\ol}{\overline}

\newcommand{\rmitem}[1]{\item[\text{\textup{(#1)}}]}
\newcommand{\mfrak}[1]{\mathfrak{#1}}
\newcommand{\mcal}[1]{\mathcal{#1}}

\newcommand{\mbf}[1]{\mathbf{#1}}
\newcommand{\mrm}[1]{\mathrm{#1}}
\newcommand{\mbb}[1]{\mathbb{#1}}

\newcommand{\smfrac}[2]{\textstyle \frac{#1}{#2}}
\newcommand{\tup}[1]{\textup{#1}}
\newcommand{\bsym}[1]{\boldsymbol{#1}}
\newcommand{\boplus}{\bigoplus\nolimits}

\newcommand{\boprod}{\prod\nolimits}
\newcommand{\ot}{\otimes}
\newcommand{\hatotimes}[1]{\, \what{{\otimes}}_{#1} \,}
\newcommand{\hot}{\hatotimes{}}
\newcommand{\til}[1]{\tilde{#1}}

\newcommand{\what}[1]{\widehat{#1}}

\newcommand{\K}{\mbb{K}}
\newcommand{\R}{\mbb{R}}
\newcommand{\N}{\mbb{N}}
\newcommand{\Z}{\mbb{Z}}
\newcommand{\OO}{\mcal{O}}
\renewcommand{\AA}{\mcal{A}}

\newcommand{\Ga}{\Gamma}

\newcommand{\g}{\mfrak{g}}
\newcommand{\m}{\mfrak{m}}
\newcommand{\h}{\mfrak{h}}
\newcommand{\n}{\mfrak{n}}
\renewcommand{\a}{\mfrak{a}}
\newcommand{\om}{\omega}
\newcommand{\al}{\alpha}
\newcommand{\be}{\beta}
\newcommand{\ga}{\gamma}
\newcommand{\de}{\delta}
\newcommand{\si}{\sigma}
\newcommand{\bwedge}{{\textstyle \bigwedge}}

\renewcommand{\d}{\mrm{d}}
\newcommand{\OX}{\mcal{O}_X}

\newcommand{\twoto}{\Rightarrow}
\newcommand{\twoiso}{\stackrel{\simeq\ }{\Longrightarrow}}

\newcommand{\gerbe}[1]{\bsym{\mcal{#1}}}
\newcommand{\twocat}[1]{\operatorname{\bsym{\cat{#1}}}}

\newcommand{\lb}{\linebreak}
\newcommand{\crvar}{\curvearrowright}

\begin{document}

\begin{abstract}
Let $X$ be a smooth algebraic variety over a field $\K$ containing the real
numbers. We introduce the notion of twisted associative (resp.\
Poisson) deformation of the structure sheaf $\mcal{O}_X$. These are
stack-like versions of usual deformations. We prove that there is a
{\em twisted quantization} operation from twisted Poisson deformations
to twisted associative deformations, which is canonical and bijective
on gauge equivalence classes. This result extends work of Kontsevich, and our
own earlier work, on deformation quantization of algebraic varieties.
\end{abstract}

\maketitle

\tableofcontents

\setcounter{section}{-1}
\section{Introduction}
\label{sec:Int}
\numberwithin{equation}{subsection}

\subsection{Deformation Quantization}
Let $\K$ be a field of characteristic $0$, and let $X$ be a
smooth algebraic variety over $\K$, with structure sheaf
$\mcal{O}_X$. Suppose $R$ is a {\em parameter algebra} over $\K$;
namely $R$ is a complete local noetherian commutative 
$\K$-algebra, with maximal ideal $\m$ and residue field $R / \m = \K$. 
The main example is $R = \K[[\hbar]]$, the formal power series ring in the
variable $\hbar$. An {\em associative $R$-deformation of $\mcal{O}_X$} is a
sheaf $\mcal{A}$ of flat $\m$-adically complete associative $R$-algebras on $X$,
with an isomorphism 
$\psi : \K \otimes_R \mcal{A} \iso \mcal{O}_X$, called an
augmentation. Similarly, a {\em Poisson $R$-deformation of $\mcal{O}_X$} is a
sheaf $\mcal{A}$ of flat $\m$-adically complete commutative Poisson $R$-algebras
on $X$, with an augmentation to $\mcal{O}_X$. 

Such deformations could be sheaf-theoretically trivial, meaning that
$\mcal{A} \cong$ \lb $R \hot \mcal{O}_X$, endowed
with either an associative multiplication (called a star product), or
a Poisson bracket. This is what happens in the differentiable setup
(i.e.\ when $X$ is a $\mrm{C}^{\infty}$ manifold and $\K = \R$). But in the
algebro-geometric setup the sheaf $\mcal{A}$ could be very complicated
-- indeed, all classical commutative deformations of $\mcal{O}_X$ are
special cases of both associative and Poisson deformations. 

Suppose $\mcal{A}$ and $\mcal{A}'$ are two $R$-deformations (of the same kind,
namely both are Poisson or associative). A {\em gauge transformation} 
$g : \mcal{A} \to \mcal{A}'$ is an isomorphism of sheaves of $R$-algebras
(Poisson or associative) that commutes with the augmentations to $\mcal{O}_X$.
If such a gauge transformation exists, then $\mcal{A}$ and $\mcal{A}'$ are said
to be {\em gauge equivalent}.

An $R$-deformation $\mcal{A}$ of $\mcal{O}_X$ determines a {\em first order
bracket} $\{ -,- \}_{\mcal{A}}$ on $\mcal{O}_X$. 
When $R = \K[[\hbar]]$ this is the usual induced Poisson bracket,
cf.\ \cite{Ye1}. For general $R$ see Definition \ref{dfn:20}.
It is important to note that the  first order bracket $\{ -,- \}_{\mcal{A}}$
is gauge invariant. 

\begin{exa} \label{exa:23}
Suppose $\{ -,- \}$ is a Poisson bracket on $\mcal{O}_X$. Consider the
commutative $\K[[\hbar]]$-algebra 
$\mcal{A} := \mcal{O}_X[[\hbar]] = \K[[\hbar]] \hot \OO_X$, endowed
with the Poisson bracket $\hbar \{ -,- \}$ and the obvious augmentation to 
$\mcal{O}_X$. We get a Poisson $\K[[\hbar]]$-deformation of $\mcal{O}_X$, and
the first order bracket is 
$\{ -,- \}_{\mcal{A}} = \{ -,- \}$.
\end{exa}

Let $\mcal{A}$ be a Poisson $R$-deformation of $\mcal{O}_X$.
A {\em deformation quantization} of $\mcal{A}$ is an associative
$R$-deformation $\mcal{B} = \opn{quant}(\mcal{A})$, such that 
$\{ -,- \}_{\mcal{B}}  = \{ -,- \}_{\mcal{A}}$.

In the fundamental paper \cite{Ko1}, M. Kontsevich proved that in the
$\mrm{C}^{\infty}$ setup, and for $R = \R[[\hbar]]$, every Poisson
$\R[[\hbar]]$-deformation $\mcal{A}$ can be quantized.
Moreover, Kontsevich proved that his  quantization  is canonical, 
and it induces a bijection between the set of gauge
equivalence classes of Poisson deformations, and the set of gauge
equivalence classes of associative deformations.
 
We proved an analogous result for {\em affine smooth
algebraic varieties} over a field $\K$ containing $\R$.
The method used in our paper \cite{Ye1}
was an adaptation of the ideas of Kontsevich from \cite{Ko1} to the
algebro-geometric setup. There were cohomological obstructions, indicating that 
the same result (namely \cite[Theorem 0.1]{Ye1}) would not hold for more
complicated smooth algebraic varieties. 

\subsection{Twisted Quantization: Overview}
In the paper \cite{Ko2}, Kontsevich suggested that on {\em any smooth algebraic
variety} $X$, every Poisson $R$-deformation $\mcal{A}$ of $\mcal{O}_X$ can be
quantized to a {\em stack of $R$-algebroids} $\gerbe{B}$ on $X$. 

The purpose of the present paper is more than just to provide a proof
of that assertion of Konstevich. 
We introduce the notions of {\em twisted Poisson $R$-deformation} and 
{\em twisted associative $R$-deformation} of $\mcal{O}_X$. 
These are stacky versions of the $R$-deformations discussed above. A 
twisted associative $R$-deformation is a refined variant of the {\em stack
of $R$-linear algebroids} from \cite{Ko2}, and is very similar to
the notion of {\em DQ algebroid} from \cite{KS2}. 
The concept of twisted Poisson deformation appears to be totally new (in
algebraic geometry, as well as in real differential and complex analytic
geometry). 

Our main result (the Twisted Quantization Theorem, number \ref{thm:14} below)
establishes a {\em twisted quantization} operation,
which is a canonical bijection between the set of twisted gauge equivalence
classes of twisted Poisson $R$-deformations of $\mcal{O}_X$, and the set of
twisted gauge equivalence classes of twisted associative $R$-deformations.
The proof of Theorem \ref{thm:14} requires many
intermediate results (in ring theory, algebraic geometry, cosimplicial
groupoids and $\mrm{L}_{\infty}$ morphisms), some of them of
independent interest. We present a few of these intermediate results below in
the Introduction. 
The Twisted Quantization Theorem raises a few questions that we find
intriguing; see end of Section \ref{sec:TDQ}.

We should mention that the first occurrence of twisted quantization was in
the paper \cite{Ka} of M. Kashiwara from 1996, that concerned contact 
manifolds. We discuss this work, and several other papers related to twisted
deformations, in Subsection \ref{subsec:discuss}.

\subsection{Twisted Deformations} \label{subsec:tw.defs}
Let us now explain what are twisted deformations. 

We shall start with the concept of {\em crossed groupoid} (also known as a
crossed module over a groupoid, or a strict $2$-groupoid, cf.\ \cite{Ye7, Bw}).
This is a structure 
$\cat{P} = ( \cat{P}_1, \cat{P}_2, 
\opn{Ad}_{\cat{P}_1 \crvar \cat{P}_{2}}, \opn{D} )$
consisting of groupoids $\cat{P}_1$ and $\cat{P}_{2}$, such that 
$\opn{Ob}(\cat{P}_1) = \opn{Ob}(\cat{P}_2)$ and $\cat{P}_2$ is totally
disconnected; an action $\opn{Ad}_{\cat{P}_1 \crvar \cat{P}_{2}}$ of 
$\cat{P}_1$ on $\cat{P}_{2}$ called the {\em twisting}; and a morphism of
groupoids  (i.e.\  a functor) $\opn{D} : \cat{P}_{2} \to \cat{P}_{1}$ called the
{\em feedback}. There are conditions relating the twisting and the feedback; see
Definition \ref{dfn:cosim.101} for more details. 
We refer to the morphisms in the groupoid $\cat{P}_1$ as {\em $1$-morphisms}
of $\cat{P}$, or as {\em gauge transformations}. The morphisms in
$\cat{P}_2$ are called {\em $2$-morphisms}, or {\em inner gauge
transformations}. There is an obvious notion of morphism between crossed
groupoids.

Now let $X$ be a topological space. A {\em stack of crossed groupoids}
$\twocat{P}$ on $X$ is the geometrization of a crossed groupoid.
Thus for every open set $U \subset X$ we have a crossed groupoid
$\twocat{P}(U)$; and for every inclusion $V \subset U$ of open sets there is a
morphism of crossed groupoids
$\mrm{rest}^{}_{V / U} : \twocat{P}(U) \to \twocat{P}(V)$. 
We  assume that the restriction morphisms satisfy 
$\mrm{rest}^{}_{W / V} \circ \mrm{rest}^{}_{V / U} = \mrm{rest}^{}_{W / U}$
for a double inclusion $W \subset V \subset U$ of open sets of $X$.
We also assume that $\twocat{P}_1$ is a stack of groupoids (in the usual
sense), and for every local object 
$\mcal{A}$ of $\twocat{P}$ the presheaf of groups 
$\twocat{P}_2(\mcal{A})$ is a sheaf. 
See Definition \ref{dfn:200} for details.

For us there are two important stacks of crossed groupoids. 
Let $\K$ be a field of characteristic $0$,
$X$ an algebraic variety over $\K$, and $(R, \m)$ some parameter
algebra over $\K$  (these are the default assumptions for the rest of the
Introduction). First there is the stack
$\twocat{P} := \twocat{AssDef}(R, \mcal{O}_X)$ of associative $R$-deformations
of $\mcal{O}_X$, which assigns to each open set $U \subset X$ the set
$\twocat{P}(U) := \cat{AssDef}(R, \mcal{O}_U)$ of associative 
$R$-deformations of $\mcal{O}_U$. The $1$-morphisms in $\twocat{P}(U)$ are the
gauge transformations. The inner gauge group 
$\twocat{P}_2(U)(\mcal{A})$ of an 
object $\mcal{A} \in \opn{Ob}(\twocat{P}(U))$ is the
multiplicative group of elements of $\Gamma(U, \mcal{A})$ that are congruent to
$1$ modulo $\m$. The feedback $\opn{D}(a)$, for an element 
$a \in \twocat{P}_2(U)(\mcal{A})$, 
is the conjugation action of $a$ on $\mcal{A}|_U$.

The second stack of crossed groupoids is the stack
$\twocat{P} := \twocat{PoisDef}(R, \mcal{O}_X)$ of Poisson
$R$-deformations of $\OX$. It assigns to each open set $U \subset X$
the set $\twocat{P}(U) := \cat{PoisDef}(R, \mcal{O}_U)$ of Poisson
$R$-deformations of $\mcal{O}_U$. Again the $1$-morphisms are the gauge
transformations. The inner gauge group $\twocat{P}_2(U)(\mcal{A})$ of an
object $\mcal{A} \in \opn{Ob}(\twocat{P}(U))$
is $\Gamma(U, \opn{Exp}(\m \mcal{A}))$, 
where $\m \mcal{A}$ is viewed as a sheaf of pronilpotent Lie
algebras with the Poisson bracket. The action of $\twocat{P}_2(U)(\mcal{A})$  on
$\mcal{A}$ is by {\em formal hamiltonian flows}. 

Given a stack of crossed groupoids $\twocat{P}$ on $X$, we can
talk about {\em twisted objects of $\twocat{P}$}. Such a twisted object 
is a triple $(\gerbe{G}, \gerbe{A}, \opn{cp})$, consisting
of a gerbe $\gerbe{G}$ on $X$, called the {\em gauge gerbe}, with a morphism of
stacks of groupoids $\gerbe{A} : \gerbe{G} \to \twocat{P}_1$,
called the {\em representation}. 
This means that to every local object $i$ of 
$\gerbe{G}$ we assign a local object $\gerbe{A}(i)$ of $\twocat{P}$; and to
every local isomorphism $g : i \to j$ in $\gerbe{G}$ we assign a local 
$1$-isomorphism (i.e.\ a gauge transformation) 
$\gerbe{A}(g) : \gerbe{A}(i) \to \gerbe{A}(j)$ in $\twocat{P}_1$.
There is also an isomorphism of sheaves of groups
$\opn{cp} : \gerbe{G}(i) \to \twocat{P}_2(\gerbe{A}(i))$
for every local object $i$ of $\gerbe{G}$, called the {\em coupling
isomorphism}. All these data have to satisfy certain compatibilities -- 
see Definitions \ref{dfn:6} and \ref{dfn:11} for full details.
By slight abuse of language, we usually refer to the
twisted object $\gerbe{A}$ and its gauge gerbe $\gerbe{G}$.
An object $\gerbe{A}(i)$, for some local object $i$ of $\gerbe{G}$, is called a
{\em local object} of $\gerbe{A}$. 

In Examples \ref{exa:21} and \ref{exa:tw.sh.101} we explain how gerbes and
stacks of $R$-algebroids can be viewed as twisted objects, of the stacks of
crossed groupoids $\twocat{Grp}^{\times}(X)$ and  
$\twocat{Assoc}^{\times}(R, X)$ respectively. 

By definition a {\em twisted associative $R$-deformation of $\mcal{O}_X$} is a 
twisted object $\gerbe{A}$ of the stack of crossed groupoids
$\twocat{AssDef}(R, \mcal{O}_X)$.
Likewise, a {\em twisted Poisson $R$-deformation of
$\mcal{O}_X$} is a twisted object $\gerbe{A}$ of the stack  of crossed groupoids
$\twocat{PoisDef}(R, \mcal{O}_X)$.

The distinction between twisted associative $R$-deformations and DQ-algebroids
(in the sense of \cite{KS2}) is explained in Remark \ref{rem:6}.
As we said before, the concept of twisted Poisson deformation appears to be
completely new -- there does not seem to be any similar definition in the
literature.

\subsection{Gauge Gerbes}
There is a benefit in working with gauge gerbes: it is easier
to study the combinatorial structure of a twisted deformation.
Indeed, let $\gerbe{A}$ be a twisted $R$-deformation of $\mcal{O}_X$, with 
gauge gerbe $\gerbe{G}$. The coupling isomorphism endows $\gerbe{G}$
with the $\m$-adic filtration, making it into a pronilpotent
gerbe (in the sense of \cite{Ye4}). Using the results from 
\cite{Ye4} on pronilpotent gerbes we obtain the next theorem
(repeated in a more general form as Theorem \ref{thm:6} in the body of
the paper): 

\begin{thm}[Total Trivialization] \label{thm:17}
Let $\K$ be a field of characteristic $0$, $X$ an algebraic variety over
$\K$, $\gerbe{A}$ a twisted associative or Poisson $R$-deformation of 
$\OX$, with gauge gerbe $\gerbe{G}$, and $U$ an open set in $X$
that satisfies 
$\opn{H}^1(U, \OX) = \opn{H}^2(U, \OX) = 0$
\tup{(}for instance an affine open set\tup{)}. 
Then the groupoid $\gerbe{G}(U)$ is nonempty and connected.
\end{thm}

\subsection{Combinatorial Descent} \label{subsec:combinat}
In this subsection we leave the geometry aside, to talk about {\em combinatorial
descent data}. We shall return to the geometry in Subsection
\ref{subsec:geodesc}. 

We are interested in {\em cosimplicial crossed groupoids}. 
A cosimplicial crossed group\-oid 
$G = \{ G^p \}_{p \in \N}$
is made up of a crossed groupoid 
$G^p = ( G^p_1, G^p_2, \opn{Ad}, \opn{D})$ 
for every $p$, and a morphism of crossed groupoids
$G(\al) : G^p \to G^q$
for every arrow $\al : p \to q$ in the simplex category $\bsym{\Delta}$.
The morphisms $G(\al)$ must satisfy the simplicial relations. 

Let $G = \{ G^p \}_{p \in \N}$
be a cosimplicial crossed groupoid. A {\em combinatorial descent datum} in
$G$ is a triple $(\om, g, a)$ consisting of elements of these sorts: 
$\om \in \opn{Ob}(G^0$),
$g \in G^1_1(\om_{(0)}, \om_{(1)})$
and 
$a \in G^2_2(\om_{(0)}, \om_{(0)})$.
These elements have to satisfy two conditions: ``failure of $1$-cocycle'' and 
``twisted $2$-cocycle''; see Definition \ref{dfn:101}. 
Here $\om_{(i)}$ is the object of the crossed groupoid $G^p$
corresponding to the vertex $(i)$ of the simplex $\bsym{\Delta}^p$, and
$G^p_p(\om_{(i)}, \om_{(j)})$ is the set of morphisms
$\om_{(i)} \to \om_{(j)}$ in the groupoid $G^p_p$. 

As we shall see in Subsection \ref{subsec:geodesc} below, combinatorial descent
data is an abstraction of the familiar descent data for gerbes.

Let us denote by $\opn{Desc}(G)$ the set of combinatorial descent data.
There is an equivalence relation on this set, called gauge equivalence (see
Definition \ref{dfn:100}), and we denote by 
$\ol{\opn{Desc}}(G)$ the quotient set. A morphism 
$\Phi : G \to H$ of cosimplicial crossed groupoids induces a
function 
$\ol{\opn{Desc}}(\Phi) : \ol{\opn{Desc}}(G) \to
\ol{\opn{Desc}}(H)$.
The notion of {\em weak equivalence} between cosimplicial crossed groupoids is
introduced in Definition \ref{dfn:equiv-cosim.103}. 
The next result is proved in the companion paper \cite{Ye8}.
See \cite{Pr} for an alternative proof. 

\begin{thm}[{Combinatorial Equivalence,  \cite[Theorem 0.1]{Ye8}}]
\label{thm:intro.100}
Let $\Phi : G \to H$ be a weak equivalence between 
cosimplicial crossed groupoids. Then the function 
\[ \ol{\opn{Desc}}(\Phi) : \ol{\opn{Desc}}(G) \to
\ol{\opn{Desc}}(H) \]
is bijective.
\end{thm}

\subsection{Lie Descent}
Let $\g = \bigoplus_{i \in \Z} \, \g^i$ be a {\em quantum type} DG Lie algebra
over $\K$,
namely $\g^i = 0$ for all $i < -1$. There is an induced
pronilpotent DG Lie $R$-algebra 
$\m \hot \g = \bigoplus_{i \in \Z} \, \m \hot \g^i$.
Following \cite{De} and \cite{Ge1} we consider the {\em Deligne crossed
groupoid} $\opn{Del}(\g, R)$ (better known as the Deligne $2$-groupoid). Recall
that the set of objects of 
$\opn{Del}(\g, R)$ is the set $\opn{MC}(\m \hot \g)$ of solutions of the
Maurer-Cartan equation. The $1$-morphisms in $\opn{Del}(\g, R)$ are the
elements of the gauge group $\opn{Exp}(\m \hot \g^0)$, and the 
$2$-morphisms in $\opn{Del}(\g, R)$ are the
elements of the groups $\opn{Exp}(\m \hot \g^{-1})_{\om}$, for 
$\om \in \opn{MC}(\m \hot \g)$. See Sections \ref{sec:defs-DG-Lie} and
\ref{sec:Lie-desc} for details. 

Now suppose $\g = \{ \g^p \}_{p \in \N}$ is a {\em cosimplicial quantum type DG
Lie algebra}. So in each simplicial dimension $p$ there is a quantum type DG Lie
algebra $\g^p = \bigoplus_{i \geq -1} \, \g^{p, i}$.
Since the Deligne crossed groupoid is functorial, we get a cosimplicial
crossed groupoid 
$\opn{Del}(\g, R) = \bigl\{ \opn{Del}(\g^p, R) 
\bigr\}_{p \in \N}$.
The combinatorial descent construction from Subsection \ref{subsec:combinat}
gives rise to the set of {\em Lie descent data} \lb
$\opn{Desc} \bigl( \opn{Del}(\g, R) \bigr)$,
and its quotient set (modulo gauge equivalence)
$\ol{\opn{Desc}} \bigl( \opn{Del}(\g, R) \bigr)$.

Suppose $\h =  \{ \h^p \}_{p \in \N}$ is another cosimplicial quantum type DG
Lie algebra. A cosimplicial
$\mrm{L}_{\infty}$ morphism $\Psi : \g \to \h$ 
is a collection $\Psi^p : \g^p \to \h^p$ of $\mrm{L}_{\infty}$ morphisms,
that respect the cosimplicial structures. We say that $\Psi$ is a cosimplicial
{\em $\mrm{L}_{\infty}$ quasi-isomorphism} if each $\Psi^p : \g^p \to \h^p$
is an $\mrm{L}_{\infty}$ quasi-isomorphism. 

\begin{thm}[Equivalence for Lie Descent] \label{thm:intro.103}
Let $\Psi : \g \to \h$ be a cosimplicial $\mrm{L}_{\infty}$ morphism between 
cosimplicial quantum type DG Lie algebras, and let 
$\nu : R \to S$ be a homomorphism between parameter algebras,
all over a field $\K$ of characteristic $0$. Then there is a function 
\[ \ol{\opn{Desc}} (\opn{Del})(\Psi, \nu)  : 
\ol{\opn{Desc}} ( \opn{Del}(\g, R) ) \to 
\ol{\opn{Desc}} ( \opn{Del}(\h, S) ) , \]
which is functorial in $\Psi$ and $\nu$.
Furthermore, if $\Psi$ is a cosimplicial  $\mrm{L}_{\infty}$ quasi-isomorphism
and $\nu$ is an isomorphism, then 
$\ol{\opn{Desc}} (\opn{Del})(\Psi, \nu)$ is bijective. 
\end{thm}

This is repeated with more details as Theorem \ref{thm:Lie-desc.101} in the body
of the paper. The proof uses Theorem
\ref{thm:intro.100}, and results from \cite{Ye7} on the Deligne crossed groupoid
and on the {\em bar-cobar construction} for DG Lie algebras.

\subsection{Geometric Descent} \label{subsec:geodesc}
Let 
$\twocat{P} = ( \twocat{P}_1, \twocat{P}_2, 
\opn{Ad}_{\twocat{P}_1 \crvar \twocat{P}_{2}}, \opn{D} )$
be some stack of crossed groupoids on a topological space
$X$. We denote the set of twisted objects of $\twocat{P}$ by
$\opn{TwOb}(\twocat{P})$. 
There is an intrinsic notion of twisted gauge equivalence between 
twisted objects (see Definition \ref{dfn:33}), and we denote the resulting set
of equivalence classes by 
$\ol{\opn{TwOb}}(\twocat{P})$. 

Let $\bsym{U} = \{ U_k \}_{k \in K}$  be an open covering of the space $X$.
For $k_0, \ldots, k_m \in K$ we write 
$U_{k_0, \ldots, k_m} := U_{k_0} \cap \cdots \cap U_{k_m}$.
A {\em geometric descent datum} for twisted objects of $\twocat{P}$, relative to
this covering, consists of a collection 
$\{ \mcal{A}_{k_0} \}_{k_0 \in K}$ of objects 
$\mcal{A}_{k_0} \in \opn{Ob}(\twocat{P}(U_{k_0}))$, a collection 
$\{ g_{k_0, k_1} \}$ of $1$-morphisms
\[ g_{k_0, k_1} : \mcal{A}_{k_0}|_{U_{k_0, k_1}} \iso 
\mcal{A}_{k_1}|_{U_{k_0, k_1}} \]
in $\twocat{P}_1(U_{k_0, k_1})$, and a collection
$\{ a_{k_0, k_1, k_2} \}$ of $2$-morphisms 
\[ a_{k_0, k_1, k_2} \in 
\twocat{P}_2(U_{k_0, k_1, k_2})(\mcal{A}_{k_0}|_{U_{k_0, k_1, k_2}}) . \]
The conditions are:
\[ g_{k_0, k_2}^{-1} \circ g_{k_1, k_2} \circ g_{k_0, k_1} =
\opn{D}(a_{k_0, k_1, k_2}) \]
as automorphisms of $\mcal{A}_{k_0}|_{U_{k_0, k_1, k_2}}$
in $\twocat{P}_1(U_{k_0, k_1, k_2})$, and 
\[ a_{k_0, k_1, k_3}^{-1} \cdot a_{k_0, k_2, k_3} \cdot 
a_{k_0, k_1, k_2} = 
\opn{Ad}_{\twocat{P}_1 \crvar \twocat{P}_{2}}(g_{k_0, k_1}^{-1})
(a_{k_1, k_2, k_3}) \]
in the group
$\twocat{P}_2(U_{k_0, \ldots, k_3})
(\mcal{A}_{k_0}|_{U_{k_0, \ldots, k_3}})$.

The above can be stated in terms of combinatorial descent (see Subsection
\ref{subsec:combinat}). 
For any open set $U \subset X$ we have a crossed groupoid
$\twocat{P}(U)$. Thus, given an open
covering $\bsym{U}$, the \v{C}ech construction gives rise to a cosimplicial
crossed groupoid $\opn{C}(\bsym{U}, \twocat{P})$. A geometric descent
datum is precisely an element of 
$\opn{Desc}(\opn{C}(\bsym{U}, \twocat{P}))$.

A consequence of Theorem \ref{thm:17} is this result (it is a combination of 
Corollary \ref{cor:4}, Theorem \ref{thm:glu.1}, Proposition
\ref{prop:bitorsors.6} and Proposition \ref{prop:206} in the body of the
paper): 

\begin{thm}[Decomposition] \label{thm:intro.101}
Let $\K$ be a field of characteristic $0$, $X$ a smooth algebraic variety
over $\K$, $R$ a parameter algebra over $\K$, and $\bsym{U}$ an
affine open covering of $X$. Denote by $\twocat{P}(R, X)$ either of the stacks
of crossed groupoids \lb $\twocat{AssDef}(R, \mcal{O}_X)$ or
$\twocat{PoisDef}(R, \mcal{O}_X)$. Then there is a bijection
\[ \ol{\opn{dec}} : 
\ol{\opn{TwOb}}  \bigl( \twocat{P}(R, X) \bigr) \iso
\ol{\opn{Desc}} \bigl( \opn{C} (\bsym{U}, \twocat{P}(R, X)) \bigr)  \]
called {\em decomposition}, which is functorial with respect to refinements
$\bsym{U}' \to \bsym{U}$ of affine open coverings, and with respect to
homomorphisms $R \to R'$ of parameter algebras.
\end{thm}

The ideas of geometric descent data and decomposition are not new; they appear
in most texts about gerbes (cf.\ \cite[Section IV.3.5]{Gi}, and also
\cite{Br, Ko2, BGNT1, DP1}). However, a result such
as Theorem \ref{thm:17} is usually not available, and the authors usually impose
something like Theorem \ref{thm:17} as a condition (explicitly or
implicitly), or else they are forced to deal with hypercoverings.

What Theorem \ref{thm:intro.101} says is that we can replace the study of
twisted deformations with that of geometric descent data.

\subsection{DG Lie Algebras and Deformations} \label{subsec:geom}
On the smooth algebraic variety $X$ there are two important sheaves of DG Lie
algebras: the sheaf $\mcal{T}_{\mrm{poly}, X}$ of {\em polyderivations}, and the
sheaf $\mcal{D}_{\mrm{poly}, X}$ of {\em polydifferential operators}. 
Within $\mcal{D}_{\mrm{poly}, X}$ there is the subalgebra
$\mcal{D}_{\mrm{poly}, X}^{\mrm{nor}}$ 
of {\em normalized polydifferential operators}. 
The sheaves $\mcal{T}_{\mrm{poly}, X}$, $\mcal{D}_{\mrm{poly}, X}$ and
$\mcal{D}_{\mrm{poly}, X}^{\mrm{nor}}$
are quasi-coherent graded $\OX$-modules, and the inclusion 
$\mcal{D}_{\mrm{poly}, X}^{\mrm{nor}} \to \mcal{D}_{\mrm{poly}, X}$
is a quasi-isomorphism. 

\begin{thm}[{Geometrization, \cite[Theorem 0.1]{Ye9}}] \label{thm:120}
Let $\K$ be a field of characteristic $0$, $X$ a smooth algebraic variety
over $\K$, $R$ a parameter algebra over $\K$, 
and $U$ an affine open set in $X$.
There are equivalences of crossed groupoids
\[  \opn{geo} : \
\opn{Del} \bigl( \Gamma(U, \mcal{D}_{\mrm{poly}, X}^{\mrm{nor}}) , R \big)
\to \cat{AssDef}(R, \mcal{O}_U) \]
and
\[  \opn{geo} : \
\opn{Del} \bigl( \Gamma(U, \mcal{T}_{\mrm{poly}, X}) , R \big)
\to \cat{PoisDef}(R, \mcal{O}_U) \]
which we call {\em geometrization}. The equivalences  $\opn{geo}$
commute with homomorphisms $R \to R'$ of parameter algebras, and with inclusions
of affine open sets $U' \to U$.
\end{thm}

See Definition \ref{dfn:227} regarding equivalences of crossed groupoids.

In earlier versions of this paper the proof of Theorem \ref{thm:120} was
part of the material. However the paper got to be extremely long, and
hence we extracted all the material leading to Theorem
\ref{thm:120}, and made it into the new paper \cite{Ye9}. 
In terms of content, everything involving stacks remained in this paper, and
the rest was moved to \cite{Ye9}. 

The proof of Theorem \ref{thm:120}, in \cite{Ye9}, has several difficult
components. First there is the issue of $\m$-adically complete flat sheaves of
$R$-modules on the variety $X$, that required a lot of basic work 
(some of it in \cite{Ye5}). Regarding associative deformations, the
geometrization required a delicate analysis of sheaves of noncommutative rings 
on an algebraic variety, using {\em Ore localization}. The fact that the DG Lie
algebra $\mcal{D}_{\mrm{poly}, X}^{\mrm{nor}}$ controls  associative
deformations (without any differentiality restriction on these deformations!) 
is a new result. In our earlier paper \cite{Ye1} we had to put a local
differentiality condition on associative deformations (that we now know is
redundant). It is important to note that the corresponding statement 
for complex analytic manifolds is an open problem; see \cite[Remark 2.2.7]{KS2}.

Consider a finite affine open covering $\bsym{U}$ of $X$. The \v{C}ech
construction gives rise to the  cosimplicial quantum type DG Lie algebras
$\opn{C}(\bsym{U}, \mcal{T}_{\mrm{poly}, X})$ and 
$\opn{C}(\bsym{U}, \mcal{D}_{\mrm{poly}, X}^{\mrm{nor}})$. 
The combination of Theorems \ref{thm:120} and \ref{thm:intro.100} yields:

\begin{cor}[Geometrization of Descent Data] \label{cor:110}
Let $\K$ be a field of characteristic $0$, $X$ a smooth algebraic variety
over $\K$, $R$ a parameter algebra over $\K$, 
and $\bsym{U}$ a finite affine open covering of $X$.
There are bijections
\[  \ol{\opn{geo}} : 
\ol{\opn{Desc}} \bigl( \opn{Del} \bigl(
\mrm{C}(\bsym{U}, \mcal{D}_{\mrm{poly}, X}^{\mrm{nor}}) , R 
\bigr) \bigr)
\iso
\ol{\opn{Desc}} \bigl( \opn{C} \bigl(
\twocat{AssDef}(R, \mcal{O}_X) , \bsym{U} \bigr) \bigr) \]
and 
\[  \ol{\opn{geo}} : 
\ol{\opn{Desc}} \bigl( \opn{Del} \bigl(
\mrm{C}(\bsym{U}, \mcal{T}_{\mrm{poly}, X}) , R 
\bigr) \bigr)
\iso
\ol{\opn{Desc}} \bigl( \opn{C} \bigl(
\twocat{PoisDef}(R, \mcal{O}_X) , \bsym{U} \bigr) \bigr) . \]
The functions $\ol{\opn{geo}}$ respect homomorphisms
$R \to R'$ of parameter algebras, and 
refinements of coverings $\bsym{U}' \to \bsym{U}$.
\end{cor}

This is Theorem \ref{thm:11} in the body of the paper.

We know that $\mcal{T}_{\mrm{poly}, X}$ and 
$\mcal{D}_{\mrm{poly}, X}^{\mrm{nor}}$ are sheaves in the \'etale topology of
$X$. This, with Corollary \ref{cor:110} and  Theorem \ref{thm:intro.101}, imply
that 
$\ol{\opn{TwOb}}  \bigl( \twocat{P}(R, X) \bigr)$
is contravariant with respect to \'etale morphisms $X' \to X$.
The set $\ol{\opn{TwOb}}  \bigl( \twocat{P}(R, X) \bigr)$
is also covariant with respect to homomorphisms $R \to R'$ of parameter
algebras. See Theorem \ref{thm:7} for the precise statement.

\subsection{The Formality Morphism}
Given a finite affine open covering $\bsym{U}$ of $X$,
we can form the {\em mixed resolution} 
$\opn{Mix}_{\bsym{U}}(\mcal{M})$ of any 
quasi-coherent sheaf $\mcal{M}$. This resolution is functorial, and moreover 
the complex $\opn{Mix}_{\bsym{U}}(\mcal{M})$ is acyclic for the global sections
functor $\Gamma(X, -)$. 
Resolving the sheaves $\mcal{T}_{\mrm{poly}, X}$ and 
$\mcal{D}_{\mrm{poly}, X}$ we get sheaves of  DG Lie algebras 
$\opn{Mix}_{\bsym{U}}(\mcal{T}_{\mrm{poly}, X})$ and
$\opn{Mix}_{\bsym{U}}(\mcal{D}_{\mrm{poly}, X})$. 
See \cite[Proposition 6.3]{Ye1} for details.

Assume $\opn{dim} X = n$. An {\em \'etale coordinate system} on an open set 
$U \subset X$ is an \'etale morphism $s : U \to \mbf{A}^n_{\K}$. 
By {\em open covering with \'etale coordinates} of $X$ we mean 
a pair $(\bsym{U}, \bsym{s})$, consisting of an open covering
$\bsym{U} = \{ U_k \}$, together with a family $\bsym{s} = \{ s_k \}$ of 
\'etale coordinate systems $s_k : U_k \to \mbf{A}^n_{\K}$.

One of the main results of \cite{Ye1} is that there is an 
$\mrm{L}_{\infty}$ quasi-isomorphism
\[ \Psi_{\bsym{s}} = \{ \Psi_{\bsym{s}; i} \}_{i \geq 1} :
\opn{Mix}_{\bsym{U}}(\mcal{T}_{\mrm{poly}, X}) \to 
\opn{Mix}_{\bsym{U}}(\mcal{D}_{\mrm{poly}, X})  \]
between sheaves of DG Lie algebras on $X$. 
This is \cite[Theorem 0.2]{Ye1}, and it relies on the universal formality
morphism of Kontsevich from \cite{Ko1}. See \cite[Theorem 1.1]{VdB} for
another proof. In Theorem \ref{thm:13} in the body of the paper we expand
\cite[Theorem 0.2]{Ye1}. As a consequence we obtain the next theorem, which is
repeated, with more details, as Theorem \ref{thm:221}.

\begin{thm}[Cosimplicial Formality] \label{thm:22}
Assume the field $\K$ contains $\R$.
Let $X$ be a smooth algebraic variety over $\K$, and let
$(\bsym{U}, \bsym{s})$ be a finite affine open covering with
\'etale coordinates of $X$. Then there is a diagram
\[ \UseTips \xymatrix @C=12ex @R=5ex {
\mrm{C}(\bsym{U}, \mcal{T}_{\mrm{poly}, X})
\ar[d]
&
\mrm{C}(\bsym{U}, \mcal{D}_{\mrm{poly}, X}^{\mrm{nor}})
\ar[d]
\\
\mrm{C} \bigl( \bsym{U} , \opn{Mix}_{\bsym{U}}(\mcal{T}_{\mrm{poly}, X})
\bigr)
\ar[r]^{\mrm{C}(\bsym{U} , \Psi_{\bsym{s}})} 
&
\mrm{C} \bigl( \bsym{U} , \opn{Mix}_{\bsym{U}}(\mcal{D}_{\mrm{poly}, X})
\bigr) \ ,
} \]
where the objects are cosimplicial quantum type DG Lie algebras, the
vertical arrows are cosimplicial DG Lie quasi-isomorphisms, and the horizontal
arrow is a cosimplicial $\mrm{L}_{\infty}$ quasi-isomorphism. 
This diagram is independent \tup{(}up to cosimplicial quasi-isomorphism\tup{)}
of the covering $(\bsym{U}, \bsym{s})$, and respects \'etale morphisms 
$X' \to X$.
\end{thm}

The reason for the condition $\R \subset \K$ is because in \cite{Ye1}
we used the original universal quantization morphism of Konstevich \cite{Ko1}. 
See Remark \ref{rem:TDQ.101} for a discussion on how to relax this condition.

\subsection{The Twisted Quantization Theorem} \label{subsec:TDQ}
There is a notion of {\em twisted gauge equivalence} between  
twisted associative $R$-deformation of $\mcal{O}_X$; this is the same as the 
twisted gauge equivalence in 
$\opn{TwOb}(\twocat{P})$ for 
$\twocat{P} := \twocat{AssDef}(R, \mcal{O}_X)$.
Likewise for twisted Poisson deformations. 

A twisted deformation $\gerbe{A}$ induces a {\em first order bracket}
$\{ -,- \}_{\gerbe{A}}$ on $\mcal{O}_X$.
See Definition \ref{dfn:22}, or the text preceding Example \ref{exa:23} above.

Here is the main result of our paper:

\begin{thm}[Twisted Quantization] \label{thm:14}
Let $\K$ be a field containing the real numbers, let $R$ be a 
parameter $\K$-algebra, and let $X$ be a smooth algebraic variety over
$\K$. Then there is a bijection of sets
\[ \begin{aligned}
& \opn{tw{.}quant} : 
\frac{ \{ \tup{twisted Poisson $R$-deformations of 
$\mcal{O}_X$} \} }
{\tup{ twisted gauge equivalence}} \\[0.3em]
& \hspace{14ex} \xar{\ \cong \ }  \
\frac{ \{ \tup{twisted associative $R$-deformations of 
$\mcal{O}_X$} \} }
{\tup{ twisted gauge equivalence}}
\end{aligned} \] 
called {\em twisted quantization}. 
The bijection $\opn{tw{.}quant}$ preserves first order brackets, and commutes
with homomorphisms of parameter algebras $R \to R'$, and with  \'etale morphisms
of varieties $X' \to X$.
\end{thm}

This theorem is repeated in greater detail as Theorem \ref{thm:8}.
It is proved by assembling Theorem \ref{thm:intro.101}, Corollary
\ref{cor:110}, Theorem \ref{thm:22} and Theorem \ref{thm:intro.103}.
The roles of these results can be seen in the diagram in Figure 
\ref{fig:diagram}. 
Specifically, the  bijections $\ol{\opn{dec}}$ in the diagram are from 
Theorem \ref{thm:intro.101}; the  bijections $\ol{\opn{geo}}$ are from 
Corollary \ref{cor:110}; and the bijections 
$\ol{\opn{mix}}$ and $\ol{\Psi}_{\bsym{s}}$ are gotten by applying 
Theorem \ref{thm:intro.103} to the diagram of cosimplicial 
$\mrm{L}_{\infty}$ quasi-isomorphisms in Theorem \ref{thm:22}.

\begin{figure}
\[ \UseTips \xymatrix @C=5ex @R=5ex {
\ol{\opn{TwOb}} \bigl( \twocat{PoisDef}(R, \mcal{O}_X) \bigr)
\ar[d]_{\ol{\opn{dec}}}
\ar@{-->}[r]^{\opn{tw{.}quant}}
&
\ol{\opn{TwOb}} \bigl( \twocat{AssDef}(R, \mcal{O}_X) \bigr)
\ar[d]^{\ol{\opn{dec}}}
\\
\ol{\opn{Desc}} \bigl( 
\mrm{C}( \bsym{U}, \twocat{PoisDef}(R, \mcal{O}_X) ) \bigr)
&
\ol{\opn{Desc}} \bigl( 
\mrm{C}( \bsym{U}, \twocat{AssDef}(R, \mcal{O}_X) ) \bigr)
\\
\ol{\opn{Desc}}  \bigl( \opn{Del}
\bigl( \mrm{C}(\bsym{U}, \mcal{T}^{}_{\mrm{poly}, X}), R \bigr) \bigr)
\ar[u]^{\ol{\opn{geo}}}
\ar[d]_{\ol{\opn{mix}}}
&
\ol{\opn{Desc}} \bigl( \opn{Del} \bigr(
\mrm{C}(\bsym{U}, \mcal{D}^{\mrm{nor}}_{\mrm{poly}, X}) , R \bigr) \bigr)
\ar[u]_{\ol{\opn{geo}}}
\ar[d]^{\ol{\opn{mix}}}
\\
\ol{\opn{Desc}} \bigl( \opn{Del} \bigl(
\mrm{C}(\bsym{U}, \opn{Mix}_{\bsym{U}}(\mcal{T}_{\mrm{poly},X}) ), R \bigr)
\bigr)
\ar[r]^{\ol{\Psi}_{\bsym{s}}}
& 
\ol{\opn{Desc}} \bigl( \opn{Del} \bigl(
\mrm{C}(\bsym{U}, \opn{Mix}_{\bsym{U}}(\mcal{D}_{\mrm{poly}, X}) ), R \bigr)
\bigr)
} \]
\caption{Diagram of bijections in the proof of Theorem \ref{thm:14}.}
\label{fig:diagram}
\end{figure}

In Corollaries \ref{cor:TDQ.105} and \ref{cor:TDQ.102} we explain what happens
when some or all of the cohomology groups
$\mrm{H}^1(X, \mcal{O}_X)$, $\mrm{H}^2(X, \mcal{O}_X)$, 
$\mrm{H}^1(X, \mcal{T}_X)$ and $\mrm{H}^1(X, \mcal{D}_X)$
vanish.

A twisted $R$-deformation $\gerbe{A}$ of $\mcal{O}_X$ is said to be {\em really
twisted} if it has no global objects, namely if the groupoid $\gerbe{G}(X)$ is
empty (where $\gerbe{G}$ is the gauge gerbe of  $\gerbe{A}$).

Any deformation $\mcal{A}$ generates a twisted
deformation $\gerbe{A}$, and $\mcal{A}$ is a global object of
$\gerbe{A}$. See Example \ref{exa:19}. Any twisted deformation $\gerbe{A}$
that is not really twisted is of this form.

One consequence of Theorem \ref{thm:6} is that in the differentiable
setup (and $\K = \mbb{R}$) there are no really twisted
$R$-deformations. 
However, when $X$ is a complex analytic manifold (and $\K = \mbb{C}$),
there do exist really twisted $R$-deformations. Presumably our methods
(with minor adjustments) should work also for complex analytic
manifolds, and a result similar to Theorem \ref{thm:14} should hold.

Here is a question that we find intriguing. 

\begin{que}
Let $X$ be a Calabi-Yau surface, with symplectic Poisson bracket 
$\{ -,- \}$. Consider the Poisson $\K[[\hbar]]$-deformation 
$\mcal{A}$ of $\mcal{O}_X$ that we get from Example \ref{exa:23}, and the 
twisted Poisson deformation $\gerbe{A}$ generated by $\mcal{A}$.
Let $\gerbe{B} := \opn{tw{.}quant} (\gerbe{A})$,
which is a twisted associative $\K[[\hbar]]$-deformation 
of $\mcal{O}_X$, well-defined up to twisted gauge equivalence. Is
$\gerbe{B}$ really twisted? 
\end{que}

This question is repeated (with a few more details) at the end of Section
\ref{sec:TDQ}, along with a some other questions.

\subsection{Discussion of Related Work} \label{subsec:discuss}
We finish the introduction with a discussion of several related papers.
This is not an easy task, since these papers were written by people working in
several mathematical branches -- differential geometry, complex
analytic geometry, algebraic geometry and mathematical physics -- and the style
of writing varies accordingly.  

The impetus for our work was trying to understand the ideas 
proposed by Kontsevich in \cite{Ko2}. In that paper Kontsevich considered a
smooth algebraic variety $X$ over a field of characteristic $0$. He claimed
that any formal Poisson deformation of $\OO_X$ can be quantized to a stack of
algebroids. However there was only a hint of a proof in \cite{Ko2}, and there
was no statement regarding classification. We found this new idea of Kontsevich
extremely interesting, and wanted to prove it, and to expand it to a
precise classification statement (generalizing the affine case). We had several
conversations with Kontsevich, in which he offered some valuable suggestions,
that eventually led us to the concept of twisted Poisson deformation 
(see Subsection \ref{subsec:tw.defs}). The precise statement
about twisted deformation quantization of an
algebraic variety appears for the first time in this paper (this is 
Theorem \ref{thm:14}); and ours is the only paper that gives a complete proof
of this result. 

As we already mentioned, Kashiwara \cite{Ka} was the first to consider 
stacky quantizations. The situation in \cite{Ka} was this: $X$ is a complex
contact manifold, so
there is an open covering $X = \bigcup U_i$, where each $U_i$ is
isomorphic to an open submanifold of a projective cotangent bundle. Thus 
on $U_i$ there is a sheaf $\mcal{E}_i$ of microdifferential operators.
Kashiwara proved that even if the  $\mcal{E}_i$ cannot be glued into a
sheaf of rings, they can always be glued into a twisted sheaf (i.e.\
a stack of algebroids) $\gerbe{E}$. 

The paper \cite{KS2} by Kashiwara and P. Schapira discusses a variant 
of the stack of  microdifferential operators, that is called a {\em DQ
algebroid}. The description of a DQ algebroid $\gerbe{A}$ as a twisted object
is explained in Remark \ref{rem:6}.
There is some overlap regarding technical points between our paper and
\cite{KS2} (see Remark \ref{rem:221}, and the Introduction of \cite{Ye9}). 
The main thrust of \cite{KS2} is the study of the Hochschild homology sheaf 
associated to a symplectic DQ algebroid $\gerbe{A}$, and related 
characteristic classes. It does not contain any sort of twisted quantization
statement similar to our Theorem \ref{thm:14}.

P. Polesello and A. D'Agnolo  \cite{Po, DP1, DP2} studied twisted associative
deformations of complex manifolds, and quantizations of contact manifolds, in
the spirit of  \cite{Ka, KS2}. The relationship between \cite{DP2}
and our paper is mostly that in both papers heavy use is made of stacks of
crossed groupoids as a means of classifying twisted forms of stacks (as in
Theorem \ref{thm:intro.101} above). The paper \cite{Po} classifies twisted
associative deformations corresponding to a fixed  symplectic structure; there
is nothing in it resembling our Theorem \ref{thm:14}. One should also note (this
was pointed to us by Polesello) that the deformations in \cite{Po} and
\cite{KS2} are not augmented (see Remark \ref{rem:6}). 

E. Getzler wrote two papers \cite{Ge1, Ge2} on higher Deligne groupoids. 
The relevant results in \cite{Ge2}, and a certain difficulty in that paper, are
explained in Remark \ref{rem:Lie-desc.101}. 
For some time we had hoped to rely on Getzler's results for a proof of 
Theorem \ref{thm:intro.103} (at least for artinian parameter algebras), but we
could not find a satisfactory way to overcome the difficulty mentioned
in Remark \ref{rem:Lie-desc.101}. This is why we eventually provided our
own proof, that does not rely on Getzler's work.
 
P. Bressler, A. Gorokhovsky, R. Nest and B. Tsygan wrote a series of papers on
deformations of real and complex manifolds. In the paper \cite{BGNT1}
there is a discussion of cosimplicial DG Lie
algebras, with results similar to our Theorem \ref{thm:intro.103},
but only for DG Lie homomorphisms
$\g \to \h$ between cosimplicial DG Lie algebras (not $\mrm{L}_{\infty}$
morphisms), and only for artinian parameter algebras. 
Nonetheless several results in \cite{BGNT1} inspired us when we discovered
Theorems \ref{thm:intro.100} and \ref{thm:intro.103}.
Stacks -- in the true geometric sense -- do
not actually occur in \cite{BGNT1}; the authors only look at what we refer to
in Subsection \ref{subsec:geodesc} 
as geometric descent data. Also there is no twisted quantization result
similar to our Theorem \ref{thm:14}; in a large part of the paper
they only consider a particular symplectic Poisson bracket and study its
twisted associative quantizations. 

The paper \cite{BGNT2} considers twisted associative deformations
of twisted $\OO_X$-algebroids on a real or complex  manifold $X$. If we
understand correctly, the associative deformations considered in \cite{BGNT2}
are by definition locally differential, so the authors avoid touching
the hard technical issues lying behind our Theorem \ref{thm:120} 
(see discussion in Subsection \ref{subsec:geom}).
The paper \cite{BGNT2} does not contain a twisted quantization result
similar to our Theorem \ref{thm:14}. As far as we can tell one of the main
results there, namely \cite[Theorem 5.22]{BGNT2}, is analogous to our Corollary
\ref{cor:110} in the associative case. 

The third paper by this team of authors is \cite{BGNT3}. 
They consider several kinds of nerves (simplicial sets) associated with
$2$-groupoids, and prove these are homotopy equivalent. One of the methods used
in \cite{BGNT3} is nonabelian multiplicative integration of surfaces, as in our
paper \cite{Ye6}; indeed, they reproduce some of our results (with shorter
proofs in special cases). Among other things, it is claimed in \cite{BGNT3}
that the difficulty in \cite{Ge2} mentioned above can be settled (more on this
in Remark \ref{rem:Lie-desc.101}). 
It should be interesting to compare the results of \cite{BGNT3} to our Theorems 
\ref{thm:intro.100} and \ref{thm:intro.103}.

The fourth paper in this series is \cite{BGNT4}. Here they look at a 
$\mrm{C}^{\infty}$ manifold $X$ over $\K = \R$, with structure sheaf $\OO_{X}$. 
They start with a twist $\gerbe{B}_0$ of $\OO_{X}$, namely 
a stack of $\R$-algebroids $\gerbe{B}_0$ on $X$ that is locally equivalent to 
$\OO_{X}$. Such twists are classified by cohomology classes 
$[\phi] \in \opn{H}^3_{\mrm{dR}}(X)$.
Let $R$ be an artinian parameter $\R$-algebra with maximal ideal $\m$.  
Theorem 1.1 of \cite{BGNT4} says that associative 
$R$-deformations $\gerbe{B}$ of $\gerbe{B}_0$ are classified by 
the Deligne crossed groupoid of an $\mrm{L}_{\infty}$ algebra, 
whose underlying graded $\R$-module is 
$\Ga(X, \mcal{T}_{\mrm{poly}, X}) \ot \m$.
The binary operation is the usual one, but there is a ternary operation 
coming from a closed $3$-form $\phi$ on $X$ that represents the de Rham class 
$[\phi]$ of $\gerbe{B}_0$. This means that stacky associative $R$-deformations 
$\gerbe{B}$ correspond to $R$-linear $\phi$-Poisson structures on $X$, in the 
sense of Severa and Weinstein (cf.\ \cite{Se}). In this way
\cite[Theorem 1.1]{BGNT4} can be viewed as describing a twisted quantization 
operation.  However, the stackyness of an $R$-deformation $\gerbe{B}$ is the 
same as that of $\gerbe{B}_0 = \R \ot_{R} \gerbe{B}$. Thus this result is 
``perpendicular'' to our twisted quantization from Theorem \ref{thm:14}: for us 
the algebroid $\gerbe{B}_0$ is always trivial (i.e.\ $\gerbe{B}_0 = \OO_X$), and 
the stackyness is only in the $R$-linear deformation $\gerbe{B}$. 

The last paper to be discussed is \cite{CH}. The aim of that paper is
to prove a twisted deformation quantization like our Theorem \ref{thm:14}.
However in \cite{CH} the authors do not consider actual geometric twisted
deformations, but only the corresponding Lie descent
data. Thus they do not have any kind of geometrization result like our Corollary
\ref{cor:110}, nor a decomposition result like our Theorem 
\ref{thm:intro.101}. Their main result, \cite[Theorem 3.5]{CH}, 
seems to assert (in our notation, and when $X$ is a
smooth algebraic variety) that there is a bijection
\[ \ol{\opn{Desc}} \bigl( \opn{Del} \bigl(
\mrm{C}(\bsym{U}, \mcal{T}_{\mrm{poly}, X}) , \K[[\hbar]] 
\bigr) \bigr)
\ \cong \
\ol{\opn{Desc}} \bigl( \opn{Del} \bigl(
\mrm{C}(\bsym{U}, \mcal{D}_{\mrm{poly}, X}^{\mrm{nor}}) , \K[[\hbar]]  
\bigr) \bigr) . \]
Thus it is only a part of our Theorem \ref{thm:14} (the bottom half of the
diagram in Figure \ref{fig:diagram}).
We have some reservations regarding the validity of the proofs in \cite{CH},
for these reasons: 
the proofs in \cite{CH} are not all clear; delicate issues about pronilpotent
groupoids (such as those dealt with in the paper \cite{Ye7}) are ignored; and
there are some specific problems in homotopy theory (cf.\ Remarks
\ref{rem:cosim.106} and \ref{rem:Lie-desc.101}).

\medskip \noindent
\textbf{Acknowledgments.}
Work on this paper began together with Fredrick Leitner, and I wish to
thank him for his contributions. Many of the ideas in this paper are influenced
by the work of Maxim Kontsevich, and I am grateful to him for discussing this
material with me. Thanks also to Michael Artin, Pavel Etingof, Damien
Calaque, Michel Van den Bergh, Pierre Deligne, Lawrence Breen, Pierre Schapira,
Assaf Hasson, George Bergman, James Stasheff, Vladimir Hinich, Pietro
Polesello, Francois Petit and Matan Prasma for their assistance on various
aspects of this project. 
Finally I wish to thank the anonymous referee for reading the paper carefully 
and suggesting some further references.

\section{Background on Crossed Groupoids and Stacks}
\label{sec:back}
\numberwithin{equation}{section}

In this paper we work a lot with stacks on topological spaces. The reader can
consult our earlier paper \cite{Ye4}, that contains a pretty detailed account
of $2$-categories, stacks and gerbes, with the same notation (Sections 1-2 of
that paper). 

Regarding set theoretic considerations, we employ Grothendieck universes. We 
consider a universe $\cat{U}$, that we call the small universe. Elements of
$\cat{U}$ are called small sets. By default all groups, rings etc.\ are small,
i.e.\ their underlying sets are small. And by default all categories $\cat{C}$
are $\cat{U}$-categories, i.e.\ $\opn{Ob}(\cat{C}) \subset \cat{U}$, and 
$\opn{Hom}_{\cat{C}}(i, j) \in \cat{U}$
for all $i, j \in \opn{Ob}(\cat{C})$.  Thus $\cat{Set}$, $\cat{Grp}$ etc.\
denote the categories of small sets, small groups etc., and they are
$\cat{U}$-categories. However some of the categories will be small categories;
a category $\cat{C}$ is small if $\opn{Ob}(\cat{C})$ is a small set.
There is also a universe $\cat{V}$, that we call the big universe, and 
$\cat{U}$ is an element of $\cat{V}$. Elements of $\cat{V}$ will be called big
sets. The $2$-category 
$\twocat{Cat}$ of all $\cat{U}$-categories is a $\cat{V}$-category (namely 
$\opn{Ob}(\twocat{Cat}) \subset \cat{V}$).  

Recall that a {\em groupoid} is a category in which all morphisms are
invertible. Given a groupoid $G$ and objects $i, j \in \opn{Ob}(G)$
we usually write
$G(i, j) := \opn{Hom}_{G}(i, j)$,
the set of morphisms from $i$ to $j$; and 
$G(i) := G(i, i)$, the automorphism group of the object $i$. 

\begin{dfn} \label{dfn:204}
Let $G$ be a groupoid.
\begin{enumerate}
\item Let $i, j \in \opn{Ob} (G)$ and $g \in G(i, j)$.  We define an isomorphism
of groups
\[ \opn{Ad}_{G}(g) : G(i) \iso G(j)  \]
by the formula
$\opn{Ad}_{G}(g)(h) := g \circ h \circ g^{-1}$
for $h \in G(i)$. 

\item The functor 
$\opn{Aut}_{G} : G \to \cat{Grp}$
is defined as follows: on objects $\opn{Aut}_{G}(i) := G(i)$, and on morphisms 
$\opn{Aut}_{G}(g) := \opn{Ad}_{G}(g)$. 

\item Let $H$ be another groupoid, and let $\Phi : G \to H$ be a functor. 
Define the natural transformation 
\[ \opn{Aut}_{\Phi} : \opn{Aut}_{G} \twoto \opn{Aut}_{H} \circ \, \Phi \]
of functors $G \to \cat{Grp}$ to be the group homomorphism
$\opn{Aut}_{\Phi} := \Phi : G(i) \to H(\Phi(i))$, for $i \in \opn{Ob}(G)$.
\end{enumerate} 
\end{dfn}

Suppose $G$ and $N$ are groupoids, such that 
$\opn{Ob}(N) = \opn{Ob}(G)$. 
An {\em action} $\Psi$ of $G$ on 
$N$ is a collection of group isomorphisms 
$\Psi(g) : N(i) \iso N(j)$
for all $i, j \in \opn{Ob}(G)$ and $g \in G(i, j)$, 
such that 
$\Psi(h \circ g) = \Psi(h) \circ \Psi(g)$
whenever $g$ and $h$ are composable, and 
$\Psi(1_i)$ is the identity automorphism of the group $N(i)$. 
The prototypical example is the action $\opn{Ad}_{G}$ of $G$ on itself. 

\begin{dfn} \label{dfn:cosim.101}
A {\em crossed groupoid} is a structure 
\[ G = 
( G_1, G_2, 
\opn{Ad}_{G_1 \crvar G_{2}}, \opn{D} ) \]
consisting of:
\begin{itemize}
\item Groupoids $G_1$ and $G_2$, such that 
$G_2$ is totally disconnected, and 
$\opn{Ob}(G_1) = \opn{Ob}(G_2)$. 
We write $\opn{Ob}(G) := \opn{Ob}(G_1)$. 

\item An action $\opn{Ad}_{G_1 \crvar G_{2}}$ of 
$G_1$ on $G_2$, called the {\em twisting}. 

\item A morphism of groupoids (i.e.\ a functor) 
$\opn{D} : G_2 \to G_1$
called the {\em feedback}, which is the identity on objects.
\end{itemize}

These are the conditions:
\begin{enumerate}
\rmitem{i} The morphism $\opn{D}$ is $G_1$-equivariant with respect to
the actions $\opn{Ad}_{G_1 \crvar G_{2}}$ and
$\opn{Ad}_{G_1}$. Namely 
\[ \opn{D}(\opn{Ad}_{G_1 \crvar G_{2}}(g)(a)) =
\opn{Ad}_{G_1}(g)(\opn{D}(a)) \]
in the group $G_1(j)$, for any $i, j \in \opn{Ob}(G)$, 
$g \in G_1(i, j)$ and $a \in G_2(i)$.

\rmitem{ii} For any $i \in \opn{Ob}(G)$ and 
$a \in G_2(i)$ there is equality
\[ \opn{Ad}_{G_1 \crvar G_{2}}(\opn{D}(a)) =
\opn{Ad}_{G_2(i)}(a) , \]
as automorphisms of the group $G_2(i)$.
\end{enumerate}
\end{dfn}

We sometimes refer to morphisms in the groupoid $G_1$ as {\em
$1$-morphisms}, and to the morphisms in $G_2$ as {\em
$2$-morphisms}. 

For later use we record the next facts (that are almost immediate from the
definitions). 

\begin{prop} \label{prop:200}
Let 
$G = ( G_1, G_2, \opn{Ad}_{G_1 \crvar G_{2}}, \opn{D} )$
be a crossed groupoid. 
\begin{enumerate}
\item For $i \in \opn{Ob}(G)$ let $\opn{IG}(i) := G_2(i)$, and for 
$g \in G_1(i, j)$ let 
$\opn{IG}(g) := \opn{Ad}_{G_1 \crvar G_{2}}(g)$. 
Then 
$\opn{IG} : G_1 \to \cat{Grp}$
is a functor. 

\item For $i \in \opn{Ob}(G)$ and $a \in \opn{IG}(i)$ let 
$\opn{ig}(a) := \opn{D}(a) \in G_1(i)$. Then 
$\opn{ig} : \opn{IG} \twoto \opn{Aut}_{G_1}$
is a natural transformation of functors $G_1 \to \cat{Grp}$. 

\item The data $( G_1, G_2, \opn{Ad}_{G_1 \crvar G_{2}}, \opn{D} )$
can be recovered from the groupoid $G_1$, the functor 
$\opn{IG} : G_1 \to \cat{Grp}$ and the 
 natural transformation 
$\opn{ig} : \opn{IG} \twoto \opn{Aut}_{G_1}$.
\end{enumerate}
\end{prop}

Proposition \ref{prop:200} suggests an alternative way to view the crossed 
groupoid $G = ( G_1, G_2, \opn{Ad}_{G_1 \crvar G_{2}}, \opn{D} )$.
We call $\opn{IG}(i) = G_2(i)$ the {\em inner gauge group} of the object $i$,
and its elements are also called {\em inner gauge transformations}. 
The group homomorphism 
$\opn{ig} = \opn{D} : \opn{IG}(i) \to G_1(i)$ is called the {\em inner action}. 
The morphisms in $G_1$ are also called {\em gauge transformations}. 
See Remark \ref{rem:cosim.101} regarding this point of view. 
Warning: the group
homomorphism $\opn{ig} = \opn{D}$
is usually not injective. 

Here are some examples of crossed groupoids.

\begin{exa} \label{exa:200}
Consider any groupoid $G_1$, and let $G_2$ be the totally disconnected
groupoid gotten from $G_1$ by removing all morphisms between distinct objects. 
(More generally one can take any normal subgroupoid $N \subset G_1$, in the
sense of \cite[Definition 3.1]{Ye4}, and define $G_2 := N$.)
Define the twisting  
$\opn{Ad}_{G_1 \crvar G_{2}} := {\opn{Ad}_{G_1}}|_{G_{2}}$,
and the feedback $\opn{D} : G_2 \to G_1$ is the inclusion. 
This is easily seen to be a crossed groupoid. 
\end{exa}

For a category $\cat{C}$ we denote by $\cat{C}^{\times}$ the subcategory having
all the objects of $\cat{C}$, and its morphisms are the isomorphisms of
$\cat{C}$. We call  $\cat{C}^{\times}$ the {\em groupoid of invertible
morphisms of $\cat{C}$}. 
Note that if $G$ is a groupoid, then any functor $G \to \cat{C}$ factors
through $\cat{C}^{\times}$.

\begin{exa} \label{exa:15}
Take the category $\cat{Grp}$ of small groups. We will make $\cat{Grp}^{\times}$
into a crossed groupoid  by introducing inner gauge groups. 
For a group $G$ let $\opn{IG}(G) := G$.
For a group isomorphism $\phi : G \to G'$ (i.e.\ for a morphism
$\phi \in \cat{Grp}^{\times}(G, G')$) we let 
$\opn{IG}(\phi) := \phi$. And for an element $g \in G$  we let
$\opn{ig}(g) := \opn{Ad}_G(g)$, i.e.\ conjugation.
Here the kernel of the homomorphism 
\[ \opn{ig} : \opn{IG}(G) = G = \cat{Grp}^{\times}_2(G)
\to \cat{Grp}^{\times}_1(G) = \opn{Aut}_{\cat{Grp}}(G) \]
is the center of the group $G$.
\end{exa}

\begin{exa} \label{exa:6}
Let $R$ be a commutative ring. Take the category $\cat{Assoc}(R)$
of associative unital $R$-algebras. Then $\cat{Assoc}^{\times}(R)$
is a crossed groupoid. 
The group of inner gauge transformations of an associative $R$-algebra $A$ is
the group $\opn{IG}(A) := A^{\times}$
of invertible elements. It is functorial, since a ring isomorphism 
$g : A \to A'$ sends invertible elements to invertible elements.
The inner action $\opn{ig}(a)$, for 
$a \in \opn{IG}(A)$, is conjugation by this invertible element, namely
$\opn{ig}(a)(a') = a \star a' \star a^{-1}$,
where $\star$ is the multiplication in $A$.
The conditions are very easy to verify.
\end{exa}

\begin{rem} \label{rem:cosim.101}
A crossed groupoid is better known as a 
{\em strict $2$-groupoid}, or a {\em crossed module over a groupoid}, or
as a {\em $2$-truncated crossed complex}; see \cite{Bw}. When
$\opn{Ob}(G)$ is a singleton then $G$ is just a crossed module. 
In earlier versions of this paper, and in \cite{Ye10}, we used the name 
{\em category with inner gauge groups}. A trace of this name remains
in the inner gauge groups $\opn{IG}$. 

We use different notation (fonts) for groupoids and crossed groupoids in this
paper, depending usually on their size. Thus $G$ will most likely be a small
groupoid, whereas $\cat{G}$ will be a big groupoid (i.e.\ a
$\cat{U}$-category).  

Traditionally papers used $2$-groupoid language to discuss descent (cf.\
\cite{BGNT1}). In this paper we realized that the crossed groupoid language is
more effective: geometric descent data comes naturally in terms of a
cosimplicial crossed groupoid (see Definition \ref{dfn:21}), and also the
Deligne construction (see Definition \ref{dfn:Lie-desc.101}) appears as a
crossed groupoid, so it is more natural to talk about the
{\em Deligne crossed groupoid}. 
\end{rem}

\begin{dfn}
Suppose 
$H = ( H_1, H_2, \opn{Ad}_{H_1 \crvar H_{2}}, \opn{D} )$ 
is another crossed groupoid. A {\em morphism of crossed groupoids}
$\Phi : G \to H$
is a pair of morphisms of groupoids 
$\Phi_p : G_p \to H_p$, $p = 1, 2$, that are equal on objects, 
and respect the twistings and the feedbacks. We denote by
$\cat{CrGrpd}$ the category consisting of crossed groupoids and morphisms
between them. 
\end{dfn}

\begin{dfn} \label{dfn:227}
A morphism of crossed groupoids $\Phi : G \to H$ is called an
{\em equivalence} if 
$\Phi_1 : G_1 \to H_1$ is an equivalence of groupoids, and for every
$i \in \opn{Ob}(G)$ the group homomorphism 
$\Phi_2 : G_2(i) \to H_2(\Phi_2(i))$ is bijective. 
\end{dfn}

Let $G$ be a crossed groupoid. We define the homotopy set 
$\bsym{\pi}_0(G) := \bsym{\pi}_0(G_1)$,
namely the set of isomorphism classes of objects of $G_1$.
The homotopy groups of $G$ are 
\[ \bsym{\pi}_1(G, i) := 
\opn{Coker} \bigl( \opn{D} : G_2(i) \to G_1(i) \bigr) \]
and
\[ \bsym{\pi}_2(G, i) := 
\opn{Ker} \bigl( \opn{D} : G_2(i) \to G_1(i) \bigr) \]
for $i \in \opn{Ob}(G)$.
The set $\bsym{\pi}_0(G)$ and the groups $\bsym{\pi}_p(G, i)$ are
functorial in $G$. 

\begin{dfn} \label{dfn:equiv-cosim.102}
A morphism of crossed groupoids $\Phi : G \to H$  
is called a {\em weak equivalence} if the function
$\bsym{\pi}_0(\Phi)  : \bsym{\pi}_0(G) \to \bsym{\pi}_0(H)$
is bijective, and the group homomorphisms
$\bsym{\pi}_p(\Phi, i) : \bsym{\pi}_p(G, i) \to \bsym{\pi}_p(H, \Phi(i))$
are bijective for all $i \in \opn{Ob}(G)$ and 
$p \in \{ 1, 2 \}$.
\end{dfn}

Of course if $\Phi$ is an equivalence then it is a weak equivalence.

\begin{dfn} \label{dfn:200}
Let $X$ be a topological space. A {\em stack of crossed groupoids} 
\[ \twocat{P} = ( \twocat{P}_1, \twocat{P}_2, 
\opn{Ad}_{\twocat{P}_1 \crvar \twocat{P}_{2}}, \opn{D} ) \]
on $X$ consists of a crossed groupoid 
\[ \twocat{P}(U) = \bigl( \twocat{P}_1(U), \twocat{P}_2(U), 
\opn{Ad}_{\twocat{P}_1(U) \crvar \twocat{P}_2(U)}, \opn{D} \bigr) \]
for every open set $U \subset X$, with a morphism of crossed
groupoids 
\[ \mrm{rest}^{}_{V / U} : \twocat{P}(U) \to \twocat{P}(V) , \]
called {\em restriction},
for every inclusion $V \subset U$ of open sets. 
The three conditions below have to hold:
\begin{enumerate}
\rmitem{i} The restriction morphisms satisfy 
$\mrm{rest}^{}_{W / V} \circ \mrm{rest}^{}_{V / U} = \mrm{rest}^{}_{W / U}$
for a double inclusion $W \subset V \subset U$ of open sets of $X$.

\rmitem{ii} The prestack of groupoids 
$\twocat{P}_1 : U \mapsto \twocat{P}_1(U)$
on $X$ is a stack. 

\rmitem{iii} For every open set $U$ and object 
$\mcal{A} \in \opn{Ob}(\twocat{P}(U))$, the presheaf of groups 
\[ \twocat{P}_2(\mcal{A}) : V \mapsto 
\twocat{P}_2(V)(\mrm{rest}^{}_{V / U}(\mcal{A}))  \]
on $U$ is a sheaf.  
\end{enumerate}
\end{dfn}

Condition (i) says that the stack $\twocat{P}_1$ is {\em split}, i.e.\ it is a
cat\'egorie fibr\'ee scind\'ee, in the sense of \cite{Gi}. 

\begin{dfn} \label{dfn:203}
Let $\twocat{P}$ be a stack of crossed groupoids on $X$.
\begin{enumerate}
\item An object $\mcal{A} \in \opn{Ob}(\twocat{P}(U))$, for some open set $U$,
is
called a {\em local object} of $\twocat{P}$.

\item Let $\mcal{A} \in \opn{Ob}(\twocat{P}(U))$,  and let $V \subset U$. 
We write $\mcal{A}|_V := \mrm{rest}^{}_{V / U} (\mcal{A})$.

\item Let $\mcal{A} \in \opn{Ob}(\twocat{P}(U))$ and
$\mcal{A}' \in \opn{Ob}(\twocat{P}(U'))$ 
be local objects of $\twocat{P}$, and let $V \subset U \cap U'$.
An element $g \in \twocat{P}_1(V)(\mcal{A}|_V, \mcal{A}'|_V)$
is called a {\em local $1$-morphism} of $\twocat{P}$; and 
an element $a \in \twocat{P}_2(V)(\mcal{A}|_V)$
is called a {\em local $2$-morphism} of $\twocat{P}$.

\item For a local object $\mcal{A}$ of $\twocat{P}$ we denote by
$\opn{IG}(\mcal{A})$ the sheaf of groups $\twocat{P}_2(\mcal{A})$
from condition (iii) in Definition \ref{dfn:200};
and we denote by 
$\opn{ig} : \opn{IG}(\mcal{A}) \to \twocat{P}_1(\mcal{A})$ 
the homomorphism of sheaves of groups induces by the feedbacks $\opn{D}$. 
\end{enumerate}
\end{dfn}

We denote by $\twocat{Grp}(X)$ the {\em stack of sheaves of groups} on $X$. 
It is the stack of categories that assigns
to each open set $U$ the category 
$\twocat{Grp}(X)(U) := \cat{Grp}(U)$
of sheaves of small groups on $U$. The restriction morphisms here are
restrictions of sheaves.

For any stack of groupoids $\twocat{P}$ on $X$ there is a functor of stacks 
$\opn{Aut}_{\twocat{P}} : \twocat{P} \to \twocat{Grp}(X)$,
which is the geometrization of Definition \ref{dfn:204}(2). 
Like Proposition \ref{prop:200} we have:

\begin{prop} \label{prop:201}
Let $\twocat{P}$ be a stack of crossed groupoids on $X$.
\begin{enumerate}
\item $\opn{IG} : \twocat{P}_1 \to \twocat{Grp}(X)$ is a functor of stacks 
on $X$.

\item $\opn{ig} : \opn{IG} \twoto \opn{Aut}_{\twocat{P}_1}$
is a natural transformation between functors of stacks
$\twocat{P}_1 \to \twocat{Grp}(X)$.

\item The data 
$( \twocat{P}_1, \twocat{P}_2, \opn{Ad}_{\twocat{P}_1 \crvar \twocat{P}_{2}},
\opn{D} )$
can be recovered from the stack of groupoids $\twocat{P}_1$, the functor of
stacks 
$\opn{IG} : \twocat{P}_1 \to \twocat{Grp}(X)$ and the natural transformation 
$\opn{ig} : \opn{IG} \twoto \opn{Aut}_{\twocat{P}_1}$.
\end{enumerate}
\end{prop}

\begin{proof}
Again this is immediate from the definitions, but here the details are of
course much more complicated than in Proposition \ref{prop:200}.
It is good to look up \cite[Section 2]{Ye4} for detailed definitions and
notation. 
\end{proof}

\begin{exa}
This is the geometric version of Example \ref{exa:200}.
Suppose $\gerbe{G}$ is a gerbe on $X$, and $\gerbe{N}$ is a normal subgroupoid
on $\gerbe{G}$, in the sense of \cite[Definition 3.6]{Ye4}.
Then there is a stack of crossed groupoids $\twocat{P}$ on $X$, with 
$\twocat{P}_1 := \gerbe{G}$, $\twocat{P}_2 := \gerbe{N}$,
$\opn{Ad}_{\twocat{P}_1 \crvar \twocat{P}_{2}} :=
{\opn{Ad}_{\gerbe{G}}}|_{\gerbe{N}}$,
and $\opn{D} : \twocat{P}_1 \to \twocat{P}_{2}$ is the inclusion. 
\end{exa}

\begin{exa} \label{exa:9}
This is the geometric version of Example \ref{exa:15}.
On any open set $U \subset X$ we have the crossed groupoid 
$\cat{Grp}^{\times}(U)$, in which the objects are the sheaves of groups
$\mcal{G}$ on $U$; the $1$-morphisms are the isomorphisms of sheaves of groups;
and the $2$-morphisms are the elements of $\Gamma(U, \mcal{G})$. 
The twisting and feedback are as in Example \ref{exa:15}.
As $U$ varies we obtain a stack of crossed groupoids
$\twocat{Grp}^{\times}(X)$.
\end{exa}

\begin{exa} \label{exa:10}
This is the geometric version of Example \ref{exa:6}.
Take a commutative ring $R$. For an open set
$U \subset X$ denote by $\cat{Assoc}(R, U)$ the category 
of sheaves of associative $R$-algebras on $U$. As $U$ varies we get 
$\twocat{Assoc}(R, X)$, the {\em stack of sheaves of associative $R$-algebras on
$X$}. Inside $\twocat{Assoc}(R, X)$ we have the stack of groupoids 
$\twocat{Assoc}^{\times}(R, U)$, in which we take only the invertible
morphisms. The latter can be made into a stack of crossed groupoids by taking
$\opn{IG}(\mcal{A}) := \mcal{A}^{\times}$, the sheaf of invertible elements. 
\end{exa}

\section{Deformations of Sheaves of Rings}
\label{sec:defs-alg-vars}

In this section we recall material from the companion papers
\cite{Ye7} and \cite{Ye9}.
For the rest of the paper we work over a base field $\K$ of characteristic
$0$. All algebras are by default $\K$-algebras, and all homomorphisms are by
default over $\K$. Further conventions are that associative algebras are
unital, and commutative algebras are associative (and unital). 
For $\K$-modules $M, N$ we write $M \ot N := M \ot_{\K} N$  and
$\opn{Hom}(M, N) := \opn{Hom}_{\K}(M, N)$.

\begin{dfn} \label{dfn:13}
A {\em parameter $\K$-algebra} is a 
complete local noetherian commutative
$\K$-algebra $R$, with maximal ideal $\m$ and residue field 
$R / \m = \K$.
We sometimes say that $(R, \m)$ is a parameter $\K$-algebra.
For $i \geq 0$ we let
$R_i := R / \m^{i+1}$.
The $\K$-algebra homomorphism $R \to \K$ is called the 
augmentation of $R$.

By homomorphism of parameter algebras $R \to R'$ we mean a $\K$-algebra
homomorphism (it is automatically a local homomorphism).
\end{dfn}

Note that $R$ can be recovered from $\m$, since
$R = \K \oplus \m$ as $\K$-modules, with the obvious multiplication. 

\begin{exa}
The most important parameter algebra in deformation theory is
$\K[[\hbar]]$, the ring of formal power series in the variable
$\hbar$. A $\K[[\hbar]]$-deformation (see below) is sometimes called
a ``$1$-parameter formal deformation''. 
\end{exa}

\begin{rem}
Some of the definitions and results in the paper hold for a field $\K$ of
arbitrary characteristic.
(A few of them hold even for any commutative ring $\K$ and any 
commutative $\K$-algebra $R$.) A rule of thumb is that
whenever exponential functions occur, we need characteristic $0$. For the sake
of clarity we restrict attention to a field of characteristic $0$ and parameter
algebras over it.
\end{rem}

Let $(R, \m)$ be a parameter algebra, and let $M$ be an $R$-module. 
The $\m$-adic completion of $M$ is the $R$-module
$\what{M} :=  \lim_{\leftarrow i}\, (R_i \otimes_R M)$. 
The module $M$ is called {\em $\m$-adically complete} if the
canonical homomorphism
$M \to \what{M}$ is bijective. 
Since $R$ is noetherian, the $\m$-adic completion of any $R$-module is 
$\m$-adically complete; see \cite[Corollary 3.5]{Ye5}. (This may be false when
$R$ is not noetherian.) 

Given a $\K$-module $V$ and an $R$-module $M$, we let
$M \hot V := \what{M \ot V}$, the $\m$-adic completion of the $R$-module 
$M \ot V$.  

Let $X$ be a topological space. Recall that a sheaf $\mcal{M}$ 
of $R$-modules on $X$ is called {\em flat} if for every point 
$x \in X$ the stalk $\mcal{M}_x$ is a flat $R$-module. 
The {\em $\m$-adic completion} of $\mcal{M}$ is the sheaf 
$\what{\mcal{M}} := \lim_{\leftarrow i}\, (R_i \otimes_R \mcal{M})$.
The sheaf $\mcal{M}$ is called {\em $\m$-adically complete} if the
canonical sheaf homomorphism 
$\mcal{M} \to \what{\mcal{M}}$ is an isomorphism.

We need some properties of sheaves on $X$.
Suppose $\bsym{U} = \{ U_k \}_{k \in K}$ is a collection of
open sets in $X$. For $k_0, \ldots, k_m \in K$ we write
$U_{k_0, \ldots, k_m} := U_{k_0} \cap \cdots \cap U_{k_m}$.

\begin{dfn} \label{dfn:23}
Let $\mcal{N}$ be a sheaf of abelian groups on a topological space
$X$. 
\begin{enumerate}
\item An open set $U \subset X$ will be called {\em
$\mcal{N}$-acyclic} if the derived functor sheaf cohomology satisfies
$\mrm{H}^i (U, \mcal{N}) = 0$ for all $i > 0$.
\item Now suppose $\bsym{U} = \{ U_k \}_{k \in K}$ is a collection of
open sets in $X$. We say that the collection $\bsym{U}$ is
{\em $\mcal{N}$-acyclic} if all the finite intersections
$U_{k_0, \ldots, k_m}$ are $\mcal{N}$-acyclic.
\item We say that {\em there are enough $\mcal{N}$-acyclic open
sets} if for any open set $U \subset X$, and any open covering
$\bsym{U}$ of $U$, there exists an $\mcal{N}$-acyclic open covering
$\bsym{U}'$ of $U$ which refines $\bsym{U}$.
\end{enumerate}
\end{dfn}

\begin{exa} \label{exa:20}
Here are a few typical examples of a topological space $X$, and a
sheaf $\mcal{N}$, such that there are enough $\mcal{N}$-acyclic 
open sets.
\begin{enumerate}
\item $X$ is an algebraic variety over the field $\K$
(i.e.\ an integral finite type separated  $\K$-scheme), with structure
sheaf $\mcal{O}_X$, and $\mcal{N}$ is a coherent $\mcal{O}_X$-module.
Then any affine open set $U$ is $\mcal{N}$-acyclic.
\item $X$ is a complex analytic manifold, with structure sheaf 
$\mcal{O}_X$, and $\mcal{N}$ is a coherent $\mcal{O}_X$-module. 
Here we take $\K = \mbb{C}$ of course. Then
any Stein open set $U$ is $\mcal{N}$-acyclic.
\item $X$ is a differentiable manifold, with structure
sheaf $\mcal{O}_X$, and $\mcal{N}$ is any $\mcal{O}_X$-module. 
The field is $\K = \mbb{R}$. Then
any open set $U$ is $\mcal{N}$-acyclic.
\item $X$ is a differentiable manifold, and $\mcal{N}$ is a locally
constant sheaf of abelian groups. Then any sufficiently small
simply connected open set $U$ is $\mcal{N}$-acyclic. 
\end{enumerate}
\end{exa}

In order to get a good behavior of $\m$-adically complete sheaves, we shall
consider deformations in the following setup:

\begin{setup} \label{setup:200}
$\K$ is a field of characteristic $0$;
$(R, \m)$ is a parameter $\K$-algebra; $X$ is a topological
space; and $\mcal{O}_X$ is a sheaf of commutative $\K$-algebras on $X$. 
Moreover $X$ has enough $\mcal{O}_X$-acyclic open sets.
\end{setup}

As Example \ref{exa:20} shows this covers many instances.
We assume Setup \ref{setup:200} for the rest of this section.

\begin{dfn} \label{dfn:230}
An {\em associative $R$-deformation of $\mcal{O}_X$} is a sheaf
$\mcal{A}$ of flat $\m$-adically complete associative
$R$-algebras on $X$, together with an isomorphism of
sheaves of $\K$-algebras
$\psi : \K \otimes_R \mcal{A} \to \mcal{O}_X$, called an {\em augmentation}.

Let $\mcal{A}$ and $\mcal{A}'$ be associative $R$-deformations of
$\mcal{O}_X$. A {\em gauge transformation}
$g : \mcal{A} \to \mcal{A}'$
is an isomorphism of sheaves of $R$-algebras that commutes with
the augmentations to $\mcal{O}_X$.
\end{dfn}

Similarly, given a commutative $\K$-algebra $C$, an {\em associative
$R$-deformation of $C$} is a flat $\m$-adically complete associative
$R$-algebra $A$, together with an isomorphism of $\K$-algebras
$\K \otimes_R A \cong C$. 

Observe that we do not impose a condition that an associative deformation $\AA$
should be ``locally differential''; see Remarks \ref{rem:defs-sh.110} and
\ref{rem:221} for a discussion. 

Recall that a Poisson bracket on a commutative $R$-algebra $A$ is 
an $R$-bilinear pairing 
$\{ -,- \} : A \times A \to A$ 
which is a Lie bracket, and is a biderivation (a derivation in each variable). 
The pair $(A, \{ -,- \} )$ is called a {\em Poisson $R$-algebra}. 
A homomorphism of Poisson $R$-algebras $A \to A'$ is an algebra homomorphism
that respects the Poisson brackets. All of this makes sense also for sheaves of
$R$-algebras. 

\begin{dfn} \label{dfn:231}
We view $\mcal{O}_X$ as a sheaf of Poisson $\K$-algebras
with the zero bracket.
A {\em Poisson $R$-deformation of $\mcal{O}_X$} is a sheaf
$\mcal{A}$ of flat $\m$-adically complete Poisson
$R$-algebras on $X$, together with an isomorphism of
sheaves of Poisson $\K$-algebras
$\psi : \K \otimes_R \mcal{A} \to \mcal{O}_X$, called an {\em augmentation}.

 Let $\mcal{A}$ and $\mcal{A}'$ be Poisson $R$-deformations of
$\mcal{O}_X$. A {\em gauge transformation}
$g : \mcal{A} \to \mcal{A}'$
is an isomorphism of sheaves of Poisson $R$-algebras that commutes with
the augmentations to $\mcal{O}_X$.
\end{dfn}

Similarly, given a commutative $\K$-algebra $C$, a {\em Poisson $R$-deformation
of $C$} is a flat $\m$-adically complete Poisson $R$-algebra $A$, together with
an isomorphism of Poisson $\K$-algebras $\K \otimes_R \mcal{A} \cong A$.

According to \cite[Proposition 4.5]{Ye9}, if $\mcal{A}$ is an associative
(resp.\ Poisson) $R$-deformation of $\OX$, and $U \subset X$ is an $\OX$-acyclic
open set, then $A := \Gamma(U, \mcal{A})$ is an associative
(resp.\ Poisson) $R$-deformation of $C :=  \Gamma(U, \OX)$.
The important fact is that $A$ is a flat and $\m$-adically complete
$R$-module.

Let $\mcal{A}$ be an associative (resp.\ Poisson) $R$-deformation of $\OX$.
It has an $R$-bilinear Lie bracket $[-,-]$ on it. In the associative case let
$\star$ denote the multiplication. Given local sections $a_1,  a_2 \in
\mcal{A}$, their bracket is 
$[a_1,  a_2] := a_1 \star a_2 - a_2 \star a_1$. 
In the Poisson case we take the Poisson bracket:
$[a_1,  a_2] := \{ a_1,  a_2 \}$. 
Note that $[\mcal{A}, \mcal{A}] \subset \m \AA$.
According to \cite[Corollary 3.8]{Ye9}, for every $j \geq 0$ the sheaf 
$\m^j \mcal{A}$ is $\m$-adically complete. Thus we obtain a sheaf of
pronilpotent groups
\begin{equation} \label{eqn:350}
\opn{Exp}(\m^j \mcal{A}) := 
\lim_{\leftarrow i}\, \opn{Exp}(\m^j \mcal{A} / \m^{j+i} \mcal{A}) ,
\end{equation}
and an isomorphism of sheaves of sets
\begin{equation} \label{eqn:351}
\exp : \m^j \AA \iso \opn{Exp}(\m^j \mcal{A}) .
\end{equation}
In particular we have a sheaf of pronilpotent groups
\begin{equation} \label{eqn:230}
\opn{IG}(\mcal{A}) := \opn{Exp}(\m \mcal{A})  
\end{equation}
that we call the {\em sheaf of inner gauge groups} of $\AA$.
A gauge transformation $g : \mcal{A} \to \mcal{A}'$ induces a group isomorphism 
\begin{equation} \label{eqn:234}
\opn{IG}(g) : \opn{IG}(\mcal{A}) \to \opn{IG}(\mcal{A}') .
\end{equation}

Consider a local section $\al \in \m \mcal{A}$.
Let $\opn{ad}_{\mcal{A}}(\al) : \mcal{A} \to \mcal{A}$ be the operator
$\opn{ad}_{\mcal{A}}(\al)(a') := [\al, a']$,
where $[-,-]$ is the Lie bracket. Then to the group element 
$a := \exp(\al) \in \opn{IG}(\mcal{A})$ we can associate the operator 
\begin{equation} \label{eqn:231}
\opn{ig}(a) := \exp( \opn{ad}_{\mcal{A}}(\al) ) 
\end{equation}
on $\mcal{A}$. It turns out that $\opn{ad}_{\mcal{A}}(\al)$ is a derivation of
the algebra $\mcal{A}$, and hence $\opn{ig}(a)$ is a gauge transformation of
$\mcal{A}$. In the associative case the group $\opn{IG}(\mcal{A})$ is
canonically identified with the kernel of the augmentation homomorphism 
$\psi : \mcal{A}^{\times} \to \mcal{O}_X^{\times}$,
and then $\opn{ig}(a)$ becomes conjugation by the invertible element $a$.
In the Poisson case we call $\opn{ig}(a)$ a 
{\em formal hamiltonian flow}.
See \cite[Section 5]{Ye9} for details. 

\begin{dfn} \label{dfn:233}
We denote by $\cat{AssDef}(R, \mcal{O}_X)$ (resp.\ 
$\cat{PoisDef}(R, \mcal{O}_X)$) the crossed groupoid 
$\cat{P} = (\cat{P}_1, \cat{P}_2, 
\opn{Ad}_{\cat{P}_1 \crvar \cat{P}_{2}}, \opn{D})$
described as follows:
\begin{itemize}
\item The objects of $\cat{P}$ are the associative (resp.\ Poisson)
$R$-deformations $\mcal{A}$ of $\mcal{O}_X$.

\item The $1$-morphisms, i.e.\ the elements of the sets
$\cat{P}_1(\mcal{A}, \mcal{A}')$, are the gauge transformations 
$g : \mcal{A} \to \mcal{A}'$.

\item The groups of $2$-morphisms are
$\cat{P}_2(\mcal{A}) := \Gamma(X, \opn{IG}(\mcal{A}))$.

\item The twisting by a $1$-morphism $g$ is 
$\opn{Ad}_{\cat{P}_1 \crvar \cat{P}_{2}}(g) := \opn{IG}(g)$.

\item The feedback 
$\opn{D} : \cat{P}_2(\mcal{A}) \to \cat{P}_1(\mcal{A})$ is
$\opn{D}(a) := \opn{ig}(a)$.
\end{itemize}
\end{dfn}

According to \cite[Proposition 5.7]{Ye9} these are indeed crossed groupoids.

In the next two propositions we write $\cat{P}(R, X)$ for either 
$\cat{AssDef}(R, \mcal{O}_X)$ or $\cat{PoisDef}(R, \mcal{O}_X)$; etc.

\begin{prop} \label{prop:206}
Let $U' \subset U$ be open subsets of $X$, and $R \to R'$ a homomorphism of
parameter algebras.  There is a morphism of crossed groupoids 
$\cat{P}(R, U) \to \cat{P}(R', U')$,
sending $\mcal{A}$ to $(R' \hatotimes{R} \mcal{A})|_{U'}$. 
\end{prop}

\begin{proof}
For $R \to R'$ this is \cite[Proposition 5.8]{Ye9}. That 
$\mcal{A}|_{U'}$ is a deformation of $\mcal{O}_{U'}$ is trivial. 
\end{proof}

\begin{prop} \label{prop:210}
The assignment 
$U \mapsto \cat{P}(R, U)$
is a stack of crossed groupoids on $X$. 
\end{prop}

\begin{proof}
By Proposition \ref{prop:206} we have a morphism of crossed groupoids
\[ \opn{rest}_{V / U} : \cat{P}(R, \mcal{O}_U) \to \cat{P}(R, \mcal{O}_V) \]
for every inclusion $V \subset U$ of open sets. So the local objects and local
$1$-morphisms form a split prestack of groupoids, say $\twocat{P}_1$. Since the
property of being an $R$-deformation of $\OX$ is local, and gauge
transformations between deformations are local, it follows that $\twocat{P}_1$
is a stack.
Finally, we know that the local $2$-morphisms form sheaves of groups
(these are the sheaves $\opn{IG}(\mcal{A})$). 
\end{proof}

\begin{dfn} \label{dfn:210}
\begin{enumerate}
\item The stack of crossed groupoids
$U \mapsto \cat{AssDef}(R, \mcal{O}_U)$ 
is called the {\em stack of associative $R$-deformations of $\OX$},
and it is denoted by $\twocat{AssDef}(R, \mcal{O}_X)$.

\item The stack of crossed groupoids
$U \mapsto \cat{PoisDef}(R, \mcal{O}_U)$ 
is called the {\em stack of Poisson $R$-deformations of $\OX$},
and it is denoted by $\twocat{PoisDef}(R, \mcal{O}_X)$.
\end{enumerate}
\end{dfn}

We now discuss first order brackets.
Suppose $\mcal{A}$ is an $R$-deformation of $\mcal{O}_X$. Since
$\mcal{A}$ is flat over $R$, and we have the augmentation 
$\psi : \K \otimes_R \mcal{A} \iso \mcal{O}_X$,
there is an induced $\K$-linear isomorphism 
$\m \mcal{A} / \m^2 \mcal{A} \cong
(\m / \m^2) \ot  \mcal{O}_X$.
This gives rise to a homomorphism of sheaves
$\psi^+ : \m \mcal{A} \to (\m / \m^2) \ot  \mcal{O}_X$.

Suppose $\mcal{A}$ is an associative deformation, with multiplication
$\star$. Given local sections $a_1,  a_2 \in \mcal{A}$, the
commutator satisfies
$a_1 \star a_2 - a_2 \star a_1 \in \m \mcal{A}$.
Hence we get
\begin{equation} \label{eqn:400}
\psi^+(a_1 \star a_2 - a_2 \star a_1) \in 
(\m / \m^2) \ot  \mcal{O}_X . 
\end{equation}
Likewise if $\mcal{A}$ is a Poisson deformation, with Poisson bracket
$\{ -,- \}$, then we have
$\psi^+(\{ a_1, a_2 \}) \in (\m / \m^2) \ot  \mcal{O}_X$.

\begin{prop} 
Let $\mcal{A}$ be associative \tup{(}resp.\ Poisson\tup{)}
$R$-deformation of $\mcal{O}_X$. 
There is a unique $\K$-bilinear sheaf morphism
\[ \{ -,- \}_{\mcal{A}} : \mcal{O}_X \times \mcal{O}_X \to 
(\m / \m^2) \ot \mcal{O}_X  \]
having the following property.
Given local sections $c_1, c_2 \in \mcal{O}_X$, choose local liftings 
$a_1,  a_2 \in \mcal{A}$ relative to the augmentation
$\psi : \mcal{A} \to \mcal{O}_X$, 
i.e.\ $c_i = \psi(a_i)$. Then 
\[ \{ c_1, c_2 \}_{\mcal{A}} = \psi^+ 
\bigl( \smfrac{1}{2} (a_1 \star a_2 - a_2 \star a_1) \bigr) \]
or
\[ \{ c_1, c_2 \}_{\mcal{A}} = \psi^+(\{ a_1, a_2 \}) , \]
as the case may be.
\end{prop}

\begin{proof}
This is a variant of the usual calculation in deformation theory. The
only thing to notice is that it makes sense for sheaves.
\end{proof}

\begin{dfn} \label{dfn:20}
The {\em first order bracket} of $\mcal{A}$ is the 
$\K$-bilinear sheaf morphism
$\{ -,- \}_{\mcal{A}}$ in the proposition above. 
\end{dfn}

\begin{prop} \label{prop:20}
\begin{enumerate}
\item The first order bracket is gauge invariant. Namely if 
$\mcal{A}$ and $\mcal{B}$ are gauge equivalent
$R$-deformations of $\mcal{O}_X$, then 
$\{ -,- \}_{\mcal{A}} = \{ -,- \}_{\mcal{B}}$.
\item The bracket $\{ -,- \}_{\mcal{A}}$ is a biderivation of
$\mcal{O}_X$-modules.
\item Suppose $R = \K[[\hbar]]$. Using the isomorphism
$\hbar^{-1} : \m / \m^2 \iso \K$, we get a bilinear function
\[ \{ -,- \}_{\mcal{A}} : \mcal{O}_X \times \mcal{O}_X \to \mcal{O}_X
. \]
Then this is a Poisson bracket on $\mcal{O}_X$.
\end{enumerate}
\end{prop}

\begin{proof}
All these statements are easy local calculations. 
\end{proof}

\section{DG Lie Algebras and Deformations of Affine Varieties}
\label{sec:defs-DG-Lie}

We now recall some facts on DG Lie algebras from \cite{Ye7, Ye9}. 

\begin{dfn} \label{dfn:220}
\begin{enumerate}
\item A DG Lie algebra 
$\g = \bigoplus_{i \in \mbb{Z}} \mfrak{g}^{i}$
is said to be of {\em quantum type} if $\g^i = 0$ for $i < -1$.

\item A {\em quasi quantum type DG Lie algebra} is a DG Lie algebra 
$\til{\mfrak{g}}$ that admits a DG Lie quasi-isomorphism
$\til{\g} \to \g$ to some quantum type DG Lie algebra $\g$.
\end{enumerate}
\end{dfn}

Fix a quasi quantum type DG Lie $\K$-algebra $\g$, and a parameter $\K$-algebra 
$(R, \m)$. There is an induced Lie $R$-algebra 
$\m \hot \g = \bigoplus_i \m \hot \g^i$.
The Lie bracket is $[-,-]$ and the differential is $\d$.
A solution $\om \in \m \hot \g^1$ of the Maurer-Cartan equation 
$\d(\om) + \smfrac{1}{2} [\om, \om] = 0$
is called an {\em MC element}, and the set of such elements is denoted by
$\opn{MC}(\m \hot \g)$.
The group $\opn{Exp}(\m \hot \g^0)$ is called the {\em gauge group}.
For $\ga \in \m \hot \g^0$ there is an affine action $\opn{af}(\ga)$ on 
the $R$-module $\m \hot \g^1$, 
namely $\opn{af}(\ga)(\om) := [\ga ,\om] - \d(\ga)$.
So for 
$g := \exp(\ga) \in \opn{Exp}(\m \hot \g^0)$
there is an affine automorphism 
$\opn{Af}(g) := \exp(\opn{af}(\ga))$ of $\m \hot \g^1$.
It is known that $\opn{Af}(g)$ preserves the set $\opn{MC}(\m \hot \g)$.
See \cite{GM}, \cite{CKTB}, \cite[Section 3]{Ye1} or \cite[Section 1]{Ye7} for
details.

Given $\om \in \opn{MC}(\m \hot \g)$ there is a square zero derivation
$\d_{\om} :=  \d + \opn{ad}(\om)$ of degree $1$ on $\m \hot \g$, and a bracket 
$[\al_1, \al_2]_{\om} := [ \d_{\om}(\al_1), \al_2 ]$
for $\al_1, \al_2 \in \m \hot \g^{-1}$ (that's not a Lie bracket in general).
Define the $R$-module
\begin{equation} \label{eqn:Lie-desc.110}
\a_{\om} := \opn{Coker}(\d_{\om} : \m \hot \g^{-2} \to 
\m \hot \g^{-1}) .
\end{equation}
By \cite[Proposition 6.6]{Ye7}, $\a_{\om}$ is an $R$-linear Lie algebra, with
bracket induced by  $[-,-]_{\om}$. And by \cite[Proposition 6.7]{Ye7},
$\a_{\om}$ is an $\m$-adically complete $R$-module (this is where we need $\g$
to be quasi quantum type!). So $\a_{\om}$ is a pronilpotent Lie algebra, and
there is an associated pronilpotent group $N_{\om} := \opn{Exp}(\a_{\om})$. 

The $R$-linear function 
$\d_{\om} : \a_{\om} \to \m \hot \g^0$
is a Lie algebra homomorphism, so it induces a group homomorphism 
\begin{equation} \label{eqn:Lie-desc.103}
\opn{D}_{\om} : N_{\om} \to \opn{Exp}(\m \hot \g^0) \ , \
\opn{D}_{\om} := \exp(\d_{\om}) .
\end{equation}
A short calculation shows that 
$\opn{af}(\d_{\om}(\al))(\om) = 0$
for any $\al \in \m \hot \g^{-1}$. The exponential of this formula is 
$\opn{Af}(\opn{D}_{\om}(a))(\om) = \om$ for any $a \in N_{\om}$.
We see that 
$\opn{D}_{\om}(a) \in \opn{Exp}(\m \hot \g^0)(\om)$,
where the latter is the stabilizer of $\om$ in the group $\opn{Exp}(\m \hot
\g^0)$.  

Now take $g \in \opn{Exp}(\m \hot \g^0)$, and let 
$\om' := \opn{Af}(g)(\om) \in \opn{MC}(\m \hot \g)$.
According to \cite[Corollary 6.9]{Ye7} there is a group isomorphism 
\begin{equation} \label{eqn:Lie-desc.104}
\opn{Ad}(g) : N_{\om} \iso N_{\om'} ,
\end{equation}
which is functorial in $g$, and the diagram 
\[ \UseTips \xymatrix @C=5ex @R=5ex {
N_{\om}
\ar[d]_{\opn{Ad}(g)}
\ar[r]^(0.34){\opn{D}_{\om}}
&
\opn{Exp}(\m \hot \g^0)
\ar[d]^{\opn{Ad}(g)}
\\
N_{\om'}
\ar[r]^(0.34){\opn{D}_{\om'}}
&
\opn{Exp}(\m \hot \g^0)
} \]
is commutative. By definition of the bracket $[-,-]_{\om}$, the adjoint
action in the Lie algebra $\a_{\om}$ is
$\opn{ad}_{\a_{\om}}(\al_1)(\al_2) = \opn{ad}(\d_{\om}(\al_1))(\al_2)$; 
hence, by exponentiating this equation, we see that conjugation in the group 
$N_{\om}$ is  
$\opn{Ad}_{N_{\om}}(a_1)(a_2) = \opn{Ad}(\opn{D}_{\om}(a_1))(a_2)$
for $a_i \in N_{\om}$.

Crossed groupoids were introduced in Definition \ref{dfn:cosim.101}.
The considerations above justify the next definition. 

\begin{dfn} \label{dfn:Lie-desc.101}
Let $\g$ be a quasi quantum type DG Lie algebra, and let $(R, \m)$ be a
parameter algebra, both over the field $\K$. The {\em Deligne crossed groupoid} 
is the crossed groupoid $\opn{Del}(\g, R)$ with these
components:
\begin{itemize}
\item The groupoid $\opn{Del}_1(\g, R)$ is the transformation groupoid 
associated to the action $\opn{Af}$ of the group $\opn{Exp}(\m \hot \g^0)$
on the set $\opn{MC}(\m \hot \g)$.
(This is the usual Deligne groupoid of $\m \hot \g$.)

\item  The groupoid $\opn{Del}_2(\g, R)$ is the totally disconnected
groupoid with set of objects $\opn{MC}(\m \hot \g)$.
For $\om \in \opn{MC}(\m \hot \g)$ its automorphism group 
in  $\opn{Del}_2(\g, R)$ is 
$N_{\om} = \opn{Exp}(\a_{\om})$, where $\a_{\om}$ is the pronilpotent Lie
algebra in
formula (\ref{eqn:Lie-desc.110}).

\item The twisting 
$\opn{Ad}_{\opn{Del}_1 \crvar \opn{Del}_2}(g)$,
for $g \in  \opn{Exp}(\m \hot \g^0)$, is the group isomorphism
$\opn{Ad}(g)$ in formula (\ref{eqn:Lie-desc.104}).

\item The feedback $\opn{D}$ is the group homomorphism in formula 
(\ref{eqn:Lie-desc.103}).
\end{itemize}
\end{dfn}

It is obvious from the construction that the crossed groupoid
 $\opn{Del}(\g, R)$ is functorial in both $\g$ and $R$. 
When $\g$ is a quantum type DG Lie algebra then 
$\a_{\om} = (\m \hot \g^{-1})_{\om}$ of course.
If $\g$ is of quantum type and $R$ is artinian, then $\opn{Del}(\g, R)$
is precisely the $2$-groupoid introduced by Deligne in \cite{De}; see also 
\cite{Ge1}.

Let $X$ be a smooth algebraic variety over $\K$. 
There are two important sheaves of DG Lie algebras on $X$: the sheaf
$\mcal{T}_{\mrm{poly}, X}$ of {\em polyderivations}, and the
sheaf $\mcal{D}_{\mrm{poly}, X}$ of {\em polydifferential operators}. 
Within $\mcal{D}_{\mrm{poly}, X}$ there is the subalgebra
$\mcal{D}_{\mrm{poly}, X}^{\mrm{nor}}$ 
of {\em normalized polydifferential operators}. 
The sheaves $\mcal{T}_{\mrm{poly}, X}$, $\mcal{D}_{\mrm{poly}, X}$ and
$\mcal{D}_{\mrm{poly}, X}^{\mrm{nor}}$
are quasi-coherent graded $\OX$-modules, and the inclusion 
$\mcal{D}_{\mrm{poly}, X}^{\mrm{nor}} \to \mcal{D}_{\mrm{poly}, X}$
is a quasi-isomorphism. 
Full details on these DG Lie algebras can be found in \cite{Ye1, Ye2}. 
For any open set $U \subset X$ we get quantum type DG Lie algebras
$\Gamma(U, \mcal{D}_{\mrm{poly}, X}^{\mrm{nor}})$
and 
$\Gamma(U, \mcal{T}_{\mrm{poly}, X})$,
so Deligne crossed groupoids exist. 

Here is how the DG Lie algebras $\mcal{T}_{\mrm{poly}, X}$ and 
$\mcal{D}_{\mrm{poly}, X}^{\mrm{nor}}$ control deformations of 
$\OX$. 

\begin{thm}[{Geometrization, \cite[Theorem 0.1]{Ye9}}] \label{thm:205}
Let $\K$ be a field of characteristic $0$, $X$ a smooth algebraic variety
over $\K$, $R$ a parameter algebra over $\K$, 
and $U$ an affine open set in $X$.
There are equivalences of crossed groupoids
\[  \opn{geo} : \
\opn{Del} \bigl( \Gamma(U, \mcal{D}_{\mrm{poly}, X}^{\mrm{nor}}) , R \big)
\to \cat{AssDef}(R, \mcal{O}_U) \]
and
\[  \opn{geo} : \
\opn{Del} \bigl( \Gamma(U, \mcal{T}_{\mrm{poly}, X}) , R \big)
\to \cat{PoisDef}(R, \mcal{O}_U) \]
which we call {\em geometrization}. The equivalences $\opn{geo}$ 
commute with homomorphisms $R \to R'$ of parameter algebras, and with inclusions
of affine open sets $U' \to U$.
\end{thm}

The dependence of $\cat{AssDef}(R, \mcal{O}_U)$ and 
$\cat{PoisDef}(R, \mcal{O}_U)$ on $R$ and $U$ is shown in Proposition
\ref{prop:206}.  

\begin{rem} \label{rem:defs-sh.110}
Suppose $R = \K[[\hbar]]$. In our earlier
paper \cite{Ye1} we referred to an associative $R$-deformation of 
$\mcal{O}_X$ as a ``deformation quantization of $\mcal{O}_X$''. In
retrospect this name seems inappropriate,
and hence the new name used here.

Another, more substantial, change is that in 
\cite[Definition 1.6]{Ye1}
we required that the associative deformation $\mcal{A}$
shall be endowed with a differential structure. This turns out to be
redundant -- indeed, Theorem \ref{thm:205} implies that any
associative $R$-deformation $\mcal{A}$ of $\mcal{O}_X$ admits a differential
structure. Moreover, any two such differential structures are equivalent.
\end{rem}

\begin{rem} \label{rem:221}
In \cite{KS2} Kashiwara and Schapira consider associative
$\mbb{C}[[\hbar]]$-defor\-mations $\mcal{A}$ of $\OX$, for a complex manifold
$X$. Their definition is very similar to our Definition \ref{dfn:230} (for 
$R = \mbb{C}[[\hbar]]$), except that they do not require an augmentation 
$\mcal{A} \to \OX$.
\cite[Proposition 2.2.3]{KS2} is very similar to the fullness of the functor 
$\opn{geo}$ on $1$-morphisms (in the associative case of Theorem \ref{thm:205}
above). However it is not known whether $\opn{geo}$ is essentially surjective
on objects in the complex analytic setting -- this is  \cite[Remark 2.2.7]{KS2}.
\end{rem}

\section{Twisted Objects in a Stack of Crossed Groupoids}
\label{sec:tw.sh}

In this section we present some geometric constructions that generalize the
concepts of gerbe and stack of algebroids.

As a warmup we begin with the non-geometric version. 
Recall that a crossed groupoid 
$\cat{P} =
(\cat{P}_1, \cat{P}_2, \opn{Ad}_{\cat{P}_1 \crvar \cat{P}_{2}}, \opn{D})$
determines a functor 
$\opn{IG} : \cat{P}_1 \to \cat{Grp}$,
and a natural transformation 
$\opn{ig} : \opn{IG} \twoto \opn{Aut}_{\cat{P}_1}$
of functors $\cat{P}_1 \to \cat{Grp}$. 
Namely 
$\opn{IG}(A) = \cat{P}_2(A)$ for any object $A$ of $\cat{P}$,
and $\opn{ig} = \opn{D}$. See Definition \ref{dfn:cosim.101} and Proposition
\ref{prop:200}.

\begin{dfn} \label{dfn:6}
Let $\cat{P}$ be a crossed groupoid. A {\em twisted object of $\cat{P}$} is a
triple $(G, \bsym{A}, \opn{cp})$ consisting of:
\begin{enumerate}
\item A small connected nonempty groupoid $G$, called the {\em gauge groupoid}.
\item A functor
$\bsym{A} : G \to \cat{P}_1$, called the {\em representation}.
\item A natural isomorphism
$\opn{cp} : \opn{Aut}_{G} \twoiso \opn{IG} \circ \, \bsym{A}$
of functors $G \to \cat{Grp}$, called the  {\em
coupling isomorphism}.
\end{enumerate} 
The condition is: 
\begin{enumerate}
\rmitem{$*$} The diagram 
\[ \UseTips \xymatrix @C=7ex @R=7ex {
\opn{Aut}_{G}
\ar@{=>}[r]^{\opn{cp}}
\ar@{=>}[dr]_{\opn{Aut}_{\bsym{A}}}
& \opn{IG} \circ \, \bsym{A}
\ar@{=>}[d]^{\opn{ig} \circ \, \bsym{1}_{\bsym{A}}} \\
& \opn{Aut}_{\cat{P}_1} \circ \, \bsym{A}
} \]
of natural transformations
between functors $G \to \cat{Grp}$ is commutative. 
Here $\bsym{1}_{\bsym{A}}$ is the identity automorphism of the functor 
$\bsym{A}$.
\end{enumerate}
\end{dfn}

The set of twisted objects of $\cat{P}$ is
denoted by $\opn{TwOb}(\cat{P})$. 
We often refer to the twisted object 
$(G, \bsym{A}, \opn{cp})$ just as $\bsym{A}$.

The definition above is quite formal and maybe difficult to
understand. So here is what it really means. 
For any $i \in \opn{Ob} (G)$ there is an object
$A_i := \bsym{A}(i) \in \opn{Ob} (\cat{P})$. 
Thus we are given a collection 
$\{ A_i \}_{i \in \opn{Ob} (G)}$ of objects of the crossed groupoid $\cat{P}$.
For any arrow $g : i \to j$ in $G$ there is given an isomorphism 
$\bsym{A}(g) : A_i \to A_j$
in $\cat{P}_1$. This tells us how we may try to identify the objects
$A_i$ and $A_j$. 

For $i \in \opn{Ob} (G)$ there is given a group isomorphism (the coupling)
\[ \opn{cp} : G(i) = \opn{Aut}_{G}(i) \iso
(\opn{IG} \circ \, \bsym{A})(i) = \opn{IG}(A_i) . \]
It forces the groupoid $G$ to be comprised of inner gauge
groups. But now an element $g \in G(i)$ has two
possible actions on the object $A_i$: it can act as $\bsym{A}(g)$, or
it can act as $\opn{ig}(\opn{cp}(g))$. Condition ($*$) says that
these two actions coincide. 

\begin{exa} \label{exa:8}
Suppose $A$ is an object of $\cat{P}$. We can turn it into a
twisted object of $\cat{P}$ as follows. Let $G$ be the one
object groupoid, say $\opn{Ob}(G) := \{ 0 \}$, with
$G(0,0) := \opn{IG}(A)$. 
The functor $\bsym{A} : G \to \cat{P}_1$ is 
$\bsym{A}(0) := A$ and $\bsym{A}(g) := \opn{ig}(g)$. 
The coupling isomorphism $\opn{cp}$ is the identity of
$\opn{IG}(A)$. 
We refer to $\bsym{A}$ as the {\em twisted object generated by $A$}.
\end{exa}

A more interesting example is:

\begin{exa} \label{exa:7}
Suppose $\cat{A}$ is an $R$-linear algebroid; namely it is an $R$-linear
category, nonempty and connected by isomorphisms (see \cite{Ko2}). We are going
to turn it into a twisted object of $\cat{Assoc}(R)$, or more precisely of the
crossed groupoid $\cat{Assoc}^{\times}(R)$ from Example \ref{exa:6}.
Consider the nonempty connected groupoid $G := \cat{A}^{\times}$.
The functor $\bsym{A} : G \to \cat{Assoc}(R)$
is defined by 
$\bsym{A}(i) = \cat{A}(i,i) \in \opn{Ob}(\cat{Assoc}(R))$ 
for any $i \in \opn{Ob}(G) = \opn{Ob}(\cat{A})$,
and for any  
$g \in G(i, j) = \cat{A}^{\times}(i,j)$
the isomorphism 
$\bsym{A}(g) : \bsym{A}(i) \to \bsym{A}(j)$ is 
$\bsym{A}(g)(a) := g \circ a \circ g^{-1}$.
The coupling isomorphism
$\opn{cp} : G(i) \to \opn{IG}(\bsym{A}(i))$
is the identity automorphism of
$G(i) = \cat{A}(i,i)^{\times} = \opn{IG}(\bsym{A}(i))$.

It is easy to see that every twisted object of  $\cat{Assoc}^{\times}(R)$ arises
this way from an $R$-linear algebroid. Thus the two concepts are the same.
\end{exa}

\begin{dfn} \label{dfn:29}
Let $(G, \bsym{A}, \opn{cp})$ and 
$(G', \bsym{A}', \opn{cp}')$ be twisted objects in a crossed groupoid $\cat{P}$.
A {\em twisted gauge transformation}
\[ \Phi = (\Phi_{\mrm{gau}}, \Phi_{\mrm{rep}}) :
(G, \bsym{A}, \opn{cp}) \to (G', \bsym{A}', \opn{cp}') \]
consists of an equivalence (of groupoids)
$\Phi_{\mrm{gau}} : G \to G'$,
and a natural isomorphism 
$\Phi_{\mrm{rep}} : \bsym{A} \twoiso \bsym{A}' \circ \Phi_{\mrm{gau}}$
of functors $G \to \cat{P}_1$.
The condition is that the diagram 
\[ \UseTips \xymatrix @C=12ex @R=6ex {
\opn{Aut}_{G}
\ar@{=>}[r]^{\opn{cp}}
\ar@{=>}[d]_{\opn{Aut}_{\Phi_{\mrm{gau}}}}
& 
\opn{IG} \circ \, \bsym{A}
\ar@{=>}[d]^{\bsym{1}_{\opn{IG}} \circ \Phi_{\mrm{rep}}}
\\
\opn{Aut}_{G'} \circ \, \Phi_{\mrm{gau}}
\ar@{=>}[r]^(0.47){\opn{cp}' \circ\, \bsym{1}_{\Phi_{\mrm{gau}}}}
& 
\opn{IG} \circ \, \bsym{A}' \circ \Phi_{\mrm{gau}}
} \]
of natural transformations between functors $G \to \cat{Grp}$
is commutative.
\end{dfn}

Let us spell out what this definition means for an object
$i \in \opn{Ob} (G)$.
Let $i' := \Phi_{\mrm{gau}}(i) \in \opn{Ob} (G')$.
Then there is an isomorphism
$\Phi_{\mrm{rep}} : \bsym{A}(i) \iso \bsym{A}'(i')$
in $\cat{P}$, and the diagram
\[ \UseTips \xymatrix @C=6ex @R=5ex {
G(i)
\ar[r]^(0.4){\opn{cp}}
\ar[d]_{\Phi_{\mrm{gau}}}
& 
\opn{IG}(\bsym{A}(i))
\ar[d]^{\opn{IG}(\Phi_{\mrm{rep}})}
\\
G'(i')
\ar[r]^(0.4){\opn{cp}'}
& 
\opn{IG}(\bsym{A}'(i'))
} \]
in $\cat{Grp}$ is commutative. 

It is easy to see that twisted gauge transformations form an equivalence
relation. We refer to this equivalence relation a {\em twisted gauge
equivalence}, and we write
\[ \ol{\opn{TwOb}}(\cat{P}) :=
\frac{ \opn{TwOb}(\cat{P}) }
{\tup{ twisted gauge equivalence}}  . \]

\begin{rem}
If one examines things a little, it becomes evident that a
twisted object $(G, \bsym{A}, \opn{cp})$
in $\cat{P}$ is twisted gauge equivalent to the twisted object
generated by $\bsym{A}(i) \in \cat{P}$, for any $i \in \opn{Ob} (G)$ (as
in Example \ref{exa:8}). Thus the whole concept is quite uninteresting. 

However, in the geometric context, where the crossed groupoid $\cat{P}$ is
replaced by a stack of crossed groupoids $\twocat{P}$ on a topological space
$X$, the concept becomes interesting: {\em really twisted objects}
(Definition \ref{dfn:25}) appear.
\end{rem}

We are ready to pass to the geometric context.

\begin{dfn} \label{dfn:11}
Let $X$ be a topological space, and let
\[ \twocat{P} = ( \twocat{P}_1, \twocat{P}_2, 
\opn{Ad}_{\twocat{P}_1 \crvar \twocat{P}_{2}}, \opn{D} ) \]
be a stack of crossed groupoids on $X$ (Definition \ref{dfn:200}). 
A {\em twisted object of $\twocat{P}$} is a triple 
$(\gerbe{G}, \gerbe{A}, \opn{cp})$ consisting of:
\begin{enumerate}
\item A small gerbe $\gerbe{G}$ on X, called the
{\em gauge gerbe}.
\item A functor of stacks
$\gerbe{A} : \gerbe{G} \to \twocat{P}_1$, called the {\em representation}.
\item A natural isomorphism
$\opn{cp} : \opn{Aut}_{\gerbe{G}} \twoiso \opn{IG} \circ \, \gerbe{A}$
between functors of stacks $\gerbe{G} \to \twocat{Grp}(X)$, 
called the {\em coupling isomorphism}.
\end{enumerate} 
The condition is: 
\begin{enumerate}
\rmitem{$*$} The diagram 
\[ \UseTips \xymatrix @C=7ex @R=7ex {
\opn{Aut}_{\gerbe{G}}
\ar@{=>}[r]^{\opn{cp}}
\ar@{=>}[dr]_{\opn{Aut}_{\gerbe{A}}}
& \opn{IG} \circ \, \gerbe{A}
\ar@{=>}[d]^{\opn{ig} \circ \, \bsym{1}_{\gerbe{A}}} \\
& \opn{Aut}_{\twocat{P}_1} \circ \, \gerbe{A}
}  \]
of natural transformations
between functors of stacks $\gerbe{G} \to \twocat{Grp}(X)$, is
commutative. 
\end{enumerate}

The set of twisted objects of $\twocat{P}$ is denoted by
$\opn{TwOb}(\twocat{P})$
\end{dfn}

What this definition amounts to is that on every open set $U$
the triple \linebreak
$\bigl( \gerbe{G}(U), \gerbe{A}, \opn{cp} \bigr)$
is almost a twisted object in the crossed groupoid $\twocat{P}(U)$; 
the exception is that the groupoid $\gerbe{G}(U)$ might be empty or
disconnected. These triples restrict correctly to smaller open sets.

In other words, to any local object
$i \in \opn{Ob} (\gerbe{G}(U))$ on an open set $U \subset X$
we attach an object $\gerbe{A}(i) \in \opn{Ob} (\twocat{P}(U))$, 
which we can also view as a sheaf on $U$ (since $\twocat{P}_1$ is a
stack). There is also a sheaf of groups 
$\opn{IG}(\gerbe{A}(i)) = \twocat{P}_2(\gerbe{A}(i))$ on $U$, 
and a homomorphism of sheaves of groups
$\opn{D} : \opn{IG}(\gerbe{A}(i)) \to \twocat{P}_1(\gerbe{A}(i))$.
To any other object $j \in \opn{Ob} (\gerbe{G}(U))$ and any
arrow $g \in \gerbe{G}(U)(i, j)$ we attach an isomorphism
$\gerbe{A}(g) : \gerbe{A}(i) \iso \gerbe{A}(j)$
in $\twocat{P}_1(U)$.
The various locally defined isomorphisms $\gerbe{A}(g)$ are
related by the composition rule in the gerbe $\gerbe{G}$.

When there is no danger of confusion we refer to the twisted object
$(\gerbe{G}, \gerbe{A}, \opn{cp})$ as $\gerbe{A}$, and we call
$\gerbe{G}$ the gauge gerbe of $\gerbe{A}$. 
An object $\gerbe{A}(i)$, for some open set $U \subset X$ and 
$i \in \opn{Ob} (\gerbe{G}(U))$, is called a {\em local object belonging
to $\gerbe{A}$}, or a {\em sheaf belonging to $\gerbe{A}$}.

\begin{exa} \label{exa:21}
Take the stack of crossed groupoids
$\twocat{Grp}^{\times}(X)$ from Example \ref{exa:9}. A twisted
object of $\twocat{Grp}^{\times}(X)$ is just a gerbe; and hence 
$\opn{TwOb}(\twocat{Grp}^{\times}(X)) = \twocat{Gerbe}(X)$.
\end{exa}

\begin{exa} \label{exa:tw.sh.101}
Consider the stack of crossed groupoids 
$\twocat{Assoc}^{\times}(R, X)$ from Example 
\ref{exa:10}. A twisted object of $\twocat{Assoc}^{\times}(R, X)$ is a stack of
$R$-algebroids on $X$, as defined in \cite{Ko2}; cf.\ Example \ref{exa:7}.
\end{exa}

We can finally define twisted deformations. 
The stacks $\twocat{AssDef}(R, \mcal{O}_X)$ and 
$\twocat{PoisDef}(R, \mcal{O}_X)$ were introduced in Definition \ref{dfn:210}. 

\begin{dfn} \label{dfn:24}
Assume Setup \ref{setup:200}. 
\begin{enumerate}
\item A twisted object of the stack of crossed groupoids
$\twocat{AssDef}(R, \mcal{O}_X)$ is called a {\em twisted
associative $R$-deformation of $\mcal{O}_X$}.
\item A twisted object of the  stack of crossed groupoids
$\twocat{PoisDef}(R, \mcal{O}_X)$ is called a {\em twisted
Poisson $R$-deformation of $\mcal{O}_X$}.
\end{enumerate}
\end{dfn}

\begin{dfn} \label{dfn:33}
Let $\twocat{P}$ be a stack of crossed groupoids on $X$, and 
let $(\gerbe{G}, \gerbe{A}, \opn{cp})$ and
$(\gerbe{G}', \gerbe{A}', \opn{cp}')$ be twisted objects of $\twocat{P}$.
A {\em twisted gauge transformation} 
\[ \Phi = (\Phi_{\mrm{gau}}, \Phi_{\mrm{rep}}) :
(\gerbe{G}, \gerbe{A}, \opn{cp}) \to 
(\gerbe{G}', \gerbe{A}', \opn{cp}') \]
consists of an equivalence of gerbes 
$\Phi_{\mrm{gau}} : \gerbe{G} \to \gerbe{G}'$,
and an isomorphism 
$\Phi_{\mrm{rep}} : \gerbe{A} \twoiso \gerbe{A}' \circ \Phi_{\mrm{gau}}$
of functors of stacks  $\gerbe{G} \to \twocat{P}_1$.
The condition is that the diagram 
\[ \UseTips \xymatrix @C=12ex @R=6ex {
\opn{Aut}_{\gerbe{G}}
\ar@{=>}[r]^{\opn{cp}}
\ar@{=>}[d]_{\opn{Aut}_{\Phi_{\mrm{gau}}}}
& 
\opn{IG} \circ \, \gerbe{A}
\ar@{=>}[d]^{\bsym{1}_{\opn{IG}} \circ \Phi_{\mrm{rep}}}
\\
\opn{Aut}_{\gerbe{G}'} \circ \, \Phi_{\mrm{gau}}
\ar@{=>}[r]^(0.47){\opn{cp}' \circ\, \bsym{1}_{\Phi_{\mrm{gau}}}}
& 
\opn{IG} \circ \, \gerbe{A}' \circ \Phi_{\mrm{gau}}
} \]
of natural transformations of functors of stacks
$\gerbe{G} \to \twocat{Grp} (X)$
is commutative.
\end{dfn}

Thus for every open set $U \subset X$ there is a twisted gauge
transformation
\[ \Phi = (\Phi_{\mrm{gau}}, \Phi_{\mrm{rep}}) :
\bigl( \gerbe{G}(U), \gerbe{A}, \opn{cp} \bigr) \to 
\bigl( \gerbe{G}'(U), \gerbe{A}', \opn{cp}' \bigr)  \]
as in Definition \ref{dfn:29}; and these are compatible with
restriction to smaller open sets.

A bit of calculation shows that twisted gauge transformations form an
equivalence relation on the
set $\opn{TwOb}(\twocat{P})$. 

\begin{dfn}
The equivalence relation given by twisted gauge transformations
is called {\em twisted gauge equivalence}, and we write
\[ \ol{\opn{TwOb}}(\twocat{P}) :=
\frac{ \opn{TwOb}(\twocat{P}) }
{\tup{ twisted gauge equivalence}} \ . \] 
\end{dfn}

\begin{rem} \label{rem:300}
Let us write $\twocat{P}(R, X)$ for either of the stacks of crossed groupoids
$\twocat{AssDef}(R, \mcal{O}_X)$ or $\twocat{PoisDef}(R, \mcal{O}_X)$.
One can show directly (but not so easily) that a homomorphism of parameter
algebras $R \to R'$ induces a function 
$\ol{\opn{TwOb}}(\twocat{P}(R, X)) \lb \to \ol{\opn{TwOb}}(\twocat{P}(R', X))$.
This is proved (indirectly) in Theorem \ref{thm:7}.
\end{rem}

\begin{rem} \label{rem:6}
In the paper \cite{KS2} the authors use the term {\em DQ algebroid} 
to denote a $\K[[\hbar]]$-linear algebroid $\gerbe{A}$ that locally looks like
an associative $\K[[\hbar]]$-deformation of $\mcal{O}_X$.
This is very close to being a twisted associative $\K[[\hbar]]$-deformation of
$\mcal{O}_X$ in our sense.
Indeed, a DQ algebroid $\gerbe{A}$ is a twisted object in the stack of crossed
groupoids $\twocat{DQ}$ in which the local objects are the sheaves
of $\K[[\hbar]]$-algebras $\mcal{A}$ locally isomorphic to $\OX[[\hbar]]$; 
the gauge transformations are the $\K[[\hbar]]$-algebra isomorphisms;
and the inner gauge group $\opn{IG}(\mcal{A})$ is the whole group
$\mcal{A}^{\times}$.
As a consequence, for DQ algebroids there are also $0$-th order
obstruction classes in $\mrm{H}^i(X, \mcal{O}_X^{\times})$,
$i = 1, 2$ (cf.\ Theorem
\ref{thm:6}, and \cite[Theorems 4.6 and 4.17]{Ye4}). We thank P.
Polesello for explaining this subtlety to us.
\end{rem}

\begin{rem}
Suppose $\gerbe{A}$ is a twisted associative $R$-deformation of
$\mcal{O}_X$ (with $R$ artinian). One can consider the stack of $R$-linear
abelian categories $\twocat{Coh} \gerbe{A}$ of coherent left
$\gerbe{A}$-modules. It is a deformation of the stack 
$\twocat{Coh} \mcal{O}_X$. The twisted associative deformation 
$\gerbe{A}$ can be recovered from the stack $\twocat{Coh} \gerbe{A}$;
and in fact these two notions of deformation are equivalent
(it is a kind of geometric Morita theory).
See the papers \cite{Ko2, LV, Lo, KS2, DP2} and the last chapter of the
book \cite{KS1}.
We do not know a similar interpretation of twisted Poisson
deformations.
\end{rem}

Suppose $U \subset X$ is an open set, and $\mcal{A}$ is an
associative or Poisson $R$-deformation of $\mcal{O}_U$. 
In Definition \ref{dfn:20} we defined the first order bracket
$\{ -,- \}_{\mcal{A}}$ on $\mcal{O}_U$. 
By Proposition \ref{prop:20} this is gauge invariant. Therefore the
next definition makes sense.

\begin{dfn} \label{dfn:22}
Let $\gerbe{A}$ be a twisted associative or
Poisson $R$-deformation of $\mcal{O}_X$. 
We define the {\em first order bracket of $\gerbe{A}$}
to be the unique $\K$-bilinear sheaf morphism
\[ \{ -,- \}_{\gerbe{A}} : 
\mcal{O}_X \times \mcal{O}_X \to (\m / \m^2) \ot  \mcal{O}_X
\]
having this property:
\begin{enumerate}
\rmitem{$*$} Let $\gerbe{G}$ be the gauge gerbe of $\gerbe{A}$, and let 
$i \in \opn{Ob} (\gerbe{G}(U))$ for some open set 
$U \subset X$. Consider $\mcal{A} := \gerbe{A}(i)$, which is an 
$R$-deformation of $\mcal{O}_U$. Then the restriction of
$\{ -,- \}_{\gerbe{A}}$ to $U$ equals 
$\{ -,- \}_{\mcal{A}}$. 
\end{enumerate}
\end{dfn}

Again, Proposition \ref{prop:20}  implies that if $\gerbe{A}$ and
$\gerbe{A}'$ are twisted associative 
\tup{(}resp.\ Poisson\tup{)} $R$-deformations of $\mcal{O}_X$ which
are twisted gauge equivalent, then 
$\{ -,- \}_{\gerbe{A}} = \{ -,- \}_{\gerbe{A}'}$.
This means that we can talk about the first order bracket of an
element of \linebreak
$\ol{\opn{TwOb}} \bigl( \twocat{AssDef}(R, \mcal{O}_X) \bigr)$
or
$\ol{\opn{TwOb}} \bigl( \twocat{PoisDef}(R, \mcal{O}_X) \bigr)$.

\begin{dfn} \label{dfn:25}
Let $(\gerbe{G}, \gerbe{A}, \opn{cp})$ be a twisted object of some
stack of crossed groupoids $\twocat{P}$ on $X$. We say that 
$(\gerbe{G}, \gerbe{A}, \opn{cp})$ is {\em really twisted} if there are no
global objects belonging to it; namely if 
$\opn{Ob} (\gerbe{G}(X)) = \emptyset$.
\end{dfn}

Sometimes there are obstruction classes that determine whether a
twisted object is really twisted (see \cite[Section 4]{Ye4}).

\section{Combinatorial Descent Data}
\label{sec:cosim}

We begin with a quick review of cosimplicial theory. 
Let $\bsym{\Delta}$ denote the simplex category. The set
of objects of $\bsym{\Delta}$ is the set $\mbb{N}$ of natural
numbers. Given $p, q \in \mbb{N}$, the morphisms $\alpha : p \to q$
in $\bsym{\Delta}$ are the order preserving functions
$\alpha : \{ 0, \ldots, p \} \to \{ 0, \ldots, q \}$.
We denote this set of morphisms by
$\bsym{\Delta}^q_p$.
An element of $\bsym{\Delta}^q_p$ may be thought of as a sequence
$\bsym{i} = (i_0, \ldots, i_p)$ of integers with
$0 \leq i_0 \leq \cdots \leq i_p \leq q$. 
We call 
$\bsym{\Delta}^q := \{ \bsym{\Delta}^q_p \}_{p \in \mbb{N}}$
the $q$-dimensional combinatorial simplex, and an element
$\bsym{i} \in \bsym{\Delta}^q_p$ is a $p$-dimensional face of
$\bsym{\Delta}^q$. 

Let $\cat{C}$ be some category. A {\em cosimplicial object} in 
$\cat{C}$ is a functor $C : \bsym{\Delta} \to \cat{C}$.
We shall usually write $C^p := C(p) \in \opn{Ob} (\cat{C})$, and leave
the morphisms
$C(\alpha) : C(p) \to C(q)$, for $\alpha \in \bsym{\Delta}^q_p$,
implicit. Thus we shall refer to the cosimplicial object $C$ as
$\{ C^p \}_{p \in \mbb{N}}$.
The category of cosimplicial objects in $\cat{C}$, where the morphisms 
are natural transformations of functors 
$\bsym{\Delta} \to \cat{C}$, is denoted by 
$\bsym{\Delta}(\cat{C})$.

If $\cat{C}$ is a category of sets with structure (i.e.\ there is a
faithful functor $\cat{C} \to \cat{Set}$), then an object 
$C \in \opn{Ob}(\cat{C})$ has elements $c \in C$. 
Let $\{ C^p \}_{p \in \mbb{N}}$ be a cosimplicial object of $\cat{C}$. 
Given a face $\bsym{i} \in \bsym{\Delta}_p^q$
and an element $c \in C^p$, it will be convenient to write
\begin{equation} \label{eqn:101}
c_{\bsym{i}} := C(\bsym{i})(c) \in C^q .
\end{equation}
The picture to keep in mind is of ``the element $c$ pushed to the face
$\bsym{i}$ of the simplex $\bsym{\Delta}^q$''. 
See Figure \ref{fig:2} for an illustration.

We shall be interested in cosimplicial crossed groupoids, i.e.\ in objects of
the category $\bsym{\Delta}(\cat{CrGrpd})$. A cosimplicial crossed groupoid 
$G = \{ G^p \}_{p \in \mbb{N}}$
has in each simplicial dimension $p$ a crossed groupoid $G^p$. 
The morphisms $G(\bsym{i}) : G^p \to G^q$, 
for $\bsym{i} \in \bsym{\Delta}_p^q$, are implicit, and we use notation
(\ref{eqn:101}). 

Let us fix $p \in \N$. Then for any 
$\om \in \opn{Ob}(G^p)$ there is a group homomorphism  (the feedback)
$\opn{D} : G^p_2(\om) \to  G^p_1(\om)$.
And for every morphism $g : \om \to \om'$ in $G_1^p$ there is a group
isomorphism (the twisting)
\[ \opn{Ad}_{G^p_1 \curvearrowright G^p_{2}}(g) :
G^p_2(\om) \to G^p_2(\om') . \]
We shall often use the abbreviated expression
$\opn{Ad}(g)$ to mean both $\opn{Ad}_{G^p_1 \curvearrowright G^p_{2}}(g)$
and 
$\opn{Ad}_{G^p_1}(g)$; hopefully that will not cause confusion.

\begin{dfn} \label{dfn:101}
Let $G = \{ G^p \}_{p \in \mbb{N}}$ be a cosimplicial
crossed groupoid. A {\em combinatorial descent datum} in $G$
is a triple $(\om, g, a)$ of elements of the following sorts:
\begin{enumerate}
\rmitem{0}  $\om \in \opn{Ob}(G^0)$. 

\rmitem{1} $g \in G^1_1(\om_{(0)}, \om_{(1)})$, where 
$\om_{(0)}, \om_{(1)} \in \opn{Ob}(G^1)$ are the objects corresponding
to the vertices $(0)$ and $(1)$ in $\bsym{\Delta}^1$.

\rmitem{2} $a \in G^2_2(\om_{(0)})$, where 
$\om_{(0)} \in \opn{Ob}(G^2)$ is the object corresponding
to the vertex $(0)$ of $\bsym{\Delta}^2$.
\end{enumerate}

The conditions are as follows:
\begin{enumerate}
\rmitem{i} (Failure of $1$-cocycle) 
\[ g_{(0, 2)}^{-1} \circ g_{(1, 2)} \circ g_{(0, 1)} = \opn{D}(a) \]
in the group $G^2_1(\om_{(0)})$.
Here $\om_{(i)} \in \opn{Ob}(G^2)$ 
and $g_{(i, j)} \in G^2_1(\om_{(i)}, \om_{(j)})$ 
correspond to the faces $(i)$ and $(i, j)$ respectively of $\bsym{\Delta}^2$.

\rmitem{ii} (Twisted $2$-cocycle) 
\[ a_{(0, 1, 3)}^{-1} \circ a_{(0, 2, 3)} \circ 
a_{(0, 1, 2)} = 
\opn{Ad}(g_{(0, 1)}^{-1})(a_{(1, 2, 3)}) \]
in the group
$G^3_2(\om_{(0)})$.
Here $\om_{(i)} \in \opn{Ob}(G^3)$, 
$g_{(i, j)} \in G^3_1(\om_{(i)}, \om_{(j)})$ and
$a_{(i, j, k)} \in G^3_2(\om_{(i)})$
correspond to the faces $(i)$,  $(i, j)$ and $(i, j, k)$ respectively of
$\bsym{\Delta}^3$.
\end{enumerate}

We denote by $\opn{Desc}(G)$ the set of all
descent data in $G$. 
\end{dfn}

See Figure \ref{fig:2} for an illustration.

\begin{figure}
\begin{center}
\includegraphics[scale=0.6]{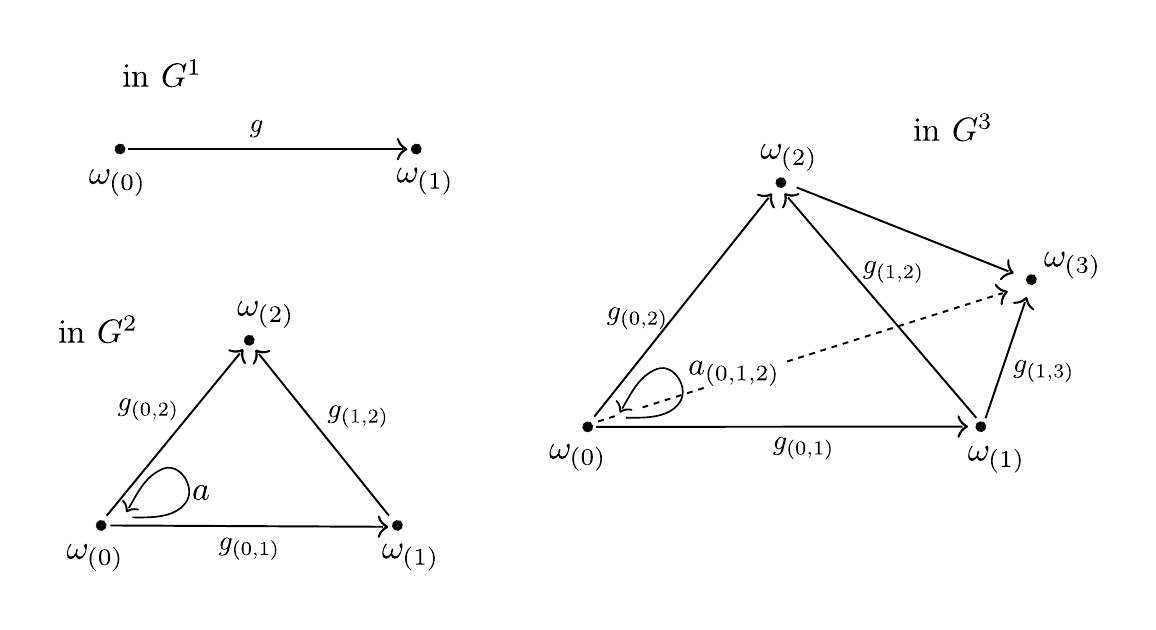}
\caption{Illustration of a combinatorial descent datum $(\om, g, a)$ in the
cosimplicial crossed groupoid 
$G = \{ G^p \}_{p \in \mbb{N}}$.}
\label{fig:2}
\end{center}
\end{figure}

\begin{dfn} \label{dfn:100}
Let $(\om, g, a)$ and $(\om', g', a')$
be combinatorial descent data in the cosimplicial crossed groupoid 
$G$. A gauge transformation
$(\om, g, a) \to (\om', g', a')$ is a pair $(f, b)$ 
of elements of the following sorts:
\begin{enumerate}
\rmitem{0}  $f \in G^0_1(\om, \om')$.

\rmitem{1} $b \in G^1_2(\om_{(0)})$, where 
$\om_{(0)} \in \opn{Ob}(G^1)$ is the object corresponding
to the vertex $(0)$ in $\bsym{\Delta}^1$.
\end{enumerate}

These two conditions must hold:
\begin{enumerate}
\rmitem{i}  
\[ g'  = f_{(1)} \circ g \circ \opn{D}(b) \circ f_{(0)}^{-1} \]
in the set
$G^1_1(\om'_{(0)}, \om'_{(1)})$.

\rmitem{ii}
\[ a' =   
\opn{Ad}(f_{(0)}) \Bigl( b_{(0, 2)}^{-1} \circ a \circ 
\opn{Ad}(g_{(0, 1)}^{-1})(b_{(1, 2)}) \circ b_{(0, 1)} \Bigr) \]
in the group $G^2_2(\om'_{(0)})$.
\end{enumerate}
\end{dfn}

This is illustrated in Figure \ref{fig:3}.

\begin{figure}
\begin{center}
\includegraphics[scale=0.55]{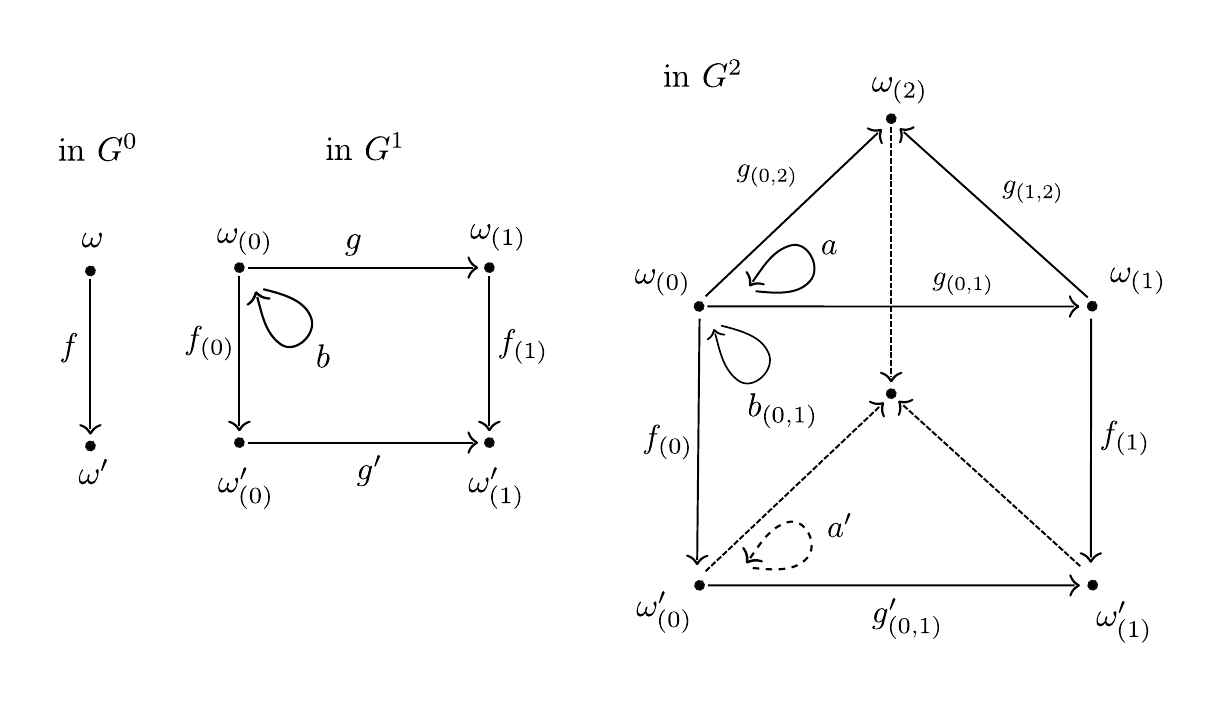}
\caption{Illustration of a gauge transformation 
$(f, b) : (\om, g, a) \to (\om', g', a')$
between descent data.}
\label{fig:3}
\end{center}
\end{figure}

The gauge transformations form an equivalence relation on the set
$\opn{Desc}(G)$; see \cite[Proposition 1.6]{Ye8}.
We call this relation {\em gauge equivalence}.

Let $\Phi : G \to H$ be a morphism of cosimplicial
crossed groupoids, i.e.\ a morphism in $\bsym{\Delta}(\cat{CrGrpd})$. 
Given a descent datum 
$(\om, g, a) \in \opn{Desc}(G)$, the triple 
$\Phi(\om, g, a) := (\Phi(\om), \Phi(g), \Phi(a))$
is a descent datum in $H$. The resulting function 
$\opn{Desc}(\Phi) : \opn{Desc}(G) \to \opn{Desc}(H)$
respects the gauge equivalence relations. 

\begin{dfn} \label{dfn:cosim.1}
For a cosimplicial crossed groupoid $G$ we write
\[ \ol{\opn{Desc}}(G) :=
\frac{ \opn{Desc}(G) } {\tup{ gauge equivalence}} . \]
For a morphism $\Phi : G \to H$
of cosimplicial crossed groupoids, we denote by 
\[ \ol{\opn{Desc}}(\Phi) : \ol{\opn{Desc}}(G) \to
\ol{\opn{Desc}}(H) \]
the induced function. 
\end{dfn}

\begin{dfn} \label{dfn:equiv-cosim.103}
A morphism $\Phi : G \to H$
of cosimplicial crossed groupoids is called a {\em weak equivalence} if in
every simplicial dimension $p$ the morphism of crossed groupoids
 $\Phi^p : G^p \to H^p$ is a weak equivalence (Definition
\ref{dfn:equiv-cosim.102}).
\end{dfn}

The next result is proved in the companion paper \cite{Ye8}.
An alternative proof (of a more general statement) can be found in M. Prasma's
paper \cite{Pr}. 

\begin{thm}[{Combinatorial Equivalence, \cite[Theorem 0.1]{Ye8}}]
\label{thm:100} 
Let $\Phi : G \to H$ be a weak equivalence between 
cosimplicial crossed groupoids. Then the function 
\[ \ol{\opn{Desc}}(\Phi) : \ol{\opn{Desc}}(G) \to
\ol{\opn{Desc}}(H) \]
from Definition \tup{\ref{dfn:cosim.1}} is bijective.
\end{thm}

\begin{rem} \label{rem:cosim.106}
If  we view the cosimplicial crossed groupoid as a cosimplicial simplicial set,
then the set of descent data  $\opn{Desc}(G)$ becomes the {\em total space} of
$G$. See \cite{BGNT2} and \cite{Pr}.

We discovered Theorem \ref{thm:100} while trying to understand, and possibly
improve, \cite[Proposition 3.3.1]{BGNT1} and the ideas in \cite{CH}. 
The precise relation to \cite[Proposition 3.3.1]{BGNT1} is explained in Remark
\ref{rem:Lie-desc.101}. As for \cite{CH}: there seems to be an implicit
assumption there that the nerve $\opn{N}(G)$ is a Reedy fibrant
cosimplicial simplicial set. This might not be true -- see \cite{Pr}.
\end{rem}

\section{Geometric Descent Data and Decomposition}
\label{sec:mdd}

In this section we study the decomposition of twisted objects on open
coverings.

Given a set $K$ and $p \in \N$, we let 
\begin{equation} \label{eqn:mdd.111}
\bsym{\Delta}_p(K) := \{ (k_0, \ldots, k_p) \mid k_i \in K \} .
\end{equation}
The collection $\{ \bsym{\Delta}_p(K) \}_{p \in \N}$ is a simplicial set. 
Recall that for an open covering $\bsym{U} = \{ U_k \}_{k \in K}$ of a
topological space $X$, and for 
$\bsym{k} = (k_0, \ldots, k_p) \in \bsym{\Delta}_p(K)$, 
we write 
$U_{\bsym{k}} = U_{k_0} \cap \cdots \cap U_{k_p}$.

Let $\{ G^j \}_{j \in J}$ be a collection of crossed groupoids
indexed by some set $J$.
Their product is the crossed groupoid  
$\prod_{j \in J} G^j$, with set of objects 
$\opn{Ob}(\prod_{j \in J} G^j) := 
\prod_{j \in J} \opn{Ob}(G^j)$
etc.

The notion of stack of crossed groupoids
\[ \twocat{P} = ( \twocat{P}_1, \twocat{P}_2, 
\opn{Ad}_{\twocat{P}_1 \crvar \twocat{P}_{2}}, \opn{D} ) \]
on a topological space $X$ was introduced in Definition \ref{dfn:200}.
We will make frequent use of the functor of stacks 
$\opn{IG} : \twocat{P}_1 \to \twocat{Grp}(X)$,
and the natural transformation
$\opn{ig} : \opn{IG} \twoto \opn{Ad}_{\twocat{P}_1}$,
from Proposition \ref{prop:201}.

\begin{dfn} \label{dfn:mdd.101}
Let $\twocat{P}$ be a stack of crossed groupoids on a
topological space $X$, and let $\bsym{U} = \{ U_k \}_{k \in K}$ be an
open covering of $X$. For any $p \in \N$ define the crossed groupoid 
\[ \opn{C}^p(\bsym{U}, \twocat{P}) := 
\prod_{\bsym{k} \in \bsym{\Delta}_p(K)} \twocat{P}(U_{\bsym{k}}) . \]

The restriction morphisms 
$\mrm{rest}^{}_{V / U}$ induce a morphism of crossed groupoids 
$\opn{C}^p(\bsym{U}, \twocat{P}) \to \opn{C}^q(\bsym{U}, \twocat{P})$
for any $\al : p \to q$ in $\bsym{\Delta}$. 
In this way we obtain a cosimplicial crossed groupoid 
$\opn{C}(\bsym{U}, \twocat{P}) = 
\{ \opn{C}^p(\bsym{U}, \twocat{P}) \}_{p \in \N}$,
which we refer to as the {\em \v{C}ech cosimplicial crossed groupoid}
associated to $\twocat{P}$ and $\bsym{U}$. 
\end{dfn}

The set of combinatorial descent data 
$\opn{Desc}(G)$ in a cosimplicial crossed groupoid $G$ 
was defined in Definition \ref{dfn:101}.

\begin{dfn} \label{dfn:21}
Let $\twocat{P}$ be a stack of crossed groupoids on a
topological space $X$, and let $\bsym{U} = \{ U_k \}_{k \in K}$ be an
open covering of $X$. A {\em geometric descent datum} for 
$(\bsym{U}, \twocat{P})$ is an element of the set 
$\opn{Desc}(\opn{C}(\bsym{U}, \twocat{P}))$.

Given $\bsym{d}, \bsym{d}' \in \opn{Desc}(\opn{C}(\bsym{U}, \twocat{P}))$,
a {\em gauge transformation} 
$\bsym{f} : \bsym{d} \to \bsym{d}'$ is a gauge transformation
of combinatorial descent data, in the sense of Definition \ref{dfn:100}.
Passing to gauge equivalence classes we get the set
$\ol{\opn{Desc}}(\opn{C}(\bsym{U}, \twocat{P}))$.
\end{dfn}

Here is an explicit interpretation of geometric descent data,
gotten by unraveling the definitions. A descent datum 
$\bsym{d} \in \opn{Desc}(\opn{C}(\bsym{U}, \twocat{P}))$ 
is actually a collection 
\begin{equation} \label{eqn:mdd.120}
\bsym{d} = \bigl( \{ \mcal{A}_{k_0} \}_{k_0 \in K}, \,
\{ g_{k_0, k_1} \}_{k_0, k_1 \in K}, \,
\{ a_{k_0, k_1, k_2} \}_{k_0, k_1, k_2 \in K} \bigr) ,
\end{equation}
where
$ \mcal{A}_{k_0} \in \opn{Ob} (\twocat{P}(U_{k_0}))$,
$g_{k_0, k_1}: \mcal{A}_{k_0}|_{U_{k_0, k_1}} \to 
\mcal{A}_{k_1}|_{U_{k_0, k_1}}$
is an isomorphism in $\twocat{P}_1(U_{k_0, k_1})$, and
$a_{k_0, k_1, k_2} \in 
\Gamma \bigl( U_{k_0, k_1, k_2}, \opn{IG}(\mcal{A}_{k_0}) \bigr)$.
Conditions (i-ii) of Definition \ref{dfn:101} say that
\begin{equation} \label{eqn:bitorsors.11}
 g_{k_0, k_2}^{-1} \circ g_{k_1, k_2} \circ g_{k_0, k_1} =
\opn{ig}(a_{k_0, k_1, k_2})
\end{equation}
as automorphisms of
$\mcal{A}_{k_0}|_{U_{k_0, k_1, k_2}}$
in $\twocat{P}_1(U_{k_0, k_1, k_2})$, and 
\begin{equation} \label{eqn:bitorsors.12}
a_{k_0, k_1, k_3}^{-1} \cdot a_{k_0, k_2, k_3} \cdot 
a_{k_0, k_1, k_2} = \opn{IG}(g_{k_0, k_1}^{-1})
(a_{k_1, k_2, k_3})
\end{equation}
in the group
$\Gamma \bigl( U_{k_0, \ldots, k_3}, \opn{IG}(\mcal{A}_{k_0})
\bigr)$.
See Figure \ref{fig:5} for an illustration. 

\begin{figure}
\begin{center}
\includegraphics[scale=0.35]{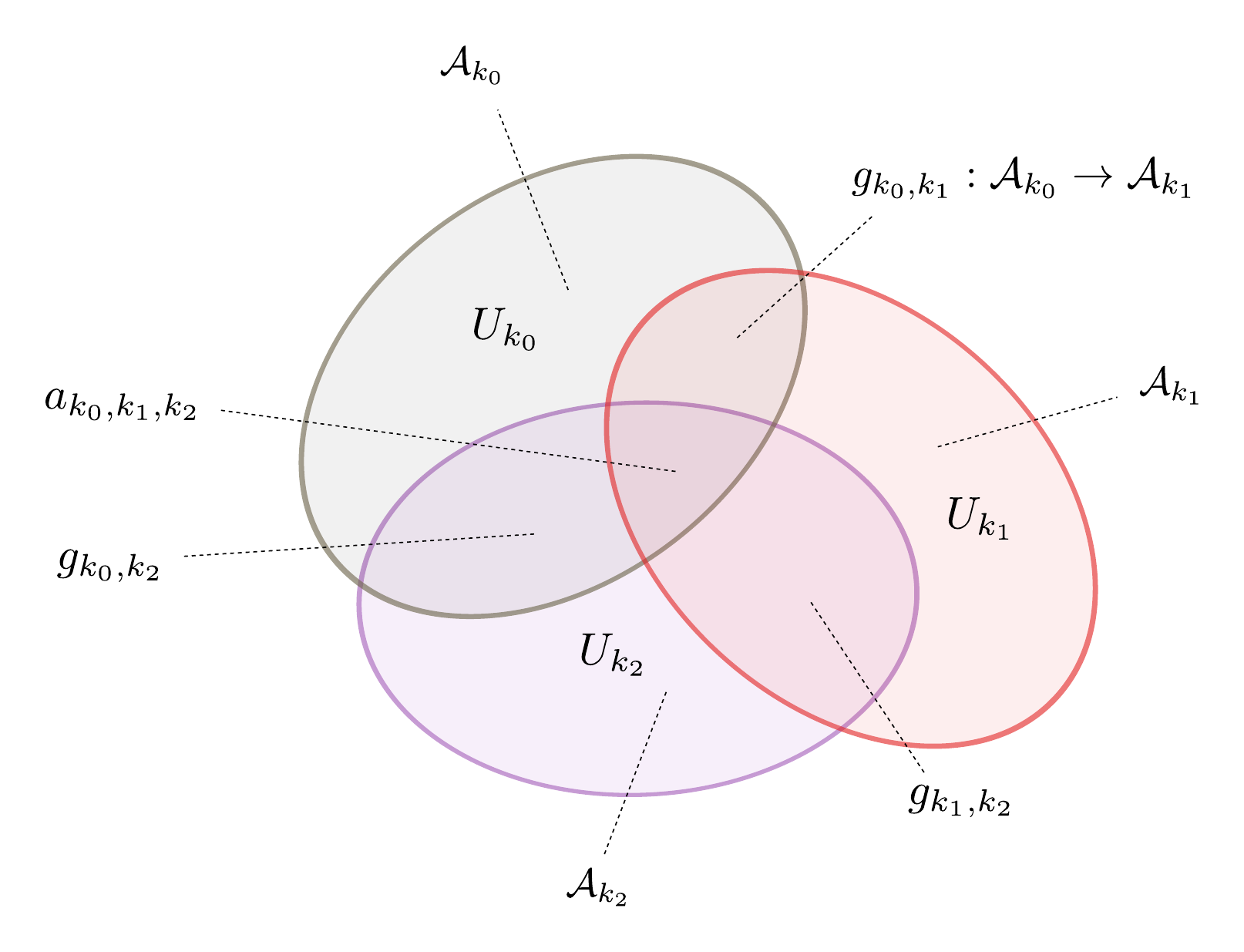}
\caption{Geometric descent datum on an open covering
$\{ U_k \}_{k \in K}$ of a topological space.}
\label{fig:5}
\end{center}
\end{figure}

Given another descent datum 
$\bsym{d}' = \bigl( \{ \mcal{A}'_{k_0} \}, \{ g'_{k_0, k_1} \},
\{ a'_{k_0, k_1, k_2} \} \bigr)$,
a gauge transformation 
$\bsym{f} : \bsym{d} \to \bsym{d}'$ 
is actually a collection 
\begin{equation} \label{eqn:bitorsors.14}
\bsym{f} = \bigl( \{ f_{k_0} \}, \{ b_{k_0, k_1} \} \bigr) ,
\end{equation}
where
$f_{k_0} : \mcal{A}_{k_0} \to \mcal{A}'_{k_0}$
is an isomorphism in $\twocat{P}_1(U_{k_0})$, and
$b_{k_0, k_1} \in \lb \Gamma(U_{k_0, k_1}, \opn{IG}(\mcal{A}_{k_0}))$.
Conditions (i-ii) of Definition \ref{dfn:100} say that  
\begin{equation} \label{eqn:bitorsors.6}
g'_{k_0, k_1} =  f_{k_1} \circ g_{k_0, k_1} \circ \opn{ig}(b_{k_0, k_1})
\circ f_{k_0}^{-1}
\end{equation}
as isomorphisms 
$\mcal{A}'_{k_0}|_{U_{k_0, k_1}} \to \mcal{A}'_{k_1}|_{U_{k_0, k_1}}$,
and 
\begin{equation} \label{eqn:bitorsors.7}
a'_{k_0, k_1, k_2} = \opn{IG}(f_{k_0}) \bigl(  
b_{k_0, k_2}^{-1} \cdot a_{k_0, k_1, k_2} \cdot 
\opn{IG}(g_{k_0, k_1}^{-1})(b_{k_1, k_2}) \cdot b_{k_0, k_1} \bigr)
\end{equation}
in the group $\opn{IG}(\mcal{A}'_{k_0}|_{U_{k_0, k_1, k_2}})$.

\begin{rem}
The similarity between our notion of geometric descent data and
the usual notion of descent data for gerbes (cf.\ \cite[Section IV.3.5]{Gi} or
\cite[Section 5]{Br}) is no coincidence. Indeed, as shown in Example
\ref{exa:21}, gerbes are an instance of twisted objects. Likewise for gauge
transformations between descent data.

In \cite{Ko2}, geometric descent data, for 
$\twocat{P} = \twocat{Assoc}(R, X)$, is called a {\em combinatorial
description of algebroids}.
Note that in the paper \cite{BGNT1} the authors refer to a
geometric descent datum as a ``stack''.
\end{rem}

\begin{dfn} \label{dfn:bitorsors.3}
Let $\gerbe{G}$ be a gerbe on $X$, and let 
$\bsym{U} = \{ U_k \}_{k \in K}$ be an open covering of $X$.  
A {\em trivialization datum} of $\gerbe{G}$ on $\bsym{U}$ is a collection 
$\bsym{s} = \bigl( \{ i_{k_0} \} , \{ s_{k_0, k_1} \} \bigr)$,
where $i_{k_0} \in \opn{Ob}(\gerbe{G}(U_{k_0}))$ and 
$s_{k_0, k_1} \in \gerbe{G}(U_{k_0, k_1})(i_{k_0}, i_{k_1})$.
\end{dfn}

Consider a twisted object
$(\gerbe{G}, \gerbe{A}, \opn{cp})$
of the stack of crossed groupoids $\twocat{P}$, 
and a trivialization datum 
$\bsym{s} = \bigl( \{ i_{k_0} \} , \{ s_{k_0, k_1} \} \bigr)$ 
of the gauge gerbe $\gerbe{G}$ on the open covering $\bsym{U}$. Define 
\begin{equation}
\begin{aligned}
\mcal{A}_{k_0} & := \gerbe{A}(i_{k_0}) \in \opn{Ob}(\twocat{P}(U_{k_0})) , 
\\
g_{k_0, k_1} & := \gerbe{A}(s_{k_0, k_1}) \in 
\twocat{P}_1(U_{k_0, k_1})(i_{k_0}, i_{k_1}) , 
\\
a_{k_0, k_1, k_2} & := \opn{cp} ( s_{k_0, k_2}^{-1} \circ s_{k_1, k_2} \circ
s_{k_0, k_1} ) \in \Gamma(U_{k_0, k_1, k_2}, \opn{IG}(\mcal{A}_{k_0})) . 
\end{aligned}
\end{equation}
The resulting collection of elements, arranged as in (\ref{eqn:mdd.120}),
is denoted by $\opn{dec}(\gerbe{A}, \bsym{s})$.

\begin{lem}   \label{lem:bitorsors.2}
Let $(\gerbe{G}, \gerbe{A}, \opn{cp})$ be a twisted object of $\twocat{P}$,
and let $\bsym{s}$ be a trivialization datum for $\gerbe{G}$ on $\bsym{U}$.
\begin{enumerate}
\item The collection $\opn{dec}(\gerbe{A}, \bsym{s})$ is a geometric descent
datum for $(\bsym{U}, \twocat{P})$. 

\item Let $(\gerbe{G}', \gerbe{A}', \opn{cp}')$ be another another twisted
object of $\twocat{P}$, which is twisted gauge equivalent to 
$(\gerbe{G}, \gerbe{A}, \opn{cp})$, 
and let $\bsym{s}'$ be a trivialization datum for $\gerbe{G}'$ on $\bsym{U}$.
Assume that the groupoids 
$\gerbe{G}(U_k)$ are connected for all $k \in K$. Then 
the descent data $\opn{dec}(\gerbe{A}, \bsym{s})$ and 
$\opn{dec}(\gerbe{A}', \bsym{s}')$ are gauge equivalent.
\end{enumerate}
\end{lem}

\begin{proof}
(1) This is a simple calculation.

\medskip \noindent
(2) First assume that 
$(\gerbe{G}', \gerbe{A}', \opn{cp}') = 
(\gerbe{G}, \gerbe{A}, \opn{cp})$, 
but 
$\bsym{s}' = \bigl( \{ i'_{k_0} \} , \{ s'_{k_0, k_1} \} \bigr)$ 
is a different descent datum. Since the groupoids 
$\gerbe{G}(U_k)$ are connected, we can find elements 
$t_k \in \gerbe{G}(U_k)(i_k, i'_k)$. 
Define $f_k := \gerbe{A}(t_k)$, which is an isomorphism 
$\mcal{A}_k \to \mcal{A}'_k$ in $\twocat{P}_1$, 
and
\[ b_{k_0, k_1} := \opn{cp}( s_{k_0, k_1}^{-1}  \circ 
t_{k_1}^{-1}  \circ s'_{k_0, k_1}  \circ t_{k_0} ) \in 
\Gamma(U_{k_0, k_1}, \opn{IG}(\mcal{A}_{k_0})) . \]
A calculation shows that the collection $\bsym{f}$, with notation as in 
(\ref{eqn:bitorsors.14}), is a gauge transformation 
$\opn{dec}(\gerbe{A}, \bsym{s}) \to \opn{dec}(\gerbe{A}', \bsym{s}')$.

Now consider a twisted gauge transformation 
$(\Phi_{\mrm{gau}}, \Phi_{\mrm{rep}}) : 
(\gerbe{G}, \gerbe{A}, \opn{cp}) \to \lb (\gerbe{G}', \gerbe{A}', \opn{cp}')$.
Since $\Phi_{\mrm{gau}} : \gerbe{G} \to \gerbe{G}'$
is an equivalence of gerbes, it follows that the groupoids 
$\gerbe{G}'(U_k)$ are all connected. So, by the previous paragraph, we don't
care which trivialization datum we use for $\gerbe{G}'$.
Let us define $\bsym{s}'$ to be 
$i'_k := \Phi_{\mrm{gau}}(i_k)$ and 
$s'_{k_0, k_1} :=  \Phi_{\mrm{gau}}(s_{k_0, k_1})$. 
Take 
$f_k := \Phi_{\mrm{rep}} : \mcal{A}_{k} \to \mcal{A}'_{k}$
and $b_{k_0, k_1} := 1$. 
The resulting collection $\bsym{f}$ is a gauge transformation 
$\opn{dec}(\gerbe{A}, \bsym{s}) \to \opn{dec}(\gerbe{A}', \bsym{s}')$.
\end{proof}

\begin{dfn} \label{dfn:18}
Let $\gerbe{G}$ be a gerbe on a topological space $X$, and let 
$\bsym{U} = \{ U_k \}_{k \in K}$ be an open covering of $X$. We say
that $\bsym{U}$ {\em totally trivializes} $\gerbe{G}$ if 
for every $p \in \{ 0, 1 \}$ and $\bsym{k} \in \bsym{\Delta}_p(K)$ the 
groupoid $\gerbe{G}(U_{\bsym{k}})$ is nonempty and connected.
\end{dfn}

\begin{dfn} \label{dfn:bitorsors.10}
Let $\twocat{P}$ be a stack of crossed groupoids on a topological space $X$,
and let $\bsym{U}$ be an open covering of $X$. 
\begin{enumerate}
\item Let $(\gerbe{G}, \gerbe{A}, \opn{cp})$ be a twisted
object of $\twocat{P}$. We say that $\bsym{U}$ totally trivializes
$(\gerbe{G}, \gerbe{A}, \opn{cp})$ if it totally trivializes the gauge gerbe
$\gerbe{G}$.

\item We say that $\bsym{U}$ totally trivializes 
$\twocat{P}$ if it totally trivializes all twisted objects of $\twocat{P}$.
\end{enumerate}
\end{dfn}

\begin{rem}
In general there is no reason to expect that totally trivializing open
coverings should exist. If we were to work with hypercoverings,
then such a difficulty would not arise; but then the combinatorics would be
more complicated. Cf.\ \cite[Section 5]{Br}.

We shall see in Corollary \ref{cor:4}  that for
the stacks with inner gauge groups \lb
$\twocat{AssDef}(R,\mcal{O}_X)$ and $\twocat{PoisDef}(R, \mcal{O}_X)$
there do exist totally trivializing open coverings.
\end{rem}

\begin{prop} \label{prop:bitorsors.5}
Let $\bsym{U}$ be an open covering of $X$ that totally trivializes 
the stack of crossed groupoids $\twocat{P}$. Then  there is a function 
\[ \ol{\opn{dec}} :  \ol{\opn{TwOb}}(\twocat{P}) \to 
\ol{\opn{Desc}}(\opn{C}(\bsym{U}, \twocat{P})) \]
called {\em decomposition}, 
which is represented by $\opn{dec}(\gerbe{A}, \bsym{s})$
for any representative $\gerbe{A} \in \opn{TwOb}(\twocat{P})$ 
and any trivialization datum $\bsym{s}$ for $\gerbe{A}$ on $\bsym{U}$.
\end{prop}

\begin{proof}
Take any twisted object $(\gerbe{G}, \gerbe{A}, \opn{cp})$. 
Since the groupoids $\gerbe{G}(U_{k_0})$ are \lb nonempty, and
the groupoids $\gerbe{G}(U_{k_0, k_1})$ are connected, we can find 
a trivialization datum $\bsym{s}$. By Lemma \ref{lem:bitorsors.2}(1) we get 
a geometric descent datum $\opn{dec}(\gerbe{A}, \bsym{s})$.
Since the groupoids $\gerbe{G}(U_{k_0})$ are connected, 
by Lemma \ref{lem:bitorsors.2}(2) this gives a well-defined function on
equivalence classes.
\end{proof}

Recall that a refinement of  $\bsym{U}$ is an open covering 
$\bsym{U}' = \{ U'_k \}_{k \in K'}$ of $X$, together with a function
$\rho : K' \to K$, such that 
$U'_k \subset U_{\rho(k)}$ for every $k \in K'$.
Sometimes we say that $\rho : \bsym{U}' \to \bsym{U}$ is a
refinement. Given such a refinement, there is a function
\begin{equation} \label{eqn:23}
\rho^* : \opn{Desc}(\opn{C}(\bsym{U}, \twocat{P}))
\to \opn{Desc}(\opn{C}(\bsym{U}', \twocat{P})) . 
\end{equation}
The formula is obvious: say 
$\bsym{d} = \bigl( \{ \mcal{A}_k \}_{k \in K}, \ldots \bigr)$;
then 
$\rho^*(\bsym{d}) = \bigl( \{ \mcal{A}'_k \}_{k \in K'}, \ldots
\bigr)$, 
where $\mcal{A}'_{k} := \mcal{A}_{\rho(k)}|_{U'_k}$, etc. 
It is easy to see that this function preserves the equivalence
relation.
Thus $\ol{\opn{Desc}}(\opn{C}(\bsym{U}, \twocat{P}))$ is
contravariant with respect
to refinements $\bsym{U}' \to \bsym{U}$.

\begin{prop} \label{prop:bitorsors.6}
Let $\bsym{U}$ and $\bsym{U}'$ be open coverings of $X$ that totally
trivialize the stack of crossed groupoids $\twocat{P}$, and let 
$\rho : \bsym{U} \to \bsym{U}'$ be a refinement. 
Then the diagram 
\[ \UseTips \xymatrix @C=9ex @R=5ex {
\ol{\opn{TwOb}}(\twocat{P})
\ar[r]^(0.4){\ol{\opn{dec}}}
\ar[dr]_{\ol{\opn{dec}}}
&
\ol{\opn{Desc}}(\opn{C}(\bsym{U}, \twocat{P}))
\ar[d]^{\rho^*} 
\\
&
\ol{\opn{Desc}}(\opn{C}(\bsym{U}', \twocat{P})) 
} \]
is commutative.
\end{prop}

\begin{proof}
This is immediate from the construction; cf.\ proof of Lemma
\ref{lem:bitorsors.2}(2).
\end{proof}

In the rest of this section we will show how to glue twisted objects from
descent data. This construction will be inverse to the decomposition operation.
But before attempting to glue twisted objects, we have to
understand how to pass from descent data to connected groupoids. 

Let $K$ be some nonempty set. By {\em groupoid descent datum} indexed by $K$
we mean a collection 
$\bsym{d} = \bigl( \{ G_{k_0} \}, \{ \phi_{k_0, k_1} \},
\{ a_{k_0, k_1, k_2} \} \bigr)$, 
where $k_0, k_1, k_2 \in K$, $G_{k_0}$ is a group,
$\phi_{k_0, k_1} : G_{k_0} \to G_{k_1}$ is a group isomorphism, and
$a_{k_0, k_1, k_2} \in G_{k_0}$.
These satisfy conditions  (\ref{eqn:bitorsors.11})
and (\ref{eqn:bitorsors.12}), with the obvious modifications in notation (e.g.\
$\phi_{k_0, k_1}$ instead of $g_{k_0, k_1}$). 

Suppose $\bsym{d}$ and $\bsym{d}'$ are groupoid descent data indexed by $K$. 
A gauge transformation $\bsym{f} : \bsym{d} \to \bsym{d}'$ is a 
collection 
$\bsym{f} = \bigl( \{ \psi_{k_0} \}, \{ b_{k_0, k_1} \} \bigr)$
of group isomorphisms $\psi_{k_0} : G_{k_0} \to G'_{k_1}$, and of
elements $b_{k_0, k_1} \in G_{k_0}$. The conditions are the translations of
formulas (\ref{eqn:bitorsors.6}) and (\ref{eqn:bitorsors.7}) to the present
context.

A trivialization datum for a connected groupoid $G$ is a collection 
$\bsym{s} = \{  s_{k_0, k_1} \}$ of elements 
$s_{k_0, k_1} \in G(k_0, k_1)$,
indexed by the set $K := \opn{Ob}(G)$.
We call the pair $(G, \bsym{s})$ a trivialized groupoid. 
By isomorphism of trivialized groupoids 
$F : (G, \bsym{s}) \to (G', \bsym{s}')$
we mean an isomorphism of groupoids $F :G \to G'$
such that $F(s_{k_0, k_1}) = s'_{F(k_0), F(k_1)}$.
Just like in Lemma \ref{lem:bitorsors.2}(1), to any trivialized groupoid 
$(G, \bsym{s})$ there is a corresponding groupoid descent datum 
$\opn{dec}(G, \bsym{s})$.

\begin{lem} \label{lem:bitorsors.3}
Let $\bsym{d}$ be a groupoid descent datum indexed by a nonempty set $K$. 
\begin{enumerate}
\item There is a connected groupoid $G$, with 
trivialization datum $\bsym{s}$, such that 
$\opn{dec}(G, \bsym{s}) = \bsym{d}$.
The trivialized groupoid $(G, \bsym{s})$ is unique up to a unique
isomorphism. We write $\opn{glu}(\bsym{d}) := G$.

\item Let $\bsym{d}'$ be another groupoid descent datum indexed by
$K$, and let $\bsym{f} : \bsym{d} \to \bsym{d}'$ be a gauge
transformation. We describe $\bsym{f}$ as in \tup{(\ref{eqn:bitorsors.14})}.
Then there is a unique isomorphism of groupoids 
$\opn{glu}(\bsym{f}) : \opn{glu}(\bsym{d}) \to \opn{glu}(\bsym{d}')$ 
such that 
$\opn{glu}(\bsym{f})(s_{k_0, k_1} \cdot b_{k_0, k_1}) = s'_{k_0, k_1}$.
\end{enumerate}
\end{lem}

\begin{proof}
(1) This seems to be known (cf.\ \cite[Appendix A1]{Ko2}), but here is a sketch
anyhow. Write $\bsym{d}$ as in (\ref{eqn:mdd.120}).
For any $k_0, k_1 \in K$ we take $G(k_0, k_1)$ to be the 
$G_{k_1}$-$G_{k_0}$-bitorsor with basis 
$s_{k_0, k_1}$, such that 
$s_{k_0, k_1} \cdot a = \phi_{k_0, k_1}(a) \cdot s_{k_0, k_1}$
for any $a \in G_{k_0}$.
The composition $\circ$ in $G$ is specified by its value on the basis
elements:
\[ s_{k_1, k_2} \circ s_{k_0, k_1} := s_{k_0, k_2} \cdot a_{k_0, k_2, k_2} . \]

\medskip \noindent
(2) A direct calculation.
\end{proof}

\begin{lem}  \label{lem:bitorsors.9}
There is a function 
\[ \opn{glu} : \opn{Desc}(\opn{C}(\bsym{U}, \twocat{P})) \to 
\opn{TwOb}(\twocat{P}) \]
having these properties\tup{:}
\begin{enumerate}
\rmitem{i} If 
$\bsym{d}, \bsym{d}' \in \opn{Desc}(\opn{C}(\bsym{U}, \twocat{P}))$
are gauge equivalent, then $\opn{glu}(\bsym{d})$ and 
$\opn{glu}(\bsym{d}')$ are twisted gauge equivalent.

\rmitem{ii} For any $\bsym{d} \in \opn{Desc}(\opn{C}(\bsym{U}, \twocat{P}))$
the twisted object 
$(\gerbe{G}, \gerbe{A}, \opn{cp}) := \opn{glu}(\bsym{d})$ has a 
trivialization datum $\bsym{s}$ such that 
$\opn{dec}(\gerbe{A}, \bsym{s})= \bsym{d}$.

\rmitem{iii} Let $(\gerbe{G}, \gerbe{A}, \opn{cp})$
be a twisted object of $\twocat{P}$, with trivialization datum
$\bsym{s}$, and let $\bsym{d} := \opn{dec}(\gerbe{A}, \bsym{s})$.
Then $\opn{glu}(\bsym{d})$ is twisted gauge equivalent to 
$(\gerbe{G}, \gerbe{A}, \opn{cp})$.
\end{enumerate}
\end{lem}

\begin{proof}
Take any $\bsym{d} \in \opn{Desc}(\opn{C}(\bsym{U}, \twocat{P}))$, which we
describe as in (\ref{eqn:mdd.120}).
For a nonempty open set $V \subset X$ let 
$K(V) := \{ k \in K \mid V \subset U_k \}$. 
Consider the collection 
\[ \opn{IG}(\bsym{d}, V) := 
\bigl( \{ \Gamma(V, \opn{IG}(\mcal{A}_{k_0})) \}, 
\{ \opn{IG}(g_{k_0, k_1}) \}, \{ a_{k_0, k_1, k_2} \} \bigr) \]
in which $k_0, k_1, k_2 \in K(V)$. So $\opn{IG}(\bsym{d}, V)$ is a 
groupoid descent datum indexed by the set $K(V)$. 
If $K(V) \neq \emptyset$, then by Lemma \ref{lem:bitorsors.3}(1) there is a
nonempty connected groupoid 
$\gerbe{G}(V) := \opn{glu}(\opn{IG}(\bsym{d}, V))$.
If If $K(V) = \emptyset$ we let $\gerbe{G}(V)$ be the empty groupoid.
For $V = U_k$ there is a distinguished object $k \in \opn{Ob}(\gerbe{G}(U_k))$, 
and for $V = U_{k_0, k_1}$ there is a distinguished isomorphism 
$s_{k_0, k_1} : k_0 \to k_1$ in the groupoid $\gerbe{G}(U_{k_0, k_1})$.
Also there is a tautological group isomorphism 
$\opn{cp} : \gerbe{G}(V)(k_0, k_0) \to 
\Gamma(V, \opn{IG}(\mcal{A}_{k_0}))$
for  $k_0 \in K(V)$.

If $V' \subset V$ then $K(V) \subset K(V')$. 
By construction of the groupoids $\opn{glu}(-)$ there is an induced morphism
of groupoids $\gerbe{G}(V) \to \gerbe{G}(V')$, 
which on the automorphism groups of $k_0 \in K(V)$ it is the restriction
homomorphism
$\Gamma(V, \opn{IG}(\mcal{A}_{k_0})) \to \Gamma(V', \opn{IG}(\mcal{A}_{k_0}))$.
In this way we obtain a presheaf of groupoids $\gerbe{G}$ on $X$, which is a
separated prestack (the isomorphism presheaves are sheaves).
The group isomorphisms $\opn{cp}$ above induce an isomorphism of sheaves of
groups 
$\opn{cp} : \gerbe{G}(k_0, k_0) \to 
\opn{IG}(\mcal{A}_{k_0})$ on $U_{k_0}$.

Let $\til{\gerbe{G}}$ be the stackification of $\gerbe{G}$; so
$\til{\gerbe{G}}$ is a gerbe on $X$, and there is a weak equivalence of
prestacks $F : \gerbe{G} \to \til{\gerbe{G}}$. Note that $\til{\gerbe{G}}$
is equipped with a trivialization datum
$\bsym{s} = \bigl( \{ i_{k_0} \} , \{ s_{k_0, k_1} \} \bigr)$
on $\bsym{U}$,
where $i_{k_0} := F(k_0) \in \opn{Ob}(\til{\gerbe{G}}(U_{k_0}))$.

For any $V \subset X$ open there is a morphism of groupoids 
$\gerbe{A}(V) : \gerbe{G}(V) \to \twocat{P}_1(V)$,
namely $\gerbe{A}(V)(k_0) := \mcal{A}_{k_0}$ for 
$k_0 \in K(V) = \opn{Ob}(\gerbe{G}(V))$,
and
$\gerbe{A}(V)(s_{k_0, k_1} \cdot a) := 
g_{k_0, k_1} \circ \opn{ig}(a)$
for $a \in \gerbe{G}(V)(k_0, k_0) = \Gamma(V, \opn{IG}(\mcal{A}_{k_0}))$.
Note that $\opn{ig}(a)$ is an automorphism of $\mcal{A}_{k_0}|_V$,
and $g_{k_0, k_1} : \mcal{A}_{k_0}|_V \to \mcal{A}_{k_1}|_V$
is an isomorphism, all in the stack of groupoids $\twocat{P}_1$.
It is a direct calculation to check that $\gerbe{A}(V)$ respects compositions,
so it is indeed a morphism of groupoids. 
As $V$ varies we get a morphism of prestacks 
$\gerbe{A} : \gerbe{G} \to \twocat{P}_1$. 
This induces a morphism of stacks 
$\til{\gerbe{A}\, \, } \! \! : \til{\gerbe{G}} \to \twocat{P}_1$.

Take a local object $i$ of $\til{\gerbe{G}}$ which is not a local object of 
$\gerbe{G}$. This new local object $i$ is created by gluing a collection of
local objects of $\gerbe{G}$ using local isomorphisms in $\gerbe{G}$. 
The same gluing operation also 
determines an isomorphism of sheaves of groups 
$\til{\opn{cp}} : \til{\gerbe{G}}(i, i) \to 
\opn{IG}( \til{\gerbe{A}\, \, } \!  (i))$. 
We thus get a twisted object
$\opn{glu}(\bsym{d}) = 
(\til{\gerbe{G}}, \til{\gerbe{A}\, \, } \! ,  \til{\opn{cp}})$.
By construction we see that the descent datum
$\opn{dec}(\opn{glu}(\bsym{d}), \bsym{s})$ equals $\bsym{d}$, so property (ii)
holds.

Suppose $\bsym{d}'$ is another normalized descent datum, and 
$\bsym{f} : \bsym{d} \to \bsym{d}'$ is a gauge transformation, described like
in (\ref{eqn:bitorsors.14}). Let's write 
$\opn{glu}(\bsym{d}') = (\til{\gerbe{G}}', {\til{\gerbe{A}\, \, } \! }', 
\til{\opn{cp}}')$.
Lemma \ref{lem:bitorsors.3}(2) implies that for every open set $V$ there's an
induced isomorphism of group\-oids 
$\gerbe{G}(V) \to \gerbe{G}'(V)$. Hence we get an equivalence of prestacks of
groupoids 
$\Phi_{\mrm{gau}} : \gerbe{G} \to \gerbe{G}'$.
Next there is an isomorphism 
$\Phi_{\mrm{rep}} : \gerbe{A} \to \gerbe{A}' \circ \Phi_{\mrm{gau}}$
of functors of prestacks $\gerbe{G} \to \gerbe{G}'$, which for an index
$k_0 \in K$ is
$\Phi_{\mrm{rep}} := f_{k_0} : \mcal{A}_{k_0} \to \mcal{A}'_{k_0}$. 
The stackification extends these to an equivalence of gerbes  
$\Phi_{\mrm{gau}} : \til{\gerbe{G}} \to \til{\gerbe{G}}'$
and an isomorphism of functors of stacks
$\Phi_{\mrm{rep}} : \til{\gerbe{A}\, \, } \! \! \to 
{\til{\gerbe{A}\, \, } \! \! }' \circ \Phi_{\mrm{gau}}$.
We obtain a twisted gauge equivalence 
$(\Phi_{\mrm{gau}}, \Phi_{\mrm{rep}}) : \opn{glu}(\bsym{d}) \to
\opn{glu}(\bsym{d}')$
between these twisted objects; so property (i) holds.
 
Finally take any twisted object $(\gerbe{G}, \gerbe{A}, \opn{cp})$ of 
$\twocat{P}$ and a trivialization datum $\bsym{s}$.
The step-by-step construction of 
$\opn{glu}(\bsym{d}) = 
(\til{\gerbe{G}}, \til{\gerbe{A}\, \, } \! , \til{\opn{cp}})$
shows that it is twisted gauge equivalent to  
$(\gerbe{G}, \gerbe{A}, \opn{cp})$. We see that property (iii) holds.
\end{proof}

\begin{thm}[Decomposition] \label{thm:glu.1}
Let $X$ be a topological space, $\bsym{U}$ an open covering of $X$,
and $\twocat{P}$ a stack of crossed groupoids on $X$. 
Assume that $\bsym{U}$ totally trivializes $\twocat{P}$. Then the
function 
\[ \ol{\opn{dec}} :  \ol{\opn{TwOb}}(\twocat{P}) \to 
\ol{\opn{Desc}}(\opn{C}(\bsym{U}, \twocat{P}))  \]
from Proposition \tup{\ref{prop:bitorsors.5}} is bijective.
\end{thm}

\begin{proof}
By property (i) Lemma \ref{lem:bitorsors.9} there is a
well-defined function 
\[ \ol{\opn{glu}} :  \ol{\opn{Desc}}(\opn{C}(\bsym{U}, \twocat{P})) \to 
\ol{\opn{TwOb}}(\twocat{P}) . \]
Properties (ii-iii) of the lemma say that $\ol{\opn{glu}}$ is the 
inverse of $\ol{\opn{dec}}$.
\end{proof}

We end this section with a simple example.

\begin{exa} \label{exa:19}
Let $\mcal{A}_0 \in \opn{Ob}(\twocat{P}(X))$. Consider the open covering
$\bsym{U} = \{ U_0 \}$, with $U_0 := X$. 
Take $g_{0,0} := \bsym{1}$ and $a_{0,0,0} := 1$.
Then 
$\bigl( \{ \mcal{A}_0 \}, \{ g_{0,0} \}, \{ a_{0,0,0} \} \bigr)$
is a geometric descent datum. {}From Theorem \ref{thm:glu.1} we get a twisted
object $\gerbe{A}$,
and  $\mcal{A}_0$ is a global object of $\gerbe{A}$. We refer to $\gerbe{A}$
as the twisted object generated by $\mcal{A}_0$.
Conversely, any twisted object $\gerbe{A}$ which is not really twisted
arises in this way (up to twisted gauge equivalence). 
\end{exa}

\section{Total Trivialization of Twisted Deformations}
\label{sec:tttf}

In this section we assume Setup \ref{setup:200}. 
Acyclic open sets were introduced in Definition \ref{dfn:23}.

Let $\mcal{A}$ be an $R$-deformation of $\mcal{O}_U$
(associative of Poisson) , for some open set 
$U \subset X$. By \cite[Corollary 3.8]{Ye9}, for every $p \geq 1$ 
the sheaf of $R$-modules $\m^p \mcal{A}$ is $\m$-adically complete.
As explained in Section \ref{sec:defs-alg-vars}, there is a sheaf of
pronilpotent groups $\opn{Exp}(\m^p \mcal{A})$ on $U$, with an isomorphism of
sheaves of sets 
$\exp : \m^p \mcal{A} \iso \opn{Exp}(\m^p \mcal{A})$.
If $V \subset U$ is an $\mcal{O}_X$-acyclic open set, then by 
\cite[Proposition 4.5]{Ye9} the algebra 
$A := \Gamma(V, \mcal{A})$ is an $R$-deformation of 
$C := \Gamma(V, \mcal{O}_X)$.

\begin{lem} \label{lem:mdd.101}
Let $U \subset X$ be an open set, and let
$\mcal{A}$ be an  $R$-deformation of $\mcal{O}_U$. Write
$\mcal{G} := \opn{IG}(\mcal{A}) = \opn{Exp} (\m \mcal{A})$,
and $\mcal{N}_p := \opn{Exp} (\m^{p} \mcal{A})$ for $p \geq 1$.
\begin{enumerate}
\item $\mcal{N}_p$ is a sheaf of normal subgroups of $\mcal{G}$, and
$\mcal{N}_p / \mcal{N}_{p+1}$ is central in $\mcal{G} / \mcal{N}_{p+1}$.
The canonical homomorphism 
$\mcal{G} \to \lim_{\leftarrow p}\, \mcal{G} / \mcal{N}_p$
is an isomorphism.

\item Suppose $V \subset U$ is an $\mcal{O}_X$-acyclic open set. 
Let $A := \Gamma(V, \mcal{A})$, and let
$N_p := \opn{Exp}(\m^{p} A)$. Then the cohomology groups 
$\mrm{H}^i(V, \mcal{N}_p / \mcal{N}_{p+1})$ are
trivial for all $p, i \geq 1$, and the canonical group homomorphisms 
$N_p \to \Gamma(V, \mcal{N}_p)$
and
$N_p / N_{q} \to \Gamma(V, \mcal{N}_p / \mcal{N}_{q})$
are bijective for all $q \geq p \geq 1$.
\end{enumerate}
\end{lem}

\begin{proof}
(1) We know that the Lie bracket satisfies 
$[\m \mcal{A}, \m^{p} \mcal{A}] \subset \m^{p+1} \mcal{A}$.
Hence $[ \mcal{G} , \mcal{N}_{p} ] \subset \mcal{N}_{p+1}$, 
so $\mcal{N}_{p}$ is normal in $\mcal{G}$ and 
$\mcal{N}_p / \mcal{N}_{p+1}$ is central in $\mcal{G} / \mcal{N}_{p+1}$.
Now $\mcal{G} / \mcal{N}_{p} \cong \opn{Exp}(R_{p-1} \ot_R \m \mcal{A})$,
and by definition 
$\opn{Exp} (\m \mcal{A}) = \lim_{\leftarrow p}\,
\opn{Exp}(R_{p} \ot_R \m \mcal{A})$.

\medskip \noindent
(2) Since $\opn{exp} : \m^{p} \mcal{A} \to \mcal{N}_p$
and
$\opn{exp} : \m^{p} \mcal{A} / \m^{q} \mcal{A} \to 
\mcal{N}_p / \mcal{N}_q$
are isomorphisms of sheaves of sets, and  
$\opn{exp} : \m^{p}  \mcal{A} / \m^{p+1} \mcal{A} \to 
\mcal{N}_p / \mcal{N}_{p+1}$
is an isomorphism of sheaves of abelian groups, we are allowed to switch 
from groups to Lie algebras. Namely, it suffices to prove the assertions 
for $\mcal{N}_p := \m^{p} \mcal{A}$ and $N_p := \m^{p} A$.

For the vanishing of cohomology we use the isomorphism
of sheaves of abelian groups
$\mcal{N}_p / \mcal{N}_{p+1} \cong 
(\m^{p} / \m^{p+1}) \ot \mcal{O}_U$
induced by the augmentation $\AA \to \OO_U$ 
(cf.\ \cite[Proposition 5.3]{Ye5}) to deduce
\[ \mrm{H}^i(V, \mcal{N}_p / \mcal{N}_{p+1}) \cong
(\m^{p} / \m^{p+1}) \ot \mrm{H}^i(V, \mcal{O}_U) = 0 \]
for $i \geq 1$.
The claim that the functions
$N_p \to \Gamma(V, \mcal{N}_p)$
and
$N_p / N_{q} \to \Gamma(V, \mcal{N}_p / \mcal{N}_{q})$
are bijective is  \cite[Corollary 3.6]{Ye9}, for $R$ and
$R_{q-1}$ respectively. 
\end{proof}

\begin{thm}[Total Trivialization] \label{thm:6}
Assume Setup \tup{\ref{setup:200}}. We consider the two cases\tup{:}
\begin{enumerate}
\item[$\vartriangleright$] The associative case, in which 
$\twocat{P}(R, X)$ is the stack  of crossed groupoids 
$\twocat{AssDef}(R, \mcal{O}_X)$ on $X$.
\item[$\vartriangleright$] The Poisson case, in which 
$\twocat{P}(R, X)$ is the stack  of crossed groupoids \lb
$\twocat{PoisDef}(R, \mcal{O}_X)$ on $X$. 
\end{enumerate}
Let 
$(\gerbe{G}, \gerbe{A}, \opn{cp})$ be a twisted object in 
$\twocat{P}(R, X)$, and let $U$ be an open set of $X$.
\begin{enumerate}
\item If $\mrm{H}^2(U, \mcal{O}_X) = 0$ then the groupoid
$\gerbe{G}(U)$ is nonempty.
\item If $\mrm{H}^1(U, \mcal{O}_X) = 0$ then the groupoid
$\gerbe{G}(U)$ is connected.
\end{enumerate}
\end{thm}

\begin{proof}
Let $i$ be a local object of the gauge gerbe $\gerbe{G}$, defined on
some open set $U \subset X$.
We write $\mcal{A}_i := \gerbe{A}(i) \in \twocat{P}(R, X)(U)$, which is
an $R$-deformation of $\mcal{O}_{U}$.
There is an isomorphism of sheaves of groups 
$\opn{cp}: \gerbe{G}(i) \iso \mrm{IG}(\mcal{A}_i)$ on $U$.
By definition 
$\mrm{IG}(\mcal{A}_i) = \opn{Exp}(\m \mcal{A}_i)$, and 
hence for any $p \geq 1$ we get a sheaf of normal subgroups
\[ \gerbe{N}_p(i) := \opn{cp}^{-1}
\bigl( \opn{Exp}(\m^{p} \mcal{A}_i) \bigr) \subset \gerbe{G}(i) . \]
By Lemma \ref{lem:mdd.101}(1) the sequence $\{ \gerbe{N}_p(i) \}_{p \geq 1}$
is a central filtration of the sheaf of groups $\gerbe{G}(i)$, in the sense
of \cite[Definition 6.1]{Ye4}
(except that here the filtration starts at $p = 1$,
whereas in \cite{Ye4} it starts at $p = 0$). Also  
$\gerbe{G}(i) \cong 
\lim_{\leftarrow p}\, \gerbe{G}(i) / \gerbe{N}_p(i)$.

Next suppose $j$ is another object of $\gerbe{G}(U)$, and
$g \in \gerbe{G}(U)(i, j)$. Since 
$\gerbe{A}(g) : \mcal{A}_i \to \mcal{A}_j$ 
is an $R$-linear sheaf isomorphism, and since 
$\opn{cp} : \opn{Aut}_{\gerbe{G}} \twoiso \opn{IG} \circ \, \gerbe{A}$
is a natural isomorphism of functors, it follows that 
$\opn{Ad}(g) \bigl( \gerbe{N}_p(i) \bigr) = 
\gerbe{N}_p(j)$. 
This says that for fixed $p$, the collection 
$\{ \gerbe{N}_p(i) \}$
is a normal subprestack of groupoids of the gerbe $\gerbe{G}$, in the
sense of \cite[Definition 3.6]{Ye4}.
Moreover, there is a central extension of gerbes
\begin{equation} \label{eqn:19}
1 \to \gerbe{N}_p / \gerbe{N}_{p+1} \to 
\gerbe{G} / \gerbe{N}_{p+1} \xar{F}
\gerbe{G} / \gerbe{N}_{p} \to 1 , 
\end{equation}
and an isomorphism
\[ \gerbe{N}_p / \gerbe{N}_{p+1} \cong 
(\m^{p} / \m^{p+1}) \ot \mcal{O}_X  \]
 of sheaves of abelian groups on $X$.
See \cite[Definition 3.10]{Ye4}.
As we let $p$ vary, we have a central filtration
$\{ \gerbe{N}_p \}_{p \geq 1}$ of the gerbe $\gerbe{G}$, and $\gerbe{G}$ is
complete with respect to this filtration, as in 
\cite[Definition 6.5]{Ye4}.

Let $U$ and $\mcal{A}_i$ be as above, and let $V \subset U$ be an
$\mcal{O}_X$-acyclic open set. 
According to Lemma \ref{lem:mdd.101}(2), the open set $V$ is acyclic with
respect to the central filtration 
$\{ \gerbe{N}_p(i) \}_{p \geq 1}$
of the sheaf of groups $\gerbe{G}(i)$, in the sense of
\cite[Definition 6.2]{Ye4}.

Now take an open covering
$\bsym{U} = \{ U_k \}_{k \in K}$ of $X$ such that 
$\opn{Ob} (\gerbe{G}(U_k)) \neq \emptyset$
for every $k \in K$. This is possible because $\gerbe{G}$ is locally nonempty.
Since $X$ has enough $\mcal{O}_X$-acyclic open sets, we can find an open
covering $\bsym{U}' = \{ U'_k \}_{k \in K'}$ which refines $\bsym{U}$, and such
that each finite intersection
$U'_{k_0, \ldots, k_m}$ is $\mcal{O}_X$-acyclic. Note that
$\opn{Ob} (\gerbe{G}(U'_{k_0, \ldots, k_m})) \neq \emptyset$.
The discussion in the previous paragraph (for 
$V := U'_{k_0, \ldots, k_m}$) tells us that the covering $\bsym{U}'$ is
acyclic with respect to the central filtration 
$\{ \gerbe{N}_p \}_{p \geq 1}$ of the gerbe $\gerbe{G}$.
We conclude that there are enough acyclic coverings with respect to 
$\{ \gerbe{N}_p \}_{p \geq 1}$, in the sense of \cite[Definition 6.9]{Ye4}.

Finally let $U$ be an open set of $X$. Then
\[ \mrm{H}^q(U, \gerbe{N}_p / \gerbe{N}_{p+1}) \cong
(\m^{p} / \m^{p+1}) \ot \mrm{H}^q(U, \mcal{O}_X) . \]
According to \cite[Theorem 6.10]{Ye4}, if 
$\mrm{H}^2(U, \gerbe{N}_p / \gerbe{N}_{p+1})$
is trivial for all $p \geq 1$, then then the groupoid
$\gerbe{G}(U)$ is nonempty; and if 
$\mrm{H}^1(U, \gerbe{N}_p / \gerbe{N}_{p+1})$
is trivial for all $p \geq 1$, then then the groupoid
$\gerbe{G}(U)$ is connected.
\end{proof}

Here is an immediate consequence of Theorem \tup{\ref{thm:6}}:

\begin{cor} \label{cor:4}
In the situation of Setup \tup{\ref{setup:200}}, and with 
$\twocat{P}(R, X)$ as in Theorem \tup{\ref{thm:6}}, suppose that
$\bsym{U}$ is an $\mcal{O}_X$-acyclic open covering of $X$. Then
$\bsym{U}$ totally trivializes the stack of crossed groupoids 
$\twocat{P}(R, X)$.
\end{cor}

For a crossed groupoid $\cat{G}$ we denote by 
$\ol{\opn{Ob}}(\cat{G})$ the set of $1$-isomorphism classes of objects.

\begin{cor} \label{cor:7}
In the situation of Theorem \tup{\ref{thm:6}}, suppose that
$\mrm{H}^2(X, \mcal{O}_X) = \lb \mrm{H}^1(X, \mcal{O}_X) = 0$.
Then the function
\begin{equation} \label{eqn:40}
\ol{\opn{Ob}} \bigl( \twocat{P}(R, X)(X) \bigr) \to
\ol{\opn{TwOb}} \bigl( \twocat{P}(R, X) \bigr)
\end{equation}
constructed in Example \tup{\ref{exa:19}} is a bijection.
\end{cor}

\begin{proof}
Let $(\gerbe{G}, \gerbe{A}, \opn{cp})$ be a twisted object of 
$\twocat{P}(R, X)$. By Theorem \ref{thm:6}(1) there exists an object
$i \in \opn{Ob} \bigl( \gerbe{G}(X) \bigr)$.
Let 
$\mcal{A}_i := \gerbe{A}(i) \in \opn{Ob} \bigl( \twocat{P}(R, X)(X) \bigr).$
Then $\gerbe{A}$ is twisted gauge
equivalent to the twisted object generated by $\mcal{A}_i$.
We see that the function (\ref{eqn:40}) is surjective. 

Now suppose that 
$\mcal{A}, \mcal{A}' \in \opn{Ob} ( \twocat{P}(R, X)(X) )$
are such that the corresponding twisted objects
$\gerbe{A}, \gerbe{A}' \in \opn{TwOb}(\twocat{P}(R, X))$
are twisted equivalent. Let 
\[ (\Phi_{\mrm{gau}}, \Phi_{\mrm{rep}}) :
(\gerbe{G}, \gerbe{A}, \opn{cp}) \to 
(\gerbe{G}', \gerbe{A}', \opn{cp}') \]
be a twisted gauge equivalence. 
There is an object
$i \in \opn{Ob} ( \gerbe{G}(X) )$
such that $\mcal{A} = \gerbe{A}(i)$, and there is an object
$i' \in \opn{Ob} ( \gerbe{G}'(X) )$
such that $\mcal{A}' = \gerbe{A}'(i')$.
Let 
$j' := \Phi_{\mrm{gau}}(i) \in \opn{Ob} ( \gerbe{G}'(X) )$; 
so there is an isomorphism 
$\Phi_{\mrm{rep}} : \mcal{A} = \gerbe{A}(i) \iso \gerbe{A}'(j')$
in $\twocat{P}(R, X)$. 

On the other hand, since the groupoid
$\gerbe{G}'(X)$ is connected, there is some isomorphism 
$g' : j' \iso i'$ in it. Therefore we get an isomorphism
$\gerbe{A}'(g') : \gerbe{A}'(j') \iso \gerbe{A}'(i') = \mcal{A}'$
in $\twocat{P}(R, X)(X)$. We see that 
$\mcal{A} \cong \mcal{A}'$ in $\twocat{P}(R, X)(X)$. This proves that 
function (\ref{eqn:40}) is injective. 
\end{proof}

We end this section with an example.

\begin{exa} \label{exa:18}
It is easy to construct an example of a commutative twisted associative (or
Poisson) $\K[[\hbar]]$-deformation of $\mcal{O}_X$ that is really
twisted (Definition \ref{dfn:25}). Take an algebraic variety $X$ with nonzero
cohomology class $c \in \mrm{H}^2(X, \mcal{O}_X)$. 
Let $\bsym{U}$ be an affine open covering of $X$, and let
$\{ c_{k_0, k_1, k_2} \}$ be a normalized \v{C}ech $2$-cocycle representing $c$
on this covering. Now consider the geometric descent datum
$\bsym{d} = \bigl( \{ \mcal{A}_{k_0} \}, \{ g_{k_0, k_1} \}, 
\{ a_{k_0, k_1, k_2} \} \bigr)$
with $\mcal{A}_{k_0} := \mcal{O}_{U_{k_0}}[[\hbar]]$, 
$g_{k_0, k_1} := \bsym{1}$ and
$a_{k_0, k_1, k_2} := \opn{exp}(\hbar c_{k_0, k_1, k_2})$.
The corresponding twisted deformation $\gerbe{A}$ that we get by Theorem
\ref{thm:glu.1} (namely $\gerbe{A} := \opn{glu}(\bsym{d})$ in Lemma
\ref{lem:bitorsors.9}) will have obstruction
class $c$ in the first order central extension.
More precisely, in the central extension of gerbes (\ref{eqn:19}),
with $p = 1$, the obstruction class for the unique (up to isomorphism)
object $j$ of $(\gerbe{G} / \gerbe{N}_1)(X)$ is
\[ \opn{cl}^2_{F}(j) = c \hbar 
\in \mrm{H}^2 \bigl( X, 
(\m / \m^{2}) \ot \mcal{O}_X \bigr) , \]
which is nonzero. According to \cite[Theorem 4.17]{Ye4} we have
$\opn{Ob} \bigl( (\gerbe{G} / \gerbe{N}_2)(X) \bigr) = \emptyset$, 
implying that 
$\opn{Ob} \bigl( \gerbe{G}(X) \bigr) = \emptyset$. 
\end{exa}

\section{Lie Descent Data}
\label{sec:Lie-desc}

We begin this section by recalling what are $\mrm{L}_{\infty}$ morphisms
between DG Lie algebras,
following \cite[Sections 7-8]{Ye7}; see also \cite{Ko1, Hi2, CKTB}.
We denote by $\cat{DGLie} (\K)$ the category of differential graded Lie algebras
over $\K$, and by $\cat{DGCog}(\K)$ the category of DG unital cocommutative
coalgebras over $\K$. There are functors 
\[ \cat{DGLie}(\K) \xar{\opn{C}} \cat{DGCog}(\K) \xar{\opn{L}}
\cat{DGLie}(\K) , \]
called the {\em bar} and {\em cobar} functors respectively. 
These functors are adjoint to each other, and they have explicit formulas. 
The functor $\opn{C}$ is faithful.
By definition an $\mrm{L}_{\infty}$ morphism 
$\Psi : \g \to \h$ between DG Lie algebras $\g$ and $\h$
is a coalgebra homomorphism $\opn{C}(\Psi) : \opn{C}(\g) \to \opn{C}(\h)$.
In this way we get a new category $\cat{DGLie}_{\infty}(\K)$, with
functors 
\[ \cat{DGLie} (\K) \xar{\opn{inc}} \cat{DGLie}_{\infty}(\K)
\xar{\opn{C}} \cat{DGCog}(\K) , \]
the first faithful and bijective of objects, and the second fully faithful.

The explicit formula for the functor $\opn{C}$ allows an explicit (yet
cumbersome) description of $\mrm{L}_{\infty}$ morphisms. An $\mrm{L}_{\infty}$
morphism $\Psi : \g \to \h$ can be described as a sequence
$\Psi = \{ \Psi_i \}_{i \geq 1}$ of $\K$-multilinear functions
$\Psi_i : \bwedge^i \g \to \mfrak{h}$,
satisfying certain equations (see \cite{Ko1}, 
\cite[Definition 3.7]{Ye1}, \cite[Section 3.7]{CKTB} or \cite[Section 7]{Ye7}).
The homomorphism $\Psi_1 : \g \to \mfrak{h}$ is a homomorphism of DG
$\K$-modules, and it respects the Lie brackets up to the homotopy
$\Psi_2$; and so on. Thus if $\Psi_i = 0$ for all $i \geq 2$, then 
$\Psi_1 : \g \to \mfrak{h}$ is a DG Lie algebra homomorphism.
An $\mrm{L}_{\infty}$ morphism 
$\Psi = \{ \Psi_i \}_{i \geq 1}$ is called a quasi-isomorphism
if $\Psi_1$ is a quasi-isomorphism. 

Take a DG Lie algebra $\g$, and let 
$\til{\g} := (\opn{L} \circ \opn{C})(\g)$.
According to \cite[Proposition 4.4.3(1)]{Hi2} the adjunction homomorphism
$\zeta_{\g} : \til{\g} \to \g$ in $\cat{DGLie}(\K)$ is a 
quasi-isomorphism. Thus if $\g$ is a quasi quantum type DG Lie algebra, then 
so is $\til{\g}$. (Note that even if $\g$ is quantum type, but nonzero, the DG
Lie algebra  $\til{\g}$ is unbounded below.) Thus we have a functor 
\[ \opn{L} \circ \opn{C} : \cat{DGLie}_{\infty}(\K) \to \cat{DGLie}(\K) \]
sending quasi quantum type DG Lie algebras to quasi quantum type DG Lie
algebras. 

Now suppose $(S, \n)$ is another parameter algebra, and
$\nu : R \to S$ is a homomorphism of parameter algebras. 
There are pronilpotent DG Lie algebras $\m \hot \g$ and 
$\n \hot \h$, an
induced $R$-multilinear $\mrm{L}_{\infty}$ morphism
$\nu \hot \Psi : \m \hot \g \to \n \hot \mfrak{h}$.
We know that there is an induced function
\begin{equation}
\opn{MC}(\nu \hot \Psi) : \opn{MC}(\m \hot \g) \to 
\opn{MC}(\n \hot \mfrak{h})
\end{equation}
between MC sets, with explicit formula
\begin{equation} \label{eqn:300}
\opn{MC}(\nu \hot \Psi) (\om) :=
\sum_{i \geq 1} \smfrac{1}{i!} \Psi_{i}(\nu(\om), \ldots, \nu(\om)) .
\end{equation}
The function $\opn{MC}(\nu \hot \Psi)$ is functorial in $\Psi$ and $\nu$.
If $\Psi$ is an $\mrm{L}_{\infty}$ quasi-isomorphism and $\nu$ is an
isomorphism, then we get an induced bijection
\begin{equation} \label{eqn:22}
\ol{\opn{MC}}(\nu \hot \Psi) : \ol{\opn{MC}}(\m \hot \g) \to 
\ol{\opn{MC}}(\n \hot \mfrak{h})  
\end{equation}
on gauge equivalence classes.
See \cite[Theorem 4.2]{Ye7}.

Now let us suppose that $\g$ and $\h$ are of quasi quantum type
(Definition \ref{dfn:220}), so
that their Deligne crossed groupoids exits. If $\Psi : \g \to \h$ is a DG Lie
homomorphism, then there is an induced morphism of crossed groupoids
\begin{equation} \label{eqn:403}
\opn{Del}(\Psi, \nu) : \opn{Del}(\g, R) \to \opn{Del}(\h, S) ,
\end{equation}
and this is functorial in $\Psi$ and $\nu$. Furthermore, if $\Psi$ is a
quasi-isomorphism and $\nu$ is an isomorphism, then $\opn{Del}(\Psi, \nu)$ is a
weak equivalence of crossed groupoids, by \cite[Theorem 0.4]{Ye7}.
However we do not know if this is true when $\Psi$ is  an $\mrm{L}_{\infty}$
morphism; to be precise, we do not know if there is a functor $\opn{Del}(-,R)$
from $\cat{DGLie}_{\infty}(\K)$ to $\cat{CrGrpd}$. 
The best we can do is the next theorem. 

\begin{thm} \label{thm:400}
Let $\Psi : \g \to \h$ be an $\mrm{L}_{\infty}$ morphism between 
quasi quantum type DG Lie algebras, and let 
$\nu : R \to S$ be a homomorphism between parameter algebras.
Consider the first diagram below. It is a diagram in $\cat{DGLie}(\K)$, 
where $\til{\g} := (\opn{L} \circ \opn{C})(\g)$, 
$\til{\h} := (\opn{L} \circ \opn{C})(\h)$, 
$\til{\Psi} := (\opn{L} \circ \opn{C})(\Psi)$,
and $\zeta_{\g}$, $\zeta_{\h}$ are the adjunction morphisms.
Applying $\opn{Del}(-,-)$ to it we obtain the second diagram, which
is in $\cat{CrGrpd}$.
\[
\UseTips \xymatrix @C=6ex @R=4ex {
\til{\g}
\ar[d]_{\zeta_{\g}}
\ar[r]^(0.5){\til{\Psi}}
&
\til{\h}
\ar[d]^{\zeta_{\h}}
\\
\g
&
\h
} 
\hspace{8ex}
\UseTips \xymatrix @C=10ex @R=5ex {
\opn{Del}(\til{\g}, R)
\ar[d]_{\opn{Del}(\zeta_{\g}, \bsym{1}_R)}
\ar[r]^(0.5){\opn{Del}(\til{\Psi}, \nu)}
&
\opn{Del}(\til{\h}, S)
\ar[d]^{\opn{Del}(\zeta_{\h}, \bsym{1}_S)}
\\
\opn{Del}(\g, R)
&
\opn{Del}(\h, S)
}
\]
Then the morphisms $\opn{Del}(\zeta_{\g}, \bsym{1}_R)$ and 
$\opn{Del}(\zeta_{\h}, \bsym{1}_S)$ in $\cat{CrGrpd}$ are weak equivalences.
If $\Psi$ is an  $\mrm{L}_{\infty}$ quasi-isomorphism, 
and if $\nu$ is an isomorphism, then 
$\opn{Del}(\til{\Psi}, \nu)$ is a weak equivalence.
\end{thm}

\begin{proof}
As already noted, the vertical arrows in the first diagram above
are quasi-isomorphisms (this is by \cite[Proposition 4.4.3(1)]{Hi2}).
According to \cite[Theorem 6.13]{Ye7} the morphisms 
$\opn{Del}(\zeta_{\g}, \bsym{1}_R)$ and 
$\opn{Del}(\zeta_{\h}, \bsym{1}_S)$ are weak equivalences. 

If $\Psi$ is an $\mrm{L}_{\infty}$ quasi-isomorphism and $\nu$ is an
isomorphism, then by \cite[Lemma 7.3]{Ye7} the  DG Lie homomorphism
$\til{\Psi} : \til{\g} \to \til{\h}$ is a quasi-isomorphism.
Again using \cite[Theorem 6.13]{Ye7} we conclude that 
$\opn{Del}(\til{\Psi}, \nu)$ is a weak equivalence.
\end{proof}

 A cosimplicial DG Lie algebra, i.e.\ an object of
$\bsym{\Delta}(\cat{DGLie} (\K))$, is  a collection
$\g = \{ \mfrak{g}^{p} \}_{p \in \mbb{N}}$
of DG Lie algebras
$\mfrak{g}^{p} = \bigoplus_{i \in \mbb{Z}} \mfrak{g}^{p, i}$.
For every $\alpha \in \bsym{\Delta}^q_p$
there is a DG Lie algebra homomorphism 
$\g(\alpha) : \mfrak{g}^{p} \to \mfrak{g}^{q}$,
and these homomorphisms satisfy the simplicial relations.
A cosimplicial DG Lie algebra 
$\g = \{ \mfrak{g}^{p} \}_{p \in \mbb{N}}$ is said to be of quasi quantum type
if every $\g^p$ is a  quasi quantum type DG Lie algebra.

A cosimplicial quasi quantum type DG Lie algebra 
$\g = \{ \mfrak{g}^{p} \}_{p \in \mbb{N}}$ and a parameter algebra $R$
give rise to a cosimplicial crossed groupoid 
$\opn{Del}(\g, R) = \bigl\{ \opn{Del}(\g^p, R)  
\bigr\}_{p \in \mbb{N}}$.

\begin{dfn} \label{dfn:Lie-desc.102}
Let $\g$ be a cosimplicial quasi quantum type DG Lie algebra, and let $R$ be a
parameter algebra, all over the field $\K$. 
A {\em Lie descent datum} in $(\g, R)$ is an element of the set 
$\opn{Desc}(\opn{Del}(\g, R))$;
cf.\ Definition \ref{dfn:101}.

Given 
$\bsym{\de}, \bsym{\de}' \in \opn{Desc}(\opn{Del}(\g, R))$,
a {\em gauge transformation} $\bsym{\delta} \to \bsym{\delta}'$ is 
a gauge transformation of combinatorial descent data, as defined in
Definition \ref{dfn:100}.

Passing to gauge equivalence classes we get the set
$\ol{\opn{Desc}}(\opn{Del}(\g, R))$.
\end{dfn}

Let $\g = \{ \mfrak{g}^{p} \}_{p \in \mbb{N}} $ be a cosimplicial quasi quantum
type DG Lie algebra. In a Lie descent datum 
$\bsym{\de} = (\de^0, \de^1, \de^2) \in \opn{Desc}(\opn{Del}(\g, R))$,
the component $\de^0$ belongs to $\opn{MC}(\m \hot \g^0)$.
Thus there is a function 
$\opn{Desc}(\opn{Del}(\g, R)) \to \opn{MC}(\m \hot \g^0)$.
This function respects gauge equivalence, and we get an induced function 
\begin{equation} \label{eqn:Lie-desc.105}
\ol{\opn{Desc}}(\opn{Del}(\g, R)) \to \ol{\opn{MC}}(\m \hot \g^0) .
\end{equation}

Let $\g = \{ \mfrak{g}^{p} \}_{p \in \mbb{N}}$ and 
$\h = \{ \mfrak{h}^{p} \}_{p \in \mbb{N}}$ 
be cosimplicial  quasi quantum type DG Lie algebras. By {\em cosimplicial 
$\mrm{L}_{\infty}$ morphism} 
$\Psi : \g \to \h$ we mean a collection 
$\Psi = \{ \Psi^{p} \}_{p \in \mbb{N}}$
of $\mrm{L}_{\infty}$ morphisms $\Psi^p : \g^p \to \h^p$, 
such that for every $\al : p \to q$ in $\bsym{\Delta}$ we have 
$\Psi^q \circ \g(\al) = \h(\al) \circ \Psi^p$
in $\cat{DGLie}_{\infty}(\K)$.
This gives us a category $\bsym{\Delta}(\cat{DGLie}_{\infty}(\K))$.
We call such $\Psi$ a {\em cosimplicial $\mrm{L}_{\infty}$ quasi-isomorphism} if
each $\Psi^p$ is an $\mrm{L}_{\infty}$ quasi-isomorphism.

\begin{thm}[Equivalence for Lie Descent]  \label{thm:Lie-desc.101}
Let $\Psi : \g \to \h$ be a cosimplicial $\mrm{L}_{\infty}$ morphism between 
cosimplicial quasi quantum type DG Lie algebras, and let 
$\nu : R \to S$ be a homomorphism between parameter algebras. Then there
is a function 
\[ \ol{\opn{Desc}}(\opn{Del})(\Psi, \nu) : 
\ol{\opn{Desc}}(\opn{Del}(\g, R)) \to \ol{\opn{Desc}}(\opn{Del}(\h, S))  \]
with these properties. 

\begin{enumerate}
\rmitem{i} The function $\ol{\opn{Desc}}(\opn{Del})(\Psi, \nu))$
is functorial in $\Psi$ and $\nu$.

\rmitem{ii}  If $\Psi$ is a morphism in $\cat{DGLie}(\K)$ then 
$\ol{\opn{Desc}}(\opn{Del})(\Psi, \nu) = 
\ol{\opn{Desc}}(\opn{Del}^{}(\Psi, \nu))$,
where  
$\opn{Del}^{}(\Psi, \nu)$ is corresponding the morphism in
$\bsym{\Delta}(\cat{CrGrpd})$.

\rmitem{iii} The diagram 
\[ \UseTips \xymatrix @C=18ex @R=5ex {
\ol{\opn{Desc}}(\opn{Del}(\g, R))
\ar[d]
\ar[r]^(0.5){\ol{\opn{Desc}}(\opn{Del})(\Psi, \nu)}
&
\ol{\opn{Desc}}(\opn{Del}(\h, S))
\ar[d]
\\
\ol{\opn{MC}}(\m \hot \g^0)
\ar[r]^(0.5){\ol{\opn{MC}}(\nu \hot \Psi)}
&
\ol{\opn{MC}}(\n \hot \h^0)
} \]
with vertical arrows \tup{(\ref{eqn:Lie-desc.105})}, commutative. 

\rmitem{iv} If $\Psi$ is a cosimplicial $\mrm{L}_{\infty}$ quasi-isomorphism, 
and if $\nu$ is an isomorphism, then 
$\ol{\opn{Desc}}(\opn{Del})(\Psi, \nu)$ is bijective.  
\end{enumerate}
\end{thm}

Please note the subtle notational distinction between 
$\ol{\opn{Desc}}(\opn{Del})(\Psi, \nu)$ and \lb
$\ol{\opn{Desc}}(\opn{Del}^{}(\Psi, \nu))$.
In general (except in the situation of property (ii)) there is no morphism of
cosimplicial crossed groupoids $\opn{Del}^{}(\Psi, \nu)$, and thus the function 
$\ol{\opn{Desc}}(\opn{Del}^{}(\Psi, \nu))$ is not defined; only the function 
$\ol{\opn{Desc}}(\opn{Del})(\Psi, \nu)$ on equivalence classes of descent data
is defined. 

\begin{proof}
Applying the functor $\ol{\opn{Desc}}$ to the second diagram in Theorem 
\ref{thm:400}  we get a diagram (the solid arrows) 
\begin{equation} \label{eqn:405}
\UseTips \xymatrix @C=18ex @R=7ex {
\ol{\opn{Desc}}(\opn{Del}(\til{\g}, R))
\ar[d]
\ar[r]^(0.5){\ol{\opn{Desc}}(\opn{Del}(\til{\Psi}, \nu))}
&
\ol{\opn{Desc}}(\opn{Del}(\til{\h}, S))
\ar[d]
\\
\ol{\opn{Desc}}(\opn{Del}(\g, R))
\ar@{-->}[r]^(0.5){\ol{\opn{Desc}}(\opn{Del})(\Psi, \nu)}
&
\ol{\opn{Desc}}(\opn{Del}(\h, S))
} 
\end{equation}
in $\cat{Set}$.
Theorem \ref{thm:100} tells us that the vertical arrows here are bijections.
Hence there is a unique function 
$\ol{\opn{Desc}}(\opn{Del})(\Psi, \nu)$ making
this diagram commutative. 
The functoriality of $\ol{\opn{Desc}}(\opn{Del})(\Psi, \nu)$
in $\Psi$ is because $\opn{L} \circ \opn{C}$ is a functor; and the
functoriality in $\nu$ is trivial.

If $\Psi$ is in $\cat{DGLie}(\K)$ then we can add the arrow 
$\opn{Del}^{}(\Psi, \nu)$ into the second diagram in Theorem 
\ref{thm:400}, making it commutative. This proves property (ii). 

Since the solid arrows in the diagram (\ref{eqn:405}) 
respect the function (\ref{eqn:Lie-desc.105}), so does the remaining arrow.
This implies property (iii). 

Finally, if $\Psi$ happens to be an $\mrm{L}_{\infty}$
quasi-isomorphism, and $\nu$ is an isomorphism, then according to 
Theorem \ref{thm:400} the morphism of crossed groupoids
$\opn{Del}(\til{\Psi}, \nu)$ is also a weak equivalence.
Hence the functions
$\ol{\opn{Desc}}(\opn{Del}(\til{\g}, R))$
and 
$\ol{\opn{Desc}}(\opn{Del})(\Psi, \nu)$ are bijective.
\end{proof}

\begin{rem}
If we assume that the parameter algebras in question are all artinian, then in
Theorems \ref{thm:400} and \ref{thm:Lie-desc.101} we can
drop the condition that the DG Lie algebras are quasi quantum type. See
\cite[Theorem 6.13]{Ye7}. 
\end{rem}

\begin{rem} \label{rem:Lie-desc.101}
Here is a brief comparison to other work on these matters. 
A special case of Theorem \ref{thm:Lie-desc.101} is 
\cite[Proposition 3.3.1]{BGNT1}, which only deals with cosimplicial DG Lie
quasi-isomorphisms between quantum type DG Lie algebras, and only with artinian
parameter algebras. The methods used in \cite{BGNT1} do not seem to be adequate
for obtaining the result we need. 

Getzler wrote two papers \cite{Ge1, Ge2} on higher Deligne groupoids. 
Suppose $\g$ is quantum type DG Lie algebra, and $R$ is an artinian parameter
algebra. In the first paper Getzler discusses the Deligne $2$-groupoid
$\opn{Del}(\g, R)$, which is denoted there by $\mcal{C}(\m \ot \g)$.  In the
second paper he introduced a simplicial set $\ga_{\bdot}(\m \ot \g)$.
A main result of \cite{Ge2} is that  an $\mrm{L}_{\infty}$ quasi-isomorphism 
$\g \to \h$ induces a homotopy equivalence 
$\ga_{\bdot}(\m \ot \g) \to \ga_{\bdot}(\m \ot \h)$.
Getzler says that the simplicial set $\ga_{\bdot}(\g \ot \m)$ ``generalizes the
Deligne $2$-groupoid'', but this is not proved, nor even explained.
For some time we had hoped to rely on Getzler's results for a quick proof of 
Theorem \ref{thm:Lie-desc.101} (at least for artinian parameter algebras), but
we could not find a satisfactory way to fill the gap mentioned above.
This is why we eventually had to come up with  Theorem \ref{thm:400}. 

Theorem \ref{thm:Lie-desc.101} (for artinian $R$ and $S$) is similar to 
\cite[Theorem 3.5]{CH}. But \cite[Theorem 3.5]{CH} is not actually proved in
that paper (there appear to be holes in the arguments, cf.\ Remark
\ref{rem:cosim.106}). Another possible difficulty is that in \cite{CH} the
authors rely on the unproved part of Getzler's work that we mentioned above. 

Here is a very recent development. The new paper \cite{BGNT3} appears to fill
the gap in \cite{Ge2}
mentioned above. According to \cite[Remark 1.1]{BGNT3}, their work implies that
the simplicial set $\ga_{\bdot}(\m \ot \g)$ equals the $2$-nerve of the 
Deligne crossed groupoid $\opn{Del}(\g, R)$~!
\end{rem}

\section{From Lie Descent to Geometric Descent}
\label{sec:Lie-Geo}

Suppose $\mcal{G}$ is a sheaf of abelian groups on a topological space $X$. 
Let us recall how to construct the \v{C}ech cosimplicial abelian group
$\mrm{C}(\bsym{U}, \mcal{G})$ associated to $\mcal{G}$ and to an  
open covering $\bsym{U} = \{ U_k \}_{k \in K}$ of $X$. 
For $p \in \mbb{N}$ we take the abelian group
\begin{equation} \label{eqn:Lie-Geo.103}
\mrm{C}^p(\bsym{U}, \mcal{G}) := 
\prod_{\bsym{k} \in \bsym{\Delta}_p(K)}
\Gamma(U_{\bsym{k}}, \mcal{G}) .
\end{equation}
See formula (\ref{eqn:mdd.111}) regarding the set $\bsym{\Delta}_p(K)$.
If $\bsym{k} \in \bsym{\Delta}_q(K)$ and
$\alpha \in \bsym{\Delta}^q_p$, then
$\al(\bsym{k}) \in \bsym{\Delta}_p(K)$, and there is an inclusion of open sets
$U_{\bsym{k}} \subset U_{\alpha(\bsym{k})}$,
so by restriction there is an induced homomorphism
$\mrm{C}^p(\bsym{U}, \mcal{G}) \to 
\mrm{C}^q(\bsym{U}, \mcal{G})$.
If $\mcal{G}$ is a sheaf of DG Lie algebras, then $\mrm{C}(\bsym{U}, \mcal{G})$
is a cosimplicial DG Lie algebra.

Now suppose 
$\bsym{U}' = \{ U_k \}_{k \in K'}$
is another open covering of $X$, and 
$\rho : \bsym{U}' \to \bsym{U}$ is a refinement.
Then there is a homomorphism of cosimplicial DG Lie algebras 
\begin{equation} \label{eqn:Lie-Geo.106}
\rho^* : \mrm{C}(\bsym{U}, \mcal{G}) \to 
\mrm{C}(\bsym{U}', \mcal{G}) ,
\end{equation}
with an obvious formula. 

\begin{setup} \label{setup:6}
$\K$ is a field of characteristic $0$; $(R, \m)$ is a parameter 
algebra over $\K$; and $X$ is a smooth algebraic variety
over $\K$, with structure sheaf $\mcal{O}_X$.
\end{setup}

This is the same as Setup \ref{setup:200}, except that here $X$ is a smooth
algebraic variety.

In Section \ref{sec:defs-DG-Lie} we discussed the sheaves of DG Lie algebras  
$\mcal{T}_{\mrm{poly}, X}$, $\mcal{D}_{\mrm{poly}, X}$ and 
$\mcal{D}_{\mrm{poly}, X}^{\mrm{nor}}$ on the variety $X$. 
Lie descent data $\opn{Desc}(\opn{Del}((\g, R))$ for a cosimplicial quantum
type DG Lie algebra $\g$ was introduced in Section \ref{sec:Lie-desc}, and
its functorial properties were shown.

Recall the stacks of crossed groupoids $\twocat{AssDef}(R, \mcal{O}_X)$ 
and $\twocat{PoisDef}(R, \mcal{O}_X)$ from Section \ref{sec:tw.sh}.
Let us write $\twocat{P}(R, X)$ for either of these stacks. 
The geometric descent data 
$\opn{Desc} \bigl( \opn{C}(\bsym{U}, \twocat{P}(R, X)) \bigr)$ was introduced in
Section \ref{sec:mdd}. The functorial dependence of 
$\ol{\opn{Desc}} \bigl( \opn{C}(\bsym{U}, \twocat{P}(R, X)) \bigr)$ on
the open covering $\bsym{U}$ is shown in
formula (\ref{eqn:23}). Regarding the dependence on $R$, consider a
homomorphism of parameter algebras $\sigma : R \to R'$.
Given a descent datum 
\[ \bsym{d} = \bigl( \{ \mcal{A}_{k_0} \}_{k_0 \in K}, \ldots \bigr) \in
\opn{Desc} \bigl( \opn{C}(\bsym{U}, \twocat{P}(R, X)) \bigr) \]
as in (\ref{eqn:mdd.120}),
let $\mcal{A}'_{k_0} := R' \hatotimes{R} \mcal{A}_{k_0}$, which by 
Proposition \ref{prop:206} is an
$R'$-deformation of $\mcal{O}_{U_{k_0}}$.
There are induced $R'$-linear gauge transformations
$g'_{k_0, k_1}$ and induced gauge elements
$a'_{k_0, k_1, k_2}$, and together these make up a descent datum
$\sigma(\bsym{d}) =  \bigl( \{ \mcal{A}'_{k_0} \}, \ldots \bigr)$.
This construction is a function
\begin{equation} \label{eqn:24}
\sigma : \opn{Desc} \bigl( \opn{C}(\bsym{U}, \twocat{P}(R, X)) \bigr)
\to \opn{Desc} \bigl( \opn{C}(\bsym{U} , \twocat{P}(R', X)) \bigr) , 
\end{equation}
which respects gauge equivalence.

\begin{thm}[Geometrization of Descent Data] \label{thm:11}
Let $(R, \m)$ and $X$ be as in Setup \tup{\ref{setup:6}}, and let
$\bsym{U}$ be a finite affine open covering of $X$.
We consider two cases\tup{:}
\begin{enumerate}
\item[$\vartriangleright$] The associative case, in which 
\[ \twocat{P}(R, X) := \twocat{AssDef}(R, \mcal{O}_X) \quad
\text{and} \quad
\g(\bsym{U}) := \mrm{C}(\bsym{U}, 
\mcal{D}_{\mrm{poly}, X}^{\mrm{nor}})
. \] 
\item[$\vartriangleright$] The Poisson case, in which 
\[ \twocat{P}(R, X) := \twocat{PoisDef}(R, \mcal{O}_X) \quad
\text{and} \quad
\g(\bsym{U}) := \mrm{C}(\bsym{U}, \mcal{T}_{\mrm{poly}, X}) . \] 
\end{enumerate}
So $\opn{C}(\bsym{U} , \twocat{P}(R, X))$
and 
$\opn{Del}(\g(\bsym{U}) , R)$ are cosimplicial crossed groupoids. 

In both cases there is a bijection
\[  \ol{\opn{geo}} : 
\ol{\opn{Desc}} \bigl( \opn{Del}(\g(\bsym{U}) , R) \bigr)
\iso
\ol{\opn{Desc}} \bigl( \opn{C}(\bsym{U} , \twocat{P}(R, X)) \bigr) , \]
which we call {\em geometrization}. The bijection $\opn{geo}$ respect
homomorphisms of parameter algebras $R \to R'$, and
refinements of coverings $\bsym{U}' \to \bsym{U}$.
\end{thm}

\begin{proof}
By Theorem \ref{thm:205} there is a weak equivalence of cosimplicial crossed 
groupoids  
\[  \opn{geo} :  \opn{Del}(\g(\bsym{U}) , R) \to
\opn{C}(\bsym{U} , \twocat{P}(R, X))  . \]
(Actually this is an equivalence of crossed groupoids in each cosimplicial
dimension.) Theorem \ref{thm:100} says that 
$\ol{\opn{geo}} := \ol{\opn{Desc}}(\opn{geo})$ is a bijection.
The functorial dependence on $R$ and $\bsym{U}$ is clear.
\end{proof}

\begin{rem}
More generally, if $\mcal{G}$ is any sheaf of quantum type DG Lie
algebras on a topological space $X$, 
$\bsym{U}$ is an open covering of $X$,
$\g(\bsym{U}) := \mrm{C}(\bsym{U}, \mcal{G})$
and 
$\bsym{\delta} \in \opn{Desc}\bigl( \opn{Del}(\g(\bsym{U}), R) \bigr)$, 
then $\opn{geo}(\bsym{\delta})$ is a geometric descent datum for a gerbe 
$\gerbe{H}$ on $X$. For an index $k \in K$ 
there is an object
$k \in \opn{Ob} (\gerbe{H}(U_k))$, and its sheaf of automorphisms
is 
$\gerbe{H}(k) = \opn{Exp}(\m \hot \mcal{G}^{-1})_{\om_{k}}$.
It might be interesting to study the kind of gerbes that arise in
this way.
\end{rem}

Suppose $X'$ is another smooth algebraic variety, and $g : X' \to X$ is an
\'etale morphism. 
It follows from \cite[Proposition 4.6]{Ye2} that there are
induced homomorphisms of sheaves of DG Lie algebras 
\begin{equation}
g^* : \mcal{T}^{}_{\mrm{poly}, X} \to
g_* (\mcal{T}^{}_{\mrm{poly}, X'})
\end{equation}
and
\begin{equation}
g^* : \mcal{D}_{\mrm{poly}, X}^{\mrm{nor}} \to 
g_* (\mcal{D}_{\mrm{poly}, X'}^{\mrm{nor}})
\end{equation}
on $X$, extending the ring homomorphism
$g^* : \mcal{O}_X \to g_* (\mcal{O}_{X'}$).

Let $\bsym{U} = \{ U_k \}_{k \in K}$ be an open covering of $X$,
and $\bsym{U}' = \{ U'_k \}_{k \in K'}$ an open covering of
$X'$. A {\em morphism of coverings extending $g$}
is a function $\rho : K' \to K$, such that $g(U'_k) \subset U_{\rho(k)}$ for
every $k \in K'$. We also say that
$\rho : \bsym{U}' \to \bsym{U}$
is a morphism of coverings extending $g$.

Suppose $\bsym{V} = \{ V_l \}_{l \in L}$ is another open
covering of $X$. We get a new open covering 
$\bsym{U} \coprod \bsym{V} := \{ W_m \}_{m \in M}$ of $X$, where
$M := K \coprod L$, $W_m := U_m$ if $m \in K$, and 
$W_m := V_m$ if $m \in L$. (Of course $\bsym{U} \coprod \bsym{V}$ has a
lot of redundancy.) If $\tau : \bsym{V}' \to \bsym{V}$
is a morphism of coverings extending $g$, then 
$\rho \amalg \tau : \bsym{U}' \coprod \bsym{V}' \to \bsym{U} \coprod \bsym{V}$
is also a morphism of coverings extending $g$.

Given a morphism of coverings
$\rho : \bsym{U}' \to \bsym{U}$ extending $g$, 
there is an induced homomorphism of cosimplicial DG Lie algebras
$\rho^* : \g(\bsym{U}) \to \g(\bsym{U}')$.
Here we use the notation of Theorem \ref{thm:11}, with obvious
modifications; e.g.\ 
$\g(\bsym{U}') := \mrm{C}(\bsym{U}', \mcal{T}_{\mrm{poly}, X'})$
in the Poisson case.

\begin{thm} \label{thm:7}
Let $R$ and $X$ be as in Setup \tup{\ref{setup:6}}, 
let $g : X' \to X$ be an \'etale morphism of varieties, and let 
$\si : R \to R'$ be a homomorphism of parameter algebras.
We use the notation of Theorem
\tup{\ref{thm:11}}, with obvious modifications pertaining to the
variety $X'$ and the algebra $R'$. Then, both in the Poisson case and
in the associative case, there is a unique function
\[ \opn{ind}_{\si, g} : 
\ol{\opn{TwOb}} \bigl( \twocat{P}(R, X) \bigr) \to 
\ol{\opn{TwOb}} \bigl( \twocat{P}(R', X') \bigr)  \]
with this property\tup{:}
\begin{enumerate}
\rmitem{$\lozenge$} Suppose $\bsym{U}$ \tup{(}resp.\ $\bsym{U}'$\tup{)}
is a finite affine open covering of $X$
\tup{(}resp.\ $X'$\tup{)}, and
$\rho : \bsym{U}' \to \bsym{U}$
is a morphism of coverings extending $g$.
Then the diagram of sets \tup{(\ref{eqn:220})} is commutative.
\end{enumerate}

\begin{equation} \label{eqn:220}
\UseTips \xymatrix @C=5ex @R=5ex {
\ol{\opn{Desc}} \bigl( \opn{Del}(\g(\bsym{U}), R) \bigr)
\ar[r]^(0.46){\ol{\opn{geo}}}
\ar[d]_{\ol{\opn{Desc}} (\opn{Del}(\rho^*, \si ))}
& 
\ol{\opn{Desc}} \bigl( \opn{C}(\bsym{U}, \twocat{P}(R, X)) \bigr)
& 
\ol{\opn{TwOb}} \bigl( \twocat{P}(R, X) \bigr) 
\ar[l]_(0.45){\ol{\opn{dec}}}
\ar[d]^{\opn{ind}_{\si, g}} 
\\
\ol{\opn{Desc}} \bigl( \opn{Del}(\g(\bsym{U}'), R') \bigr)
\ar[r]^(0.46){\ol{\opn{geo}}}
& 
\ol{\opn{Desc}} \bigl( \opn{C}(\bsym{U}', \twocat{P}(R', X')) \bigr)
& 
\ol{\opn{TwOb}} \bigl( \twocat{P}(R', X') \bigr) 
\ar[l]_(0.45){\ol{\opn{dec}}}
}
\end{equation}
Here $\ol{\opn{dec}}$ is the decomposition from Theorem \tup{\ref{thm:glu.1}},
$\ol{\opn{geo}}$ is the geometrization from Theorem \tup{\ref{thm:11}}, and 
$\ol{\opn{Desc}} ( \opn{Del}( \rho^*, \si ))$ is the function from Theorem
\tup{\ref{thm:Lie-desc.101}}.
\end{thm}

\begin{proof}
Choose coverings  $\bsym{U}$ and $\bsym{U}'$, and a morphism of
coverings
$\rho : \bsym{U}' \to \bsym{U}$, 
as in property ($\lozenge$). This is possible of course. 
The horizontal arrows in the diagram are bijections.
Define 
\[ \opn{ind}_{\si, g} :
\ol{\opn{TwOb}} \bigl( \twocat{P}(R, X) \bigr) \to 
\ol{\opn{TwOb}} \bigl( \twocat{P}(R', X') \bigr) \]
to be the unique function that makes the diagram commute.

It remains to show that the function
$\opn{ind}_{\si, g}$ is independent of
the choice of open coverings $\bsym{U}$ and $\bsym{U}'$, and the of morphism 
$\rho$ between them. So let $\bsym{V}$ be another finite affine
open covering of $X$, let $\bsym{V}'$ be another such covering of $X'$, and let
$\tau : \bsym{V}' \to \bsym{V}$
be a morphism of coverings extending $g$.
Consider the open covering 
$\bsym{W} := \bsym{U} \coprod \bsym{V}$ of $X$, the open covering 
$\bsym{W}' := \bsym{U}' \coprod \bsym{V}'$ of $X'$,
and the morphism of coverings 
$\chi := \rho \amalg \tau : \bsym{W}' \to \bsym{W}$
extending $g$. We obtain a commutative diagram of morphisms of coverings
\begin{equation} \label{eqn:Lie-Geo.15}
\UseTips \xymatrix @C=6ex @R=4ex {
\bsym{U}
\ar[r]
&
\bsym{W}
&
\bsym{V}
\ar[l]
\\
\bsym{U}'
\ar[u]^{\rho}
\ar[r]
& 
\bsym{W}' 
\ar[u]_{\chi}
&
\bsym{V}'
\ar[l]
\ar[u]_{\tau}
}
\end{equation}
extending $g$ vertically, and extending $\bsym{1}_X$ and  $\bsym{1}_{X'}$
horizontally.  
The functoriality in Theorem \ref{thm:11} and in Proposition
\ref{prop:bitorsors.6}, applied to the left square in diagram
(\ref{eqn:Lie-Geo.15}), and to the homomorphism $\si : R \to R'$,
yields a big commutative diagram in the shape of a rectangular box. 
The front face of this big diagram is the diagram in property ($\lozenge$),
and the rear face is the same diagram but with $\bsym{W}$ instead of
$\bsym{U}$. All horizontal arrows in the big diagram are bijections. 
It follows that the function $\opn{ind}_{\si, g}$ is the same when defined
using $\chi : \bsym{W}' \to \bsym{W}$.
But then, working with the right square in diagram (\ref{eqn:Lie-Geo.15}),
we see that the function $\opn{ind}_{\si, g}$ is also the same when defined
using $\tau : \bsym{V}' \to \bsym{V}$.
\end{proof}

\section{Cosimplicial Formality}
\label{sec:cosim-form}

In this section we prove Theorem \ref{thm:221}. We work in Setup
\ref{setup:6}.

Recall the sheaves of DG Lie algebras 
$\mcal{T}_{\mrm{poly}, X}$ and 
$\mcal{D}_{\mrm{poly}, X}$ from Section \ref{sec:defs-DG-Lie}. 
The Lie brackets are not $\mcal{O}_X$-linear (they are differential operators),
but forgetting them, both $\mcal{T}_{\mrm{poly}, X}$ and 
$\mcal{D}_{\mrm{poly}, X}$ are bounded below complexes of quasi-coherent 
$\mcal{O}_X$-modules.

The antisymmetrization homomorphism (also called the HKR map)
\begin{equation} \label{eqn:TDQ.120}
\Upsilon_1 : \mcal{T}^p_{\mrm{poly}, X} \to \mcal{D}^p_{\mrm{poly}, X}
\end{equation}
is defined as follows. For $p \geq 0$ and for local sections 
$\xi_i \in \mcal{T}^0_{\mrm{poly}, X}$ and $c_i \in \mcal{O}_X$, we let
\[ \Upsilon_1(\xi_1 \wedge \cdots \wedge \xi_{p+1})(c_1, \ldots, c_{p+1}) :=
{\smfrac{1}{(p+1)!}} \sum_{\sigma} \opn{sign}(\sigma)
\xi_{\sigma(1)}(c_1) \cdots \xi_{\sigma(p+1)}(c_{p+1}) . \]
The summation is over permutations $\sigma$ of the set
$\{ 1, \ldots, p+1 \}$. 
For $p = -1$ we let $\Upsilon_1$ be the identity map of $\mcal{O}_X$. 
According to \cite[Corollary 4.12]{Ye2} the homomorphism
$\Upsilon_1 : \mcal{T}_{\mrm{poly}, X} \to \mcal{D}_{\mrm{poly}, X}$
is a quasi-isomorphism of sheaves of $\mcal{O}_X$-modules. Note that
$\Upsilon_1$ does not respect the Lie brackets; but
$\mrm{H}(\Upsilon_1) : \mcal{T}_{\mrm{poly}, X} \to 
\opn{H} (\mcal{D}_{\mrm{poly}, X})$
is an isomorphism of sheaves of graded Lie algebras. 

Let $\mcal{G}$ and $\mcal{H}$ be sheaves of DG Lie algebras on
$X$. We denote by $\boprod^i \mcal{G}$ the $i$-fold cartesian product of the
sheaf $\mcal{G}$. An $\mrm{L}_{\infty}$ morphism
$\Psi : \mcal{G} \to \mcal{H}$ is a sequence
$\Psi = \{ \Psi_i \}_{i \geq 1}$ of $\K$-multilinear
sheaf morphisms $\Psi_i : \boprod^i \mcal{G} \to \mcal{H}$, 
such that for any open set $U$ the sequence
$\{ \Gamma(U, \Psi_i) \}_{i \geq 1}$ is an $\mrm{L}_{\infty}$ morphism
$\Gamma(U, \mcal{G}) \to \Gamma(U, \mcal{H})$. 
We say that $\Psi$ is an $\mrm{L}_{\infty}$ quasi-isomorphism if 
$\Psi_1 : \mcal{G} \to \mcal{H}$ is a quasi-isomorphism of sheaves of
$\K$-modules. See Section \ref{sec:Lie-desc} for a reminder on
$\mrm{L}_{\infty}$ morphisms.

Let $\bsym{U}$ finite affine open covering of $X$.
Recall the mixed resolution
$\opn{mix} : \mcal{M} \to \opn{Mix}_{\bsym{U}}(\mcal{M})$
of any quasi-coherent $\mcal{O}_X$-module $\mcal{M}$ 
from \cite[Section 6]{Ye1}. This is a functorial quasi-isomorphism, which is a
``mix'' of the commutative \v{C}ech resolution relative to the covering 
$\bsym{U}$, and the jet resolution (with its Grothendieck connection). 
The complex $\opn{Mix}_{\bsym{U}}(\mcal{M})$ is acyclic for the functor 
$\Gamma(V, -)$ for any affine open set $V \subset X$. 
Hence the homomorphism of cosimplicial $\K$-modules 
\[ \opn{C}(\bsym{U}, \opn{mix}) :
\opn{C}(\bsym{U}, \mcal{M}) \to 
\opn{C}(\bsym{U}, \opn{Mix}_{\bsym{U}}(\mcal{M})) \]
(cf.\ (\ref{eqn:Lie-Geo.103}))
is a cosimplicial quasi-isomorphism (i.e.\ it is a quasi-isomorphism in each
simplicial dimension).
Taking the sheaves $\mcal{T}_{\mrm{poly}, X}$ and $\mcal{D}_{\mrm{poly}, X}$,
we obtain quasi-isomorphisms of sheaves of DG Lie algebras
$\mcal{T}_{\mrm{poly}, X} \to  \opn{Mix}_{\bsym{U}}(\mcal{T}_{\mrm{poly}, X})$
and 
$\mcal{D}_{\mrm{poly}, X} \to  \opn{Mix}_{\bsym{U}}(\mcal{D}_{\mrm{poly}, X})$.
See \cite[Proposition 6.3]{Ye1}.

Let $g : X' \to X$ be a morphism of varieties, $\mcal{M}$ (resp.\ 
$\mcal{M}'$) a quasi-coherent $\mcal{O}_X$-module (resp.\ 
$\mcal{O}_{X'}$-module), $\phi : \mcal{M} \to g_* (\mcal{M}')$ an 
$\mcal{O}_X$-linear map, $\bsym{U}$ (resp.\ $\bsym{U}'$) a finite affine
open covering of $X$ (resp.\ $X'$), and 
$\rho : \bsym{U}' \to \bsym{U}$ a morphism of coverings extending $g$. 
The construction of mixed resolutions (see \cite[Section 6]{Ye1} or
\cite[Section 4]{Ye3}) shows that there is an induced sheaf homomorphism 
\[ \opn{Mix}_{\rho}(\phi) : \opn{Mix}_{\bsym{U}}(\mcal{M}) \to 
g_* (\opn{Mix}_{\bsym{U}'}(\mcal{M}')) \]
on $X$. 
And of course globally there is a homomorphism of cosimplicial $\K$-modules 
\[ \opn{C}(\rho, \phi) : \opn{C}(\bsym{U}, \mcal{M}) \to
\opn{C}(\bsym{U}', \mcal{M}') . \]

Now assume $g : X' \to X$ is an \'etale morphism. Let $\mcal{G}$ denote either 
$\mcal{T}_{\mrm{poly}}$ or $\mcal{D}_{\mrm{poly}}$, so that 
$\mcal{G}_X$ is either $\mcal{T}_{\mrm{poly}, X}$ or
$\mcal{D}_{\mrm{poly}, X}$.
As we already saw in Section \ref{sec:Lie-Geo}, there is a homomorphism of 
sheaves of DG Lie algebras 
$g^* : \mcal{G}_X \to g_* (\mcal{G}_{X'})$ on $X$.  
Then there is an induced homomorphism of sheaves of DG Lie algebras
\[ \opn{Mix}_{\rho}(g^*) : \opn{Mix}_{\bsym{U}}(\mcal{G}_{X}) \to 
g_* (\opn{Mix}_{\bsym{U}'}(\mcal{G}_{X'})) \]
on $X$. Globally there is a commutative diagram of homomorphism of cosimplicial
DG Lie algebras
\begin{equation} \label{eqn:TDQ.101}
\UseTips \xymatrix @C=15ex @R=5ex {
\opn{C}(\bsym{U}, \mcal{G}_{X})
\ar[r]^{\opn{C}(\rho, g^*)}
\ar[d]_{\opn{C}(\bsym{U}, \opn{mix})}
&
\opn{C}(\bsym{U}', \mcal{G}_{X'})
\ar[d]^{\opn{C}(\bsym{U}', \opn{mix})}
\\
\opn{C} \bigl( \bsym{U}, \opn{Mix}_{\bsym{U}}(\mcal{G}_{X}) \bigr)
\ar[r]^{\opn{C}(\rho, \opn{Mix}_{\rho}(g^*))}
&
\opn{C} \bigl( \bsym{U}', \opn{Mix}_{\bsym{U}'}(\mcal{G}_{X'}) \bigr)
}
\end{equation}
where the vertical arrows are cosimplicial quasi-isomorphisms.
If $g : X' \to X$ is an isomorphism (e.g.\ $X' = X$ and 
$\bsym{U}' \to \bsym{U}$ is a refinement), then the horizontal arrows are also 
cosimplicial quasi-isomorphisms.

Let $n$ be the dimension of $X$.
An {\em \'etale coordinate system} on an open set $U \subset X$ is 
an \'etale morphism $s : U \to \mbf{A}^n_{\K}$.
Since $X$ is smooth, it is possible to find a
finite affine open covering 
$\bsym{U} = \{U_k \}_{k \in K}$, with an \'etale coordinate system
$s_k : U_k \to \mbf{A}^n_{\K}$ for each $k$. Let us write
$\bsym{s} := \{s_k \}_{k \in K}$. We refer to
$(\bsym{U}, \bsym{s})$ succinctly as a {\em covering with \'etale coordinates}.

Suppose $g : X \to X'$ is an \'etale morphism of varieties, and we are given a
covering with \'etale coordinates
$(\bsym{U}', \bsym{s}')$ of $X'$, where the indexing set is $K'$. A {\em
morphism of coverings with coordinates
extending $g$} is a function 
$\rho : K' \to K$,
such that $g(U'_{k'}) \subset U_{\rho(k')}$ and
$s_{\rho(k')} \circ g = s'_{k'}$ for every $k' \in K'$.
We refer to this morphism by 
$\rho : (\bsym{U}', \bsym{s}') \to (\bsym{U}, \bsym{s})$. 

The next result is a slight improvement of \cite[Theorem 0.2]{Ye1}. 
A similar result is \cite[Theorem 1.1]{VdB}.

\begin{thm}[Sheaf Formality] \label{thm:13}
Let $X$ be a smooth algebraic variety over $\K$, and assume $\R
\subset \K$. Let $(\bsym{U}, \bsym{s})$ be a finite affine open covering with
\'etale coordinates of $X$. Then\tup{:}
\begin{enumerate}
\item There is an $\mrm{L}_{\infty}$ quasi-isomorphism
\[ \Psi_{\bsym{s}} = \{ \Psi_{\bsym{s}; i} \}_{i \geq 1} :
\opn{Mix}_{\bsym{U}}(\mcal{T}_{\mrm{poly}, X}) \to 
\opn{Mix}_{\bsym{U}}(\mcal{D}_{\mrm{poly}, X})  \]
between sheaves of DG Lie algebras on $X$. 
\item The diagram of isomorphisms of sheaves of graded Lie algebras 
\[ \UseTips \xymatrix @C=9ex @R=5ex {
\mcal{T}_{\mrm{poly}, X}
\ar[r]^{\mrm{H}(\Upsilon_1)}
\ar[d]_{\mrm{H}(\opn{mix})}
&
\mrm{H} (\mcal{D}_{\mrm{poly}, X})
\ar[d]^{\mrm{H}(\opn{mix})}
\\
\mrm{H} (\opn{Mix}_{\bsym{U}}(\mcal{T}_{\mrm{poly}, X}))
\ar[r]^{\mrm{H}(\Psi_{\bsym{s}; 1})}
&
\mrm{H} (\opn{Mix}_{\bsym{U}}(\mcal{D}_{\mrm{poly}, X})) \ 
} \]
on $X$ is commutative.
\item Suppose $g : X \to X'$ is an \'etale morphism of varieties, 
$(\bsym{U}', \bsym{s}')$ is a finite affine open covering with \'etale
coordinates of $X'$, and 
$\rho : (\bsym{U}', \bsym{s}') \to (\bsym{U}, \bsym{s})$
is a morphism of coverings with coordinates extending $g$.
Then the diagram of $\mrm{L}_{\infty}$ morphisms on $X$
\[ \UseTips \xymatrix @C=9ex @R=5ex {
\opn{Mix}_{\bsym{U}}(\mcal{T}_{\mrm{poly}, X})
\ar[r]^{\Psi_{\bsym{s}}}
\ar[d]_{\opn{Mix}_{\rho}(g^*)}
&
\opn{Mix}_{\bsym{U}}(\mcal{D}_{\mrm{poly}, X})
\ar[d]^{\opn{Mix}_{\rho}(g^*)}
\\
g_* (\opn{Mix}_{\bsym{U}'}(\mcal{T}_{\mrm{poly}, X'}))
\ar[r]^{g_*(\Psi_{\bsym{s}'})}
&
g_* (\opn{Mix}_{\bsym{U}'}(\mcal{D}_{\mrm{poly}, X'}))
} \]
is commutative. 
\end{enumerate} 
\end{thm}

\begin{proof}
Part (1) is the content of \cite[Theorem 0.2]{Ye1}, which is repeated in
greater detail (and with a correct proof) as \cite[Erratum, Theorem 1.2]{Ye1}.
Part (3) is a direct consequence of the construction of the 
 $\mrm{L}_{\infty}$ quasi-isomorphism $\Psi_{\bsym{s}}$
in the proof of \cite[Erratum, Theorem 1.2]{Ye1}.

The idea for the proof of part (2) was communicated to us by M. Van den Bergh.
Let $\mcal{M}$ be a bounded below complex of quasi-coherent
$\mcal{O}_X$-modules. For $p \geq 0$ let
\[ G^p (\opn{Mix}_{\bsym{U}}(\mcal{M})) :=
\boplus_{i \geq p} \opn{Mix}^i_{\bsym{U}}(\mcal{M}) . \]
This gives a decreasing filtration $G = \{ G^p \}_{p \geq 0}$
of $\opn{Mix}_{\bsym{U}}(\mcal{M})$
by subcomplexes. Note that 
$\opn{gr}_G^p (\opn{Mix}_{\bsym{U}}(\mcal{M})) = 
\opn{Mix}^p_{\bsym{U}}(\mcal{M})[-p]$
as complexes. 
The filtration $G$ gives rise to a convergent spectral sequence
$E_r^{p, q}(\mcal{M}) \Rightarrow 
\mrm{H}^{p+q} (\opn{Mix}_{\bsym{U}}(\mcal{M}))$,
and its first page is
\[ E_1^{p, q}(\mcal{M}) =
\mrm{H}^{p+q} \bigl( \opn{gr}_G^p (\opn{Mix}_{\bsym{U}}(\mcal{M})) \bigr)
\cong \opn{Mix}^p_{\bsym{U}}(\mrm{H}^q (\mcal{M})) . \]
The differential 
$E_1^{p, q}(\mcal{M}) \to E_1^{p+1, q}(\mcal{M})$
is the differential of the mixed resolution. 
Note that all these objects are quasi-coherent sheaves on $X$. 
Because of the quasi-isomorphism 
$\opn{mix} :   \mrm{H}^q (\mcal{M}) \to
\opn{Mix}_{\bsym{U}}(\mrm{H}^q (\mcal{M}))$
we have 
\begin{equation} \label{eqn:33}
E_2^{p, q}(\mcal{M}) \cong
\begin{cases}
\mrm{H}^q (\mcal{M}) & \text{if } p = 0 \\
0 & \text{if } p \neq 0 . 
\end{cases} 
\end{equation}
We see that the spectral sequence collapses, and 
$E_{\infty}^{p, q}(\mcal{M}) = E_2^{p, q}(\mcal{M})$.
In particular the induced filtration on the limit 
$\mrm{H}^{p+q} (\opn{Mix}_{\bsym{U}}(\mcal{M}))$
of the spectral sequence has only one nonzero jump (at level $G^0$).

According to \cite[Erratum, Theorem 1.2]{Ye1} the homomorphism of complexes
\[ \Psi_{\bsym{s}; 1} :
\opn{Mix}_{\bsym{U}}(\mcal{T}_{\mrm{poly}, X}) \to 
\opn{Mix}_{\bsym{U}}(\mcal{D}_{\mrm{poly}, X}) \]
respects the filtrations $G$, and for every $p$ there is equality
of homomorphisms of complexes
\begin{equation} \label{eqn:34}
\opn{gr}^p_G(\Psi_{\bsym{s}; 1})[p] = \opn{Mix}^p_{\bsym{U}}(\Upsilon_1) :
\opn{Mix}^p_{\bsym{U}}(\mcal{T}_{\mrm{poly}, X}) \to 
\opn{Mix}^p_{\bsym{U}}(\mcal{D}_{\mrm{poly}, X}) .
\end{equation}
Moreover, by \cite[Theorem 4.17]{Ye3}, 
$\opn{Mix}^p_{\bsym{U}}(\Upsilon_1)$ is a quasi-isomorphism.

Since $\Psi_{\bsym{s}; 1}$ respects the filtrations, there is an induced 
map of spectral sequences
$E_r^{p, q}(\Psi) : E_r^{p, q}(\mcal{T}) \to E_r^{p, q}(\mcal{D})$,
where we abbreviate 
$\mcal{T} := \mcal{T}_{\mrm{poly}, X}$ etc. 
{}From (\ref{eqn:34}) we see that there is an isomorphism in the first pages of
the spectral sequences
\[ E_1^{p, q}(\Psi) :
\opn{Mix}^p_{\bsym{U}}(\mcal{T}^q_{\mrm{poly}, X})  \cong
E_1^{p, q}(\mcal{T}) \to E_1^{p, q}(\mcal{D}) \cong
\opn{Mix}^p_{\bsym{U}}(\mrm{H}^{q}  (\mcal{D}_{\mrm{poly}, X})) . \]
Since the differentials are the same, it follows that
$E_2^{p, q}(\Psi) : E_2^{p, q}(\mcal{T}) \to E_2^{p, q}(\mcal{D})$
is an isomorphism.

Finally let's examine the diagram of isomorphisms
\[ \UseTips \xymatrix @C=9ex @R=5ex {
\mcal{T}^q_{\mrm{poly}, X}
\ar[r]^{\mrm{H}^q(\Upsilon_1)}
\ar[d]_{\alpha}
&
\mrm{H}^q (\mcal{D}_{\mrm{poly}, X})
\ar[d]^{\alpha}
\\
E_2^{0, q}(\mcal{T}) 
\ar[r]^{E_2^{0, q}(\Psi)}
\ar[d]_{\be}
&
E_2^{0, q}(\mcal{D}) 
\ar[d]^{\be}
\\
\mrm{H}^q (\opn{Mix}_{\bsym{U}}(\mcal{T}_{\mrm{poly}, X}))
\ar[r]^{\mrm{H}^q(\Psi_{\bsym{s}; 1})}
&
\mrm{H}^q (\opn{Mix}_{\bsym{U}}(\mcal{D}_{\mrm{poly}, X})) \ .
} \]
The arrows $\al$ come from (\ref{eqn:33}); and the top square commutes
because of (\ref{eqn:34}) with $p = 0$.
The arrows $\be$ come from the collapse of the spectral sequence, and for
this reason the bottom square is commutative. So the whole diagram is
commutative.
\end{proof}

\begin{thm}[Cosimplicial Formality] \label{thm:221}
In the situation of Theorem \tup{\ref{thm:13}} there is a diagram
\[ \UseTips \xymatrix @C=12ex @R=5ex {
\mrm{C}(\bsym{U}, \mcal{T}_{\mrm{poly}, X})
\ar[d]_{\opn{C}(\bsym{U}, \opn{mix})}
&
\mrm{C}(\bsym{U}, \mcal{D}_{\mrm{poly}, X}^{\mrm{nor}})
\ar[d]^{\opn{C}(\bsym{U}, \opn{mix})}
\\
\mrm{C} \bigl( \bsym{U} , \opn{Mix}_{\bsym{U}}(\mcal{T}_{\mrm{poly}, X})
\bigr)
\ar[r]^{\mrm{C}(\bsym{U} , \Psi_{\bsym{s}})} 
&
\mrm{C} \bigl( \bsym{U} , \opn{Mix}_{\bsym{U}}(\mcal{D}_{\mrm{poly}, X})
\bigr) \ ,
} \]
where the objects are cosimplicial quantum type DG Lie algebras, the
vertical arrows are cosimplicial DG Lie quasi-isomorphisms, and the arrow 
$\Psi_{\bsym{s}}$ is a cosimplicial $\mrm{L}_{\infty}$ quasi-isomorphism. 

This diagram is functorial with respect to 
\'etale morphisms $g : X' \to X$ and morphisms of coverings 
$\rho : (\bsym{U}', \bsym{s}') \to (\bsym{U}, \bsym{s})$
extending $g$, as in diagram \tup{(\ref{eqn:TDQ.101})}. 
\end{thm}

\begin{proof}
The diagram is obtained by applying the \v{C}ech cosimplicial
construction $\mrm{C}(\bsym{U} , -)$ to the  $\mrm{L}_{\infty}$
quasi-isomorphism $\Psi_{\bsym{s}}$
in  Theorem \ref{thm:13}(1), to the mixed resolutions 
$\mcal{T}_{\mrm{poly}, X} \to 
\opn{Mix}_{\bsym{U}}(\mcal{T}_{\mrm{poly}, X})$
and 
$\mcal{D}_{\mrm{poly}, X} \to 
\opn{Mix}_{\bsym{U}}(\mcal{D}_{\mrm{poly}, X})$, 
and to the inclusion 
$\mcal{D}_{\mrm{poly}, X}^{\mrm{nor}} \to 
\mcal{D}_{\mrm{poly}, X}$.

Functoriality is guaranteed by Theorem \ref{thm:13}(3), and by the 
\'etale functoriality of $\mcal{T}_{\mrm{poly}, X}$,
$\mcal{D}_{\mrm{poly}, X}$ and $\mcal{D}_{\mrm{poly}, X}^{\mrm{nor}}$
from \cite[Proposition 4.6]{Ye2}.
\end{proof}

\section{Twisted Quantization}
\label{sec:TDQ}

Let $X$ be a smooth algebraic variety over $\K$, and $R$ a parameter algebra.
We have the stacks of crossed groupoids 
$\twocat{PoisDef}(R, \mcal{O}_X)$ and 
$\twocat{AssDef}(R, \mcal{O}_X)$ on $X$.
Recall that a twisted object of $\twocat{PoisDef}(R, \mcal{O}_X)$
is a twisted Poisson $R$-deformation of $\OX$, and 
a twisted object of $\twocat{AssDef}(R, \mcal{O}_X)$
is a twisted associative $R$-deformation of $\OX$.
We denote by $\ol{\opn{TwOb}}(-)$ the sets of twisted gauge equivalence classes
of twisted objects. First order brackets were defined in Definition
\ref{dfn:22}.

Here is the main result of our paper (the expanded form of Theorem
\ref{thm:14}).

\begin{thm} \label{thm:8}
Let $\K$ be a field containing the real numbers, let $R$ be a 
parameter $\K$-algebra, and let $X$ be a smooth algebraic variety over
$\K$. Then there is a bijection of sets
\[ \opn{tw{.}quant} : \
\ol{\opn{TwOb}} \bigl( \twocat{PoisDef}(R, \mcal{O}_X) \bigr) 
\ \iso \
\ol{\opn{TwOb}} \bigl( \twocat{AssDef}(R, \mcal{O}_X) \bigr) \] 
called  {\em twisted quantization}, having these
properties\tup{:}
\begin{enumerate}
\rmitem{i} If $g : X' \to X$ is an \'etale morphism of varieties, and
if $\sigma : R \to R'$ is a homomorphism of parameter algebras, then
the diagram
\[ \UseTips \xymatrix @C=9ex @R=5ex {
\ol{\opn{TwOb}} \bigl( \twocat{PoisDef}(R, \mcal{O}_X) \bigr)
\ar[r]^{\opn{tw{.}quant}}
\ar[d]_{\opn{ind}_{\sigma, g}}
& 
\ol{\opn{TwOb}} \bigl( \twocat{AssDef}(R, \mcal{O}_X) \bigr)
\ar[d]_{\opn{ind}_{\sigma, g}} 
\\
\ol{\opn{TwOb}} \bigl( \twocat{PoisDef}(R', \mcal{O}_{X'}) \bigr)
\ar[r]^{\opn{tw{.}quant}}
& 
\ol{\opn{TwOb}} \bigl( \twocat{AssDef}(R', \mcal{O}_{X'}) \bigr)
} \]
\tup{(}cf.\ Theorem \tup{\ref{thm:7})} is commutative.
\rmitem{ii} The bijection $\opn{tw{.}quant}$ preserves first order brackets.
Namely if 
$\gerbe{A}$ is a twisted Poisson $R$-deformation of
$\mcal{O}_X$, and $\gerbe{B} := \opn{tw{.}quant}(\gerbe{A})$, 
then 
\[ \{ -,- \}_{\gerbe{A}} = \{ -,- \}_{\gerbe{B}} . \]
\end{enumerate}
\end{thm}

\begin{proof}
We shall use the diagram in Figure \ref{fig:diagram}
(that appears in Subsection \ref{subsec:TDQ} of the Introduction).  
Let $(\bsym{U}, \bsym{s})$ be a finite affine open covering with \'etale
coordinates of $X$.
By Theorem \ref{thm:glu.1} and Corollary \ref{cor:4} we get the bijections
$\ol{\opn{dec}}$ in the diagram.
{}From Theorem \ref{thm:11} we get the bijections $\ol{\opn{geo}}$ in that
diagram. By applying Theorem \ref{thm:Lie-desc.101} to the diagram in Theorem 
\ref{thm:221} we obtain the bijections 
\[ \ol{\opn{mix}} := 
\ol{\opn{Desc}} (\opn{Del}) ( \mrm{C}(\bsym{U}, \opn{mix}), \bsym{1}_R) \]
and 
\[ \ol{\Psi}_{\bsym{s}} := 
\ol{\opn{Desc}} (\opn{Del}) (\mrm{C}(\bsym{U}, \Psi_{\bsym{s}}), \bsym{1}_R)
\]
in the diagram in Figure \ref{fig:diagram}.
We define $\opn{tw{.}quant}$ to be the unique bijection making this diagram
commutative. 

The whole diagram in Figure \ref{fig:diagram} is
functorial w.r.t.\ $R$ and $(\bsym{U}, \bsym{s})$; this is a consequence of
Theorems \ref{thm:7} and \ref{thm:221}.

Consider another finite affine open covering with \'etale coordinates
$(\bsym{V}, \bsym{t})$ of $X$, and a morphism of coverings with coordinates
$(\bsym{V}, \bsym{t}) \to (\bsym{U}, \bsym{s})$. 
The functoriality of the diagram in Figure \ref{fig:diagram} implies that 
the bijection $\opn{tw{.}quant}$ is the same when defined with respect 
$(\bsym{V}, \bsym{t})$.

Now suppose $(\bsym{V}, \bsym{t})$ is any finite affine open covering with
\'etale coordinates of $X$. As in the proof of Theorem \ref{thm:7}, 
we get a third covering with \'etale coordinates
$(\bsym{W}, \bsym{r})$, where 
$\bsym{W} := \bsym{U} \coprod \bsym{V}$ and  
$\bsym{r} := \bsym{s} \coprod \bsym{t}$. There are morphisms of coverings 
$(\bsym{U}, \bsym{s}) \to (\bsym{W}, \bsym{r})$
and $(\bsym{V}, \bsym{t}) \to (\bsym{W}, \bsym{r})$.
By the previous paragraph, the bijection $\opn{tw{.}quant}$ is the same when
defined with respect to any of these three coverings with \'etale coordinates.
We conclude that $\opn{tw{.}quant}$ is independent of the choices made. 

Let  $g : X' \to X$ be an \'etale morphism of varieties, and let
$\sigma : R \to R'$ be a homomorphism of parameter algebras. Choose any
morphism of coverings with coordinate 
$\rho : (\bsym{U}', \bsym{s}') \to (\bsym{U}, \bsym{s})$
extending $g$. The functoriality of the diagram in Figure \ref{fig:diagram}
implies that there is a commutative diagram as stated in property (i). 

It remains to prove that property (ii) holds. Let $\gerbe{A}$ and 
$\gerbe{B}$ be such twisted $R$-deformations. 
Consider the corresponding Lie descent data 
\[ (\om, g, a) \in 
\opn{Desc} \bigl( \opn{Del} (\mrm{C}(\bsym{U}, \mcal{T}_{\mrm{poly},X}), R)
\bigr) \]
and 
\[ (\chi, h, b) \in 
\opn{Desc} \bigl(  \opn{Del} (\mrm{C}(\bsym{U}, 
\mcal{D}^{\mrm{nor}}_{\mrm{poly},X}) , R) \bigr) . \]
They are related by the equation
\[ \begin{aligned}
& \ol{(\chi, h, b)} = \ol{\opn{Desc}} (\opn{Del})
(\mrm{C}(\bsym{U}, \Psi_{\bsym{s}}), \bsym{1}_R) ( \ol{(\om, g, a)} ) 
\\
& \qquad \qquad \qquad \in 
\ol{\opn{Desc}} \bigl(  \opn{Del}(\mrm{C} (\bsym{U}, 
\opn{Mix}_{\bsym{U}} (\mcal{D}_{\mrm{poly},X})), R) \bigr) ,
\end{aligned} \]
where $\ol{(\chi, h, b)}$ and $\ol{(\om, g, a)}$ are the equivalence classes.
We have to show that the first order brackets
$\{ -,- \}_{\gerbe{A}}$ and $\{ -,- \}_{\gerbe{B}}$ on 
$\mcal{O}_X$ are equal.

The first order brackets 
$\{ -,- \}_{\gerbe{A}}$ and $\{ -,- \}_{\gerbe{B}}$
depend only on $\om$ and $\chi$. Indeed, by Definitions \ref{dfn:22} and
\ref{dfn:20} we have
$\{ c_1, c_2 \}_{\gerbe{A}} = \{ c_1, c_2 \}_{\om}$
and 
$\{ c_1, c_2 \}_{\gerbe{B}} = \{ c_1, c_2 \}_{\chi}$
for local sections $c_1, c_2 \in \OO_X$,
where 
$\{ c_1, c_2 \}_{\om} := \psi^+(\Upsilon_1(\om)(c_1, c_2))$
and 
$\{ c_1, c_2 \}_{\chi} := \psi^+(\chi(c_1, c_2))$. 
Here $\psi^+$ is the homomorphism from equation (\ref{eqn:400}), and 
$\Upsilon_1$ is the antisymmetrization.
Property (iii) of Theorem \ref{thm:Lie-desc.101} says that
$\ol{\chi} = \ol{\opn{MC}}(\bsym{1}_R \ot \Psi_{\bsym{s}})( \ol{\om} )$
in 
\[ \ol{\opn{MC}} \bigl(\m \hot 
\mrm{C}^0(\bsym{U}, \opn{Mix}_{\bsym{U}} (\mcal{D}_{\mrm{poly},X}))  \bigr) . \]

The first order brackets are functorial in $R$, and hence
we can assume that $\m^2 = 0$, so $\m / \m^2 = \m$. 
Now the contribution of the higher order terms $\Psi_{\bsym{s}; i}$, $i \geq 2$,
to $\opn{MC}(\bsym{1}_R \ot \Psi_{\bsym{s}})$ is zero, because 
$\m^2 = 0$. Therefore the elements
$\chi$ and 
$\chi' := (\bsym{1}_R \ot \Psi_{\bsym{s}; 1})(\om)$
of 
$\m \ot \mrm{C}^0(\bsym{U}, \opn{Mix}_{\bsym{U}} (\mcal{D}_{\mrm{poly},X}))$
are gauge equivalent.
We conclude that $\{ -,- \}_{\chi} = \{ -,- \}_{\chi'}$.

A little calculation shows that the first order bracket $\{ -,- \}_{\chi'}$
depends only on the cohomology class $[\chi']$ of $\chi'$ in
$\m \ot \mrm{C}^0 \bigl( \bsym{U}, 
\mrm{H}^1 ( \opn{Mix}_{\bsym{U}} (\mcal{D}_{\mrm{poly},X})) \bigr)$;
this is a manifestation of the gauge invariance of first order brackets.
But according to Theorem \ref{thm:13}(2) we have 
$[\chi'] = [\Upsilon_1(\om)]$.
We see that $\{ -,- \}_{\chi'} = \{ -,- \}_{\om}$.
\end{proof}

\begin{cor} \label{cor:TDQ.105}
In the situation of Theorem \tup{\ref{thm:8}}, 
assume the cohomology groups 
$\mrm{H}^1(X, \mcal{O}_X)$ and $\mrm{H}^2(X, \mcal{O}_X)$ vanish. 
Then there is a bijection
\[ \begin{aligned}
& \opn{quant} : 
\frac{ \{ \tup{Poisson $R$-deformations of $\mcal{O}_X$} \} }
{\tup{ gauge equivalence}} \\
& \hspace{9ex} \iso \
\frac{ \{ \tup{associative $R$-deformations of $\mcal{O}_X$} \} }
{\tup{ gauge equivalence}} . 
\end{aligned} \] 
It preserves first order brackets, and commutes with
homomorphisms $R \to R'$ and  \'etale morphisms $X' \to X$
\tup{(}such that $X'$ also has these cohomology vanishing properties\tup{)}.
\end{cor}

\begin{proof}
Combine Theorem \ref{thm:8} and Corollary \ref{cor:7}.
\end{proof}

A {\em formal Poisson bracket} on the augmented commutative $R$-algebra 
$R \hot \OX$ is an $R$-bilinear Poisson bracket $\{ -,- \}$ that vanishes modulo
$\m$; so it makes $R \hot \OX$ into a Poisson $R$-deformation of $\OX$. 
A {\em star product} on the augmented $R$-module $R \hot \OX$ is an 
associative unital $R$-bilinear multiplication, with unit 
$1 \ot 1$, that reduces modulo $\m$ to the commutative multiplication of $\OX$;
so it makes $R \hot \OX$ into an associative $R$-deformation of $\OX$. 

By definition 
$\mcal{T}^0_{\mrm{poly},X} = \mcal{T}_X$, the sheaf of derivations, and 
$\mcal{D}^{0}_{\mrm{poly},X} = \mcal{D}_X$, the sheaf of differential operators.

\begin{cor} \label{cor:TDQ.102}
In the situation of Theorem \tup{\ref{thm:8}}, 
assume the cohomology groups 
$\mrm{H}^1(X, \mcal{O}_X)$, $\mrm{H}^2(X, \mcal{O}_X)$, 
$\mrm{H}^1(X, \mcal{T}_X)$ and $\mrm{H}^1(X, \mcal{D}_X)$
all vanish. Then there is a canonical bijection
\[ \begin{aligned}
& \opn{quant} : 
\frac{ \{ \tup{formal Poisson brackets on $R \hot \mcal{O}_X$} \} }
{\tup{ gauge equivalence}} \\
& \hspace{9ex} \iso \
\frac{ \{ \tup{star products on $R \hot \mcal{O}_X$} \} }
{\tup{ gauge equivalence}} . 
\end{aligned} \] 
It preserves first order brackets, and commutes with
homomorphisms $R \to R'$ and  \'etale morphisms $X' \to X$ 
\tup{(}such that $X'$ also has these cohomology vanishing properties\tup{)}.
\end{cor}

\begin{proof}
Theorem \ref{thm:205} implies the following: if 
$\mrm{H}^1(X, \mcal{D}_{\mrm{poly}, X}^{\mrm{nor}, 0}) = 0$, then every
associative $R$-deformation $\mcal{A}$ of $\mcal{O}_X$ is sheaf theoretically
trivial, namely there is an isomorphism $\mcal{A} \cong R \hot \mcal{O}_X$ of
sheaves of $R$-modules, that commutes with the augmentations to $\mcal{O}_X$
and preserves the units. 
The proof is the same as that of \cite[Theorem 1.13]{Ye1}, using Remark
\ref{rem:defs-sh.110}. Since there is an isomorphism of sheaves of $\K$-modules 
$\mcal{D}_X \cong \mcal{O}_X \oplus \mcal{D}_{\mrm{poly}, X}^{\mrm{nor}, 0}$,
our assumptions imply that 
$\mrm{H}^1(X, \mcal{D}_{\mrm{poly}, X}^{\mrm{nor}, 0}) = 0$.
Hence associative deformations are the same, up to gauge equivalence, as
star products. 

Similarly, the vanishing of $\mrm{H}^1(X, \mcal{T}_X)$ implies, via 
Theorem \ref{thm:205}, that Poisson deformations are the same as formal
Poisson brackets. 

The quantization statement is now the same as in Corollary \ref{cor:TDQ.105}.
\end{proof}

\begin{exa}
The conditions of Corollary \ref{cor:TDQ.102} hold for 
$X := \mbf{P}^n_{\K}$, $n \geq 1$. Indeed, since $X$ is a $\mcal{D}$-affine
variety, we know that $\mrm{H}^i(X, \mcal{D}_X) = 0$ for all $i > 0$. 
The vanishing of $\mrm{H}^i(X, \mcal{O}_X)$ and 
$\mrm{H}^i(X, \mcal{T}_X)$ for $i > 0$ follows from 
\cite[Theorem III.5.1 and Example II.8.20.1]{Ha}. 
We thank Van den Bergh for mentioning this to us.
\end{exa}

\begin{rem}
When $X$ is an affine variety, the quantization map of Corollary
\ref{cor:TDQ.102} coincides with that of \cite[Theorem 0.1]{Ye1}. This can be
seen by looking at the construction of these quantization maps. (For 
\cite[Theorem 0.1]{Ye1} please consult the corrected proof in the Erratum.)
\end{rem}

\begin{rem} \label{rem:TDQ.101}
The reason we require that $\K$ contains $\R$ is because we rely on the
original universal $\mrm{L}_{\infty}$ formality quasi-isomorphism 
\begin{equation} \label{eqn:extras.12}
\Upsilon = \{ \Upsilon_i \}_{i \geq 1} : 
\mcal{T}_{\mrm{poly}}(\R[\bsym{t}]) \to 
\mcal{D}_{\mrm{poly}}(\R[\bsym{t}])
\end{equation}
of Kontsevich for the polynomial ring $\R[\bsym{t}] = \R[t_1, \ldots, t_n]$; see
\cite{Ko1}. However $\Upsilon$ can be replaced with any other $\mrm{L}_{\infty}$
quasi-isomorphism
$\Upsilon'$ in the proof of \cite[Theorem 0.2]{Ye1}, as long as $\Upsilon'$ has
properties (P)-(P5) of \cite{Ko1}. 

Now D.E. Tamarkin found a construction of such an $\mrm{L}_{\infty}$
quasi-isomorphism $\Upsilon'$ over $\mbb{Q}$ (see \cite{Ta}, and
\cite[Appendix]{CV} for a complete proof.)
Therefore the condition $\R \subset \K$ can in fact be removed.
\end{rem}

We end the paper with a few questions. 
A twisted $\K[[\hbar]]$-deformation $\gerbe{A}$ of $\mcal{O}_X$
is called {\em symplectic} if the first order bracket
$\{ -,- \}_{\gerbe{A}}$ is a symplectic Poisson bracket on
$\mcal{O}_X$ (cf.\ Proposition \ref{prop:20}(3)). 
This is the most noncommutative deformation possible: the center of a local
object $\mcal{A}$ of $\gerbe{A}$ is the subsheaf 
$\K[[\hbar]] \subset \mcal{A}$. 

\begin{que}
It is easy to construct an example of a {\em commutative} twisted
associative $\K[[\hbar]]$-deformation of $\mcal{O}_X$ that is
really twisted -- see Definition \ref{dfn:25} and Example \ref{exa:18}. But
does there exist a variety $X$, with  a  {\em symplectic} twisted associative
$\K[[\hbar]]$-deformation $\gerbe{A}$ of $\mcal{O}_X$ which is {\em really
twisted}? Perhaps the results of \cite{BK} can be useful here.
\end{que}

A more concrete (but perhaps much more challenging) question is:

\begin{que} \label{que:1}
Let $X$ be a Calabi-Yau surface over $\K$ (e.g.\ an abelian surface or
a K3 surface), and let
$\alpha$ be a symplectic Poisson bracket on $\mcal{O}_X$ (namely any
nonzero section of $\Gamma(X, \bwedge^{2}_{\mcal{O}_X} \mcal{T}_X)$).
Consider the Poisson $\K[[\hbar]]$-deformation 
$\mcal{A} := \mcal{O}_X[[\hbar]]$, with formal Poisson bracket
$\hbar \alpha$, and let $\gerbe{A}$ be the corresponding twisted
deformation (see Example \ref{exa:19}). 
Let $\gerbe{B} := \opn{tw{.}quant}(\gerbe{A})$. Is $\gerbe{B}$ really
twisted? 
If so, what is the significance of this phenomenon?
Note that the obstruction classes for $\gerbe{B}$
can be calculated explicitly; but these calculations look quite
complicated. 
Kontsevich [private communication] appears to think that the twisted
deformation $\gerbe{B}$ is really twisted, and he has an indirect
argument for that.
\end{que}

\begin{que}
The construction of the $\mrm{L}_{\infty}$ quasi-isomorphism 
$\Psi_{\bsym{s}}$ in Theorem \ref{thm:13} relied on the explicit universal
formality quasi-isomorphism $\Upsilon$ of Kontsevich; see (\ref{eqn:extras.12})
above. But suppose another universal formality quasi-isomorphism $\Upsilon'$
is used. Then the twisted quantization map 
$\opn{tw{.}quant}$ may change. Indeed, it is claimed by Kontsevich
\cite{Ko3} that the Grothendieck-Teichm\"uller group acts on the
quantizations by changing the formality quasi-iso\-morphism (or in
other words, the Drinfeld associator), and sometimes this action is
nontrivial. The question is: does this action change the geometric
nature of the resulting twisted associative deformation -- namely can
it change from being really twisted to being not really twisted?
\end{que}


\end{document}